\documentclass[a4paper,12pt,reqno,draft]{amsart}
\usepackage{amssymb,amsmath,array,amscd,amsthm,hhline,bibcheck}

\usepackage[mathscr]{euscript}
\usepackage{stmaryrd}
\usepackage{ulem}

\usepackage{tikz-cd}

\usepackage{mathrsfs}

\makeatletter
\newcommand{\slunlhd}{%
  \mathrel{\mathpalette\sl@unlhd\relax}%
}

\newcommand{\sl@unlhd}[2]{%
  \sbox\z@{$#1\lhd$}%
  \sbox\tw@{$#1\leqslant$}%
  \dimen@=\ht\tw@
  \advance\dimen@-\ht\z@
  \ifx#1\displaystyle
    \advance\dimen@ .2pt
  \else
    \ifx#1\textstyle
      \advance\dimen@ .2pt
    \fi
  \fi
  \ooalign{\raisebox{\dimen@}{$\m@th#1\lhd$}\cr$\m@th#1\leqslant$\cr}%
}
\makeatother

\renewcommand\trianglelefteq{\slunlhd}

\renewcommand{\labelenumi}{{\rm \theenumi}}
\renewcommand{\theenumi}{{\rm(\arabic{enumi})}}
\renewcommand\epsilon{\varepsilon}
\renewcommand\emptyset{\varnothing}

\def\lm{\lambda}
\def\Lm{\Lambda}
\renewcommand\kappa{\varkappa}
\def\bbeta{{\boldsymbol\beta}}
\def\aa{\mathbf{a}}

\def\simd{\triangle}

\def\<{\langle}
\def\>{\rangle}

\def\ito{\stackrel\sim\to}

\def\C{\mathbb C}

\def\Z{\mathbb Z}

\def\Q{\mathbb Q}

\newtheoremstyle{mytheoremstyle} 
    {\topsep}                    
    {\topsep}                    
    {\rmfamily}                   
    {1em}                           
    {\bf}                   
    {.}                          
    {.5em}                       
    {}  

\newtheorem{theorem}{Theorem}
\newtheorem{proposition}[theorem]{Proposition}
\newtheorem{lemma}[theorem]{Lemma}
\newtheorem{corollary}[theorem]{Corollary}
\newtheorem{remark}[theorem]{Remark}
\newtheorem{problem}{Problem}

\def\le{\leqslant}
\def\ge{\geqslant}

\def\pt{{\rm pt}}

\def\df{{\rm d}}

\def\ll{\ell}
\def\b{{\mathbf b}}
\def\c{{\mathbf c}}

\def\Tr{\mathop{\rm Tr}\nolimits}

\def\Map{\mathop{\rm Map}}

\def\supp{\mathop{\rm supp}}

\def\id{\mathop{\rm id}}
\def\im{\mathop{\rm im}}

\def\SL{\mathop{\rm SL}}
\def\PSL{\mathop{\rm PSL}}

\def\rep{\mathop{\rm rep}}

\def\suchthat{\mathbin{\rm |}}
\def\and{\,\mathbin{\&}\,}

\def\Hom{\mathop{\rm Hom}\nolimits}
\def\Tor{\mathop{\rm Tor}\nolimits}

\def\pr{{\rm pr}}

\def\dlim{\mathop{\rm lim}\limits_{\longrightarrow}}

\def\Ext{\mathop{\rm Ext}\nolimits}

\def\OSigma{\bar{\Sigma}}
\def\OGamma{\bar{\Gamma}}

\def\OX{\bar{\X}}

\def\F{\mathscr F}
\def\X{\mathcal X}

\def\E{\mathscr E}

\def\P{\mathcal P}

\def\C{\mathbb C}

\renewcommand\phi{\varphi}

\def\={\equiv}

\renewcommand{\(}{\left(}
\renewcommand{\)}{\right)}

\def\sectsign{\mathhexbox278}

\def\m{\mathfrak m}

\def\csh#1#2{\underline{#1}{}_{{}_{\scriptstyle #2}}}

\makeatletter

\def\linefill@#1{\m@th\setboxz@h{$#1-$}\ht\z@\z@
    $#1\copy\z@\mkern-6mu\cleaders
    \hbox{$#1\mkern-2mu\box\z@\mkern-2mu$}\hfill
    \mkern-2mu\vphantom{\rightarrow}$}

  \def\rightarrowfill@#1{\m@th\setboxz@h{$#1-$}\ht\z@\z@
    $#1\copy\z@\mkern-6mu\cleaders
    \hbox{$#1\mkern-2mu\box\z@\mkern-2mu$}\hfill
    \mkern-6mu\mathord\rightarrow$}

\def\edge#1#2#3{
\global\bigaw@1.8pc
\setboxz@h{$\m@th\scriptstyle\;{#2}\;\;$}%
  \ifdim\wdz@>\bigaw@\global\bigaw@\wdz@\fi
#1\stackrel{#2}{\mathop{\hbox to \bigaw@{\linefill@\displaystyle}}}#3
}
\def\edgeright#1#2#3{
\global\bigaw@1.8pc
\setboxz@h{$\m@th\scriptstyle\;{#2}\;\;$}%
  \ifdim\wdz@>\bigaw@\global\bigaw@\wdz@\fi
#1\stackrel{#2}{\mathop{\hbox to \bigaw@{\rightarrowfill@\displaystyle}}}#3
}
\makeatother

\title{Bases of T-equivariant cohomology of Bott-Samelson varieties}
\author{Vladimir Shchigolev}
\email{shchigolev\_vladimir@yahoo.com}

\begin{document}

\maketitle

\begin{abstract} We construct combinatorial bases of the $T$-equivariant cohomology $H^\bullet_T(\Sigma,k)$
of the Bott-Samelson variety $\Sigma$ under some mild restrictions on the field of coefficients $k$.
This bases allow us to prove the surjectivity of the restrictions
$H^\bullet_T(\Sigma,k)\to H^\bullet_T(\pi^{-1}(x),k)$ and
$H^\bullet_T(\Sigma,k)\to H^\bullet_T(\Sigma\setminus\pi^{-1}(x),k)$, where $\pi:\Sigma\to G/B$
is the canonical resolution. In fact, we also construct bases of the targets of these restrictions by
picking up certain subsets of certain bases of $H^\bullet_T(\Sigma,k)$ and restricting them
to $\pi^{-1}(x)$ or $\Sigma\setminus\pi^{-1}(x)$ respectively.

As an application, we calculate the cohomology of the costalk-to-stalk embedding for the direct image
$\pi_*\csh k\Sigma$. This algorithm avoids division by 2, which allows us to reestablish
2-torsion for parity sheaves in Braden's example~\cite{BW}.
\end{abstract}



\section{Introduction}

Let $\Sigma$ be a Bott-Samelson variety for a connected semisimple complex group $G$. In this paper, we study the $T$-equivariant
cohomology $H^\bullet_T(\Sigma,k)$, where $T$ is a maximal torus in $G$ and $k$ is a PID. The direction of our research
is mainly determined by H\"arterich's preprint~\cite{Haerterich}. However, this preprint uses difficult
Arabia's results~\cite{A89} and~\cite{A98}, which, as explicitly stated, are valid for the ring of coefficients $\Q$.
Therefore, we prefer not to use geometrical bases (coming from Bia\l ynicki-Birula cells) and construct combinatorial
bases instead. If the sequence of simple reflections determining $\Sigma$ has length $r$, then there are in total $2^{2^r-1}$ bases
$B_\rho$ of $H^\bullet_T(\Sigma,k)$ under some mild restriction on the characteristic of $k$
(Theorem~\ref{theorem:4} and Lemma~\ref{lemma:15}).

Let $\pi:\Sigma\to G/B$ be the canonical resolution and $x\in G/B$ be an arbitrary $T$-fixed point.
Using the previously constructed bases of $H^\bullet_T(\Sigma,k)$, we can construct
a basis of $H^\bullet_T(\pi^{-1}(x),k)$ as follows (Theorem~\ref{theorem:2}, Remark~\ref{remark:1} and Lemma~\ref{lemma:16}):
\begin{enumerate}
\itemsep=3pt
\item\label{step:1} choose an index $\rho$;
\item\label{step:2} choose a subset $M\subset B_\rho$;
\item\label{step:3} consider the restrictions $\{f|_{\pi^{-1}(x)}\suchthat f\in M\}$.
\end{enumerate}
This fact automatically implies that the restriction $H^\bullet_T(\Sigma,k)\to H^\bullet_T(\pi^{-1}(x),k)$ is surjective.

One may naturally ask what happens if we consider the compliment $\Sigma\setminus\pi^{-1}(x)$ instead
of $\pi^{-1}(x)$? It turns out that there exists a basis $H^\bullet_T(\Sigma\setminus\pi^{-1}(x),k)$
that can be constructed from a basis $B_\rho$ of $H^\bullet_T(\Sigma,k)$ by steps similar to steps~\ref{step:1}--\ref{step:3}
above (Theorem~\ref{theorem:3}, Remark~\ref{remark:2} and Lemma~\ref{lemma:17}).

A plausible motivation to consider the $T$-equivariant cohomology of $\Sigma\setminus\pi^{-1}(x)$ is to calculate the decomposition
of the direct image $\pi_*\csh k\Sigma$ into a direct sum of parity sheaves introduced in~\cite{Williamson_parity_sheaves}.
It was noted by the authors of this paper that the natural map $i_\lm^!{\mathcal F}\to i_\lm^*{\mathcal F}$ plays a decisive role
in determining such a decomposition at least when $k$ is a field (see~\cite[Proposition 2.26]{Williamson_parity_sheaves}).
Here $i_\lm$ is the embedding of a (closed) stratum. In this paper, we address the following question:

\begin{problem}\label{problem:1} Let $\pi:\Sigma\to G/B$ be a Bott-Samelson resolution and $x\in G/B$ be a $T$-fixed point.
Denote by $i_x:\{x\}\hookrightarrow G/B$ the natural embedding. How to calculate the map
$H_T^\bullet(\Sigma,i_x^!\pi_*\csh k\Sigma)\to H_T^\bullet(\Sigma,i_x^*\pi_*\csh k\Sigma)$?
\end{problem}

It is answered in this paper by Corollary~\ref{corollary:5}. Note that unlike~\cite{Williamson_parity_sheaves}
this problem does not involve any stratifications. However, we can apply its solution to parity sheaves
by considering the stratification $G/B=\bigsqcup_{x\in W}BxB/B$ and dividing by the $T$-equivariant Euler classes
of the natural embeddings $\{x\}\hookrightarrow BxB/B$. The corresponding constriction is given in
Section~\ref{Application_to_parity_sheaves}.

Our algorithm is similar to the one described in~\cite{Shchigolev}. It uses the same construction of the transition matrix
as in~\cite[Theorem 4.10.3]{Shchigolev}, which was previously used by Fiebig for his upper bound for Lusztig's
conjecture~\cite{Fiebig} and which originally comes from the same H\"arterich's preprint~\cite{Haerterich}.
The advantage of our approach here compared with~\cite{Shchigolev} is that we do not divide by $2$
when we compute products of the basis elements of $H_T^\bullet(\pi^{-1}(x),k)$ (cf.~formula~(\ref{eq:40}) of this paper
and~\cite[Lemma 4.8.3]{Shchigolev}). This allows us to reestablish the 2-torsion for hexagonal permutations
in Braden's example~\cite[Appendix A]{BW}.

The paper is organized as follows. In Section~\ref{Bott-Samelson_varieties}, we introduce the main notions:
the Bott-Samelson variety, combinatorial galleries, load bearing walls, orders $\vartriangleleft$ and $<$,
tree analogs of combinatorial galleries, etc. In Section~\ref{Localization_theorems}, we explain how
to adjust Brion's proofs~\cite{Brion} of localization theorems to our situation of coefficients different from $\C$
and of noncompact spaces.
We use some ideas from~\cite{FW}, where the authors also prove localization theorems additionally assuming
finiteness of $T$-curves (which is not the case for Bott-Samelson varieties).
It is important to notice that we do need some form of GKM-restriction
(\ref{corollary:1:condition:3} in Corollary~\ref{corollary:1}) to prove the intersection formula for the image.
To ensure this condition, we first work with coefficients $\Z'=\Z$ or $\Z'=\Z[1/2]$ if the root system contains
a component of type $C_n$ and then change coefficients to a PID in Section~\ref{Change_of_coefficients}.

In Section~\ref{Bases_of_XXx}, we construct bases of $H^\bullet_T(\Sigma,\Z')$ and $H^\bullet_T(\Sigma_x,\Z')$. We use
here the criterions proved by H\"arterich~\cite[Theorems~6.2 and~6.3]{Haerterich}. The proofs of these results use only
the smooth case of $\SL_2(\C)$ or $\PSL_2(\C)$, which can be handled by~\cite{Atiyah_Bott}.
We develop here the main combinatorial tool of this paper --- operators of copy $\Delta$ and concentration $\nabla_t$,
which construct elements of $H^\bullet_T(\Sigma,\Z')$ from elements of $H^\bullet_T(\Sigma',\Z')$,
where $\Sigma'$ is the Bott-Samelson variety for the truncated sequence.
Finally, we construct bases of $H^\bullet_T(\Sigma_x,\Z')$ in the way similar to~\cite{Shchigolev}.
Our new product of bases elements~(\ref{eq:40}) does not include division, which is a definite advantage.

In Section~\ref{Basis_of_the_image_bar}, we construct a basis of $H^\bullet_T(\Sigma\setminus\pi^{-1}(x),\Z')$.
To achieve this goal, we need to prove the localization theorem for $\Sigma\setminus\pi^{-1}(x)$ and
to prove a criterion for $H^\bullet_T(\Sigma\setminus\pi^{-1}(x),\Z')$ (Proposition~\ref{proposition:3})
similar to H\"arterich's criterions~\cite[Theorems~6.2 and~6.3]{Haerterich}.

Section~\ref{costalk-to-stalk embedding} is devoted to applications of the obtained results.
We begin with the change\linebreak of coefficients in Section~\ref{Change_of_coefficients}, which
allows us to obtain bases of $H^\bullet_T(\Sigma,k)$, $H^\bullet_T(\Sigma_x,k)$ and
$H^\bullet_T(\Sigma\setminus\pi^{-1}(x),k)$ for any PID $k$ of characteristic not $2$ if the root system contains
a component of type $C_n$. Then we solve Problem~\ref{problem:1} and in Section~\ref{Application_to_parity_sheaves}
show how this information can be used to decompose the direct image $\pi_*\csh k\Sigma[r]$ to a direct sum of
indecomposable parity sheaves. As an example, we show in Section~\ref{Example_of_torsion} that this decomposition may
depend on the characteristic of $k$ (Theorem~\ref{theorem:6}) by considering a hexagonal permutation
as in Braden's example~\cite[Appendix A]{BW}.

Finally, we note that all the above results are valid in the affine setting with the corresponding restriction on the characteristic,
as we use only local techniques. The reader may consult, for example, \cite{Gaussent} about affine pavings.

\section{Bott-Samelson variety}\label{Bott-Samelson_varieties}


Let $G$ be a connected semisimple complex algebraic group, $T$ be its maximal torus
and $B$ be its Borel subgroup containing $T$. We denote by $W$, $\Phi$, $\Phi^+$, $\Pi$ the Weyl group, the set of all roots,
the set of positive roots and the set of simple roots respectively.

Let $\alpha$ be a root. We denote by $s_\alpha$ and $U_\alpha$ the simple reflection and
the unipotent subgroup corresponding to $\alpha$ respectively. Let $G_\alpha$ be the subgroup of $G$ generated
$U_\alpha$ and $U_{-\alpha}$. This subgroup is isomorphic either to $\SL_2(\C)$ or to $\PSL_2(\C)$.
We set $B_\alpha=G_\alpha\cap B$. Let $P_\alpha$ be the parabolic subgroup of $G$ corresponding to $\alpha$.
If $\alpha$ is simple, then $P_\alpha=B\cup Bs_\alpha B$.
We denote by $x_\alpha:\C\to U_\alpha$ the canonical homomorphism.

Throughout the paper, we fix a sequence $s=(s_1,s_2,\ldots,s_r)$ of simple reflections, where $s_i=s_{\alpha_i}$
for some $\alpha_i\in\Pi$, and consider the {\it Bott-Samelson} variety
$$
\Sigma=P_{\alpha_1}\times P_{\alpha_2}\times\cdots\times P_{\alpha_r}/B^r,
$$
where $B^r$ acts as follows:
$$
(p_1,p_2,\ldots,p_r)\cdot(b_1,b_2,\ldots,b_r)=(p_1b_1,b_1^{-1}p_2b_2,\ldots,b_{r-1}^{-1}p_rb_r).
$$
We denote by $[p_1,\ldots,p_r]$ the point of $\Sigma$ corresponding to $(p_1,\ldots,p_r)$.
It is well known that $\Sigma$ is a smooth complex variety of dimension $r$.
Moreover, $T$ acts on $\Sigma$ by $t\cdot[p_1,p_2,\ldots,p_r]=[tp_1,p_2,\ldots,p_r]$ and $\Sigma$ is covered
by open $T$-equivariant subsets
$$
U^\gamma=\{[x_{\gamma_1(-\alpha_1)}(c_1)\gamma_1,x_{\gamma_2(-\alpha_2)}(c_2)\gamma_2,\ldots,x_{\gamma_r(-\alpha_r)}(c_r)\gamma_r]\suchthat c_1,c_2,\ldots,c_r\in\C\},
$$
where $\gamma$ runs through $\Gamma$.

Let $\pi:\Sigma\to G/B$ be the map $\pi([p_1,\ldots,p_r])=p_1\cdots p_rB/B$.
For any $x\in G/B$, we fix the following notation $\Sigma_x=\pi({x})$ and $\OSigma_x=\Sigma\setminus\Sigma_x$.
We can also view $\Sigma$ as a closed subvariety of $(G/B)^r$ via the embedding $\iota:\Sigma\hookrightarrow(G/B)^r$
defined by
$$
\iota([p_1,\ldots,p_r])=(p_1B,p_1p_2B,\ldots,p_1p_2\cdots p_rB).
$$
This map is an isomorphism for $G=\SL_2(\C)$ and $G=\PSL_2(\C)$.

Each point of $G/B$ fixed by $T$ can be written uniquely as $wB$ for some $w\in W$. So, abusing notation,
we will denote this point simply by $w$. Consider the following set:
$$
\Gamma=\{(\gamma_1,\ldots,\gamma_r)\suchthat\gamma_i=s_i\text{ or }\gamma_i=e\}.
$$
The elements of this set are called {\it combinatorial galleries}. Note that $\Gamma$ can be thought of as the set of all
$T$-fixed points of $\Sigma$ if we identify $(\gamma_1,\ldots,\gamma_r)$ with $[\gamma_1,\ldots,\gamma_r]$.
For each $\gamma=(\gamma_1,\ldots,\gamma_r)\in\Gamma$ and $i=0,\ldots,r$, we denote $\gamma^i=\gamma_1\cdots\gamma_i$.
So we get $\gamma^0=e$. If additionally $i>0$, then we denote $\bbeta_i(\gamma)=\gamma^i(-\alpha_i)$ and
$\widetilde\bbeta_i(\gamma)=\gamma^{i-1}(-\alpha_i)$. If $\bbeta_i(\gamma)>0$, then we say that $i$ is {\it load bearing }
for $\gamma$ or that the wall corresponding to $\bbeta_i(\gamma)$ is {\it load bearing}. For any $A\subset W$, we denote
$$
\Gamma_A=\{\gamma\in\Gamma\suchthat\pi(\gamma)\in A\},\quad \OGamma_A=\Gamma\setminus\Gamma_A.
$$
If $A=\{x\}$, then we use the simplified notation $\Gamma_x=\Gamma_{\{x\}}$ and $\OGamma_x=\OGamma_{\{x\}}$.

For $\alpha\in\Phi^+$ and $\gamma\in\Gamma$, we set
$$
J(\gamma)=\{i\suchthat\bbeta_i(\gamma)>0\},\; M_\alpha(\gamma)=\{i\suchthat\bbeta_i(\gamma)=\pm\alpha\},\;
J_\alpha(\gamma)=\{i\suchthat\bbeta_i(\gamma)=\alpha\}=J(\gamma)\cap M_\alpha(\gamma),$$
$$
D(\gamma)=\{i\suchthat\widetilde\bbeta_i(\gamma)>0\},\quad
D_\alpha(\gamma)=\{i\suchthat\widetilde\bbeta_i(\gamma)=\alpha\}=D(\gamma)\cap M_\alpha(\gamma).
$$
Note that $\bbeta_i(\gamma)>0\Leftrightarrow\gamma^is_i<\gamma^i$ and
$\widetilde\bbeta_i(\gamma)>0\Leftrightarrow\gamma^{i-1}s_i<\gamma^{i-1}$.
Using these subsets, we introduce the following equivalence relation on $\Gamma$:
$$
\gamma\sim_\alpha\delta\Longleftrightarrow \gamma_i=\delta_i\text{ unless }\bbeta_i(\gamma)=\pm\alpha.
$$
One can easily check that $M_\alpha(\gamma)$ depends only on the $\sim_\alpha$-equivalence class of $\gamma$.

We will use the following two relations on $\Gamma$:
$$
\delta\vartriangleleft\gamma\Longleftrightarrow\delta^0=\gamma^0,\ldots,\delta^{i-1}=\gamma^{i-1},\delta^i<\gamma^i\text{ for some }i=0,\ldots,r;
$$
$$
\delta<\gamma\Longleftrightarrow\delta^i<\gamma^i,\delta^{i+1}=\gamma^{i+1},\ldots,\delta^r=\gamma^r\text{ for some }i=0,\ldots,r.
$$
Clearly, $\vartriangleleft$ is a total order on $\Gamma$, whereas $<$ in general becomes a total order only
when restricted to some $\Gamma_x$.
As usual, we set $\delta\trianglelefteq\gamma\Leftrightarrow\delta\vartriangleleft\gamma$ or $\delta=\gamma$ and
$\delta\le\gamma\Leftrightarrow\delta<\gamma$ or $\delta=\gamma$.
Note that $\delta\sim_\alpha\gamma$ and $J_\alpha(\delta)\subset J_\alpha(\gamma)$ imply $\delta\trianglelefteq\gamma$
We recall the following Lemma from~\cite{Haerterich}.

\begin{proposition}\label{lemma:4}
Let $M_\alpha(\gamma)=\{i_1<\cdots<i_\ll\}$. Then for $1\le j<\ll$ we have
$$
i_j\in J_\alpha(\gamma)\Longleftrightarrow i_{j+1}\in D_\alpha(\gamma).
$$
and $i_\ll\in J_\alpha(\gamma)$ if and only if $s_\alpha\pi(\gamma)<\pi(\gamma)$.
In particular, if $\gamma\sim_\alpha\delta$ and $\pi(\gamma)=\pi(\delta)$,
then $J_\alpha(\delta)\subset J_\alpha(\gamma)\Leftrightarrow D_\alpha(\delta)\subset D_\alpha(\gamma)$.
\end{proposition}


%
%
We use the symbol $\cdot$ to denote the addition of a new entry to a sequence: $(a_1,\ldots,a_n)\cdot b=(a_1,\ldots,a_n,b)$.
Conversely, for a nonempty sequence $a=(a_1,\ldots,a_n)$, we denote $a'=(a_1,\ldots,a_{n-1})$ its truncation.

For any integer $r\ge0$, let $\Tr_r$ denote the binary tree that consisting of all sequences (including the empty one)
with entries $0$ or $1$ of length less than $r$. We obviously have $|\Tr_r|=2^r-1$ and $\Tr_r=\Tr_{r-1}\cdot0\sqcup\Tr_{r-1}\cdot1\sqcup\{\emptyset\}$
for $r>0$. To construct bases of the $T$-equivariant cohomology of the Bott-Samelson varieties, we consider the following set:
$$
\Upsilon=\{\rho:\Tr_r\to\{s_1,\ldots,s_r\}\suchthat\rho_u=e\text{ or }\rho_u=s_{r-|u|}\}.
$$
Elements of this set are thus tree analogs of combinatorial galleries. For example, for $r=3$, we draw $\Upsilon$ as follows:
$$
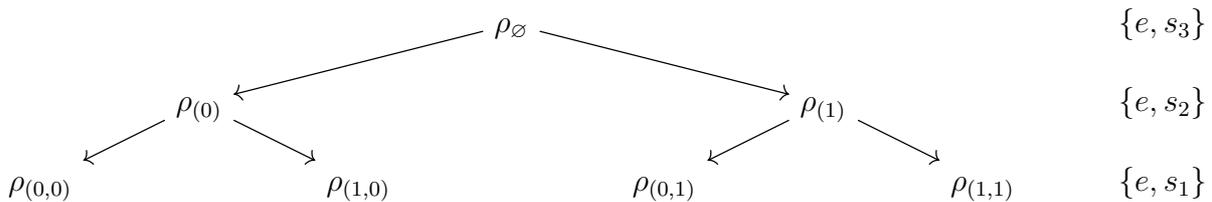
\begin{tikzcd}[row sep=small]
{}&&&\arrow{lld}\rho_\emptyset\arrow{rrd}&&&&\{e,s_3\}\\
{}&\arrow{ld}\rho_{(0)}\arrow{rd}&&&&\arrow{ld}\rho_{(1)}\arrow{rd}&&\{e,s_2\}\\
\rho_{(0,0)}&&\rho_{(1,0)}&&\rho_{(0,1)}&&\rho_{(1,1)}&\{e,s_1\}
\end{tikzcd}
$$
where the right column shows the sets to which the elements of the corresponding rows belong.

Our notation above implicitly referred to the sequence $s=(s_1,\ldots,s_r)$ and the group $G$.
If, for the sake of induction, we want to consider the same objects for the shorter sequence $s'=(s_1,\ldots,s_{r-1})$,
we add ${}'$ to our symbols: $\Sigma'$, $\Gamma'$, $\OGamma'_x$, etc. For example, let $r>0$ and $\rho\in\Upsilon$. For any
$\epsilon\in\{0,1\}$, we consider its {\it$\epsilon$-truncation} $\rho'_\epsilon\in\Upsilon'$ defined by
$(\rho'_\epsilon)_{u'}=\rho_{u'\cdot\epsilon}$
(at the picture above, $\rho'_0$ and $\rho'_1$ are the left and right subtrees respectively).

In Section~\ref{Basis_of_the_image_bar}, we consider the cases $G=\SL_2(\C)$ and $G=\PSL_2(\C)$. The sequence $s$
is characterized then only by its length $r$. We denote by $\Sigma^2_r$ and $\Gamma^2_r$ the Bott-Samelson variety
and the set of combinatorial galleries respectively.

Finally, note that we also consider the Bott-Samelson variety corresponding to the empty sequence ($r=0$).
This is the one-point variety $\Sigma=\Gamma=\{\emptyset\}$.

\section{Localization theorems}\label{Localization_theorems}

\subsection{Set theoretical notation} We denote the fact that $N$ is a subset of $M$, including the case $N=M$, by $N\subset M$,
reserving the notation $N\varsubsetneq M$ for the proper inclusion. We write $i_{M,N}:N\hookrightarrow M$
for the natural inclusion map. We sometimes write $r_{N,M}^\bullet$ for the map $H_G^\bullet(M,k)\to H_G^\bullet(N,k)$
induced by a $G$-equivariant embedding $i_{M,N}:N\hookrightarrow M$.
 We denote by $|X|$ the cardinality of a finite set $X$ and by
$\Map(X,Y)$ the set of all maps from $X$ to $Y$.
For a set $S$ with an equivalence relation $\sim$, we denote by $\rep(S,\sim)$ any set
of representatives of $\sim$-equivalence classes. Throughout the paper, ``graded'' means ``$\Z$-graded''.

\subsection{Isomorphism of localizations of modules} We want to formulate here a simple lemma from commutative algebra
whose prove is left to the reader.

Let $S$ be a (unitary) commutative ring, $M$ and $N$ be $S$-modules and $q\in S$. Consider the ring of quotients
$S'=S[q^{-1}]$ and $S'$-modules of quotients $M'=M[q^{-1}]$ and $N'=N[q^{-1}]$.
Any homomorphism of $S$-modules $f:M\to N$ gives rise to the homomorphism $f':M'\to N'$ of $S'$-modules acting as follows:
$f'(m/q^k)=f(m)/q^k$.

\begin{lemma}\label{lemma:1}
Suppose that for some integers $a,b\ge0$ the following conditions hold:
\begin{enumerate}
\item\label{lemma:1:part:1} $q^aN\subset\im f$;
\item\label{lemma:1:part:2} $q^b\ker f=0$.
\end{enumerate}
Then $f':M'\to N'$ is an isomorphism of $S'$-modules.
\end{lemma}

\subsection{The equivariant Mayer-Vietoris sequence for open subsets}\label{Mayer-Vietoris} It this paper, we consider the equivariant derived
category $D^b_T(X,k)$ for a commutative ring $k$ and a topological group $T$. For any object $\F$ of this category,
one can define the $T$-equivariant cohomology $H_T^\bullet(X,\F)$.
The basic definitions and properties of this category and $T$-equivariant cohomologies can be found in~\cite{BL}.
We shall also use the functors $f_*$, $f^*$, $f_!$, $f^!$ between equivariant derived categories defined in this book.
Remember the following  well-known result.

\begin{proposition}[Mayer-Vietoris sequence]\label{proposition:6} Let $X$ be a $T$-space.
For any open $T$-stable subsets $U$, $V$ and an object $\F\in D_T^b(X,k)$, we have the following exact sequence:
\begin{multline*}
\cdots\to H_T^{i-1}(X,\F|_{U\cap V})\to H_T^i(X,\F|_{U\cup V})\to H_T^i(U,\F|_U)\oplus H_T^i(V,\F|_V)\to\\
\to H_T^i(X,\F|_{U\cap V})\to H_T^{i+1}(X,\F|_{U\cup V})\to\cdots
\end{multline*}
\end{proposition}

In the proofs of the localization theorems, this proposition is applied as follows. Suppose that $X=U\cup V$, where
$X$ is a $T$-space and $U$, $V$ are its open $T$-stable subspaces. \linebreak Suppose additionally
that there exist elements $u\in H_T^n(\pt,k)$ and $v\in H_T^m(\pt,k)$ such that $u$ annihilates $H_T^\bullet(O_U,\F_U)$ and
$v$ annihilates $H_T^\bullet(O_V,\F_V)$ for any open $T$-stable subsets $O_U\subset U$, $O_V\subset V$ and
any objects $\F_U\in D_T^b(O_U,k)$, $\F_V\in D_T^b(O_V,k)$. Then Proposition~\ref{proposition:6}
implies that $u^2v$ and $uv^2$ annihilate $H_T^\bullet(O,\F)$ for any open $O\subset X$ and object $\F\in D_T^b(O,k)$.

Another trivial corollary of the Mayer-Vietoris sequence is as follows.

\begin{corollary}\label{corollary:2}
Let $X$ be a $T$-space, $X=\bigsqcup_{i\in I} X_i$, each $X_i$ be open and $T$-stable and $I$ be finite.
Suppose that $Y\subset X$ is another $T$-subspace. We denote $Y_i=Y\cap X_i$.
An element $f\in H^\bullet(Y,k)$ belongs to the image of the restriction $H^\bullet_T(X,k)\to H^\bullet_T(Y,k)$
if and only if each $f|_{Y_i}$ belongs to image of the restriction $H^\bullet_T(X_i,k)\to H^\bullet_T(Y_i,k)$.
\end{corollary}
\begin{proof}
Induction together with the finiteness of $I$ reduces the problem to the case $I=\{1,2\}$.
As $X_1\cap X_2=\emptyset$ and the Mayer-Vietoris sequence is compatible with restrictions, we get the following commutative
diagram with exact rows:
$$
\begin{tikzcd}
0\arrow{r}&H_T^\bullet(X,k)\arrow{d}\arrow{r}&H_T^\bullet(X_1,k)\oplus H_T^\bullet(X_2,k)\arrow{d}\arrow{r}&0\\
0\arrow{r}&H_T^\bullet(Y,k)\arrow{r}&H_T^\bullet(Y_1,k)\oplus H_T^\bullet(Y_2,k)\arrow{r}&0
\end{tikzcd}
$$
Hence the required result follows.
\end{proof}

The above arguments apply to any topological group $T$ not necessarily a torus. In this paper, we are however interested
only in the case of a torus and use the following notation:
$$
S_k=H_T^\bullet(\pt,k)\simeq S(X(T)\otimes_\Z k),
$$
where $X(T)$ is the character group of $T$ and $S$ in the right-hand side means taking the symmetric algebra.
This is a graded algebra such that $S_k^2=X(T)\otimes_\Z k$.
Finally, note that in the next section, we need to consider the compact subtorus $K=(S^1)^n$ of $T\simeq\C^n$.
We can replace $T$-equivariant cohomology with $K$-equivariant cohomology if necessary.

\subsection{Localization} We prove here some localization theorems, closely following~\cite{Brion}
(see also~\cite{FW} for the case of coefficients different from $\C$).

\begin{theorem}\label{theorem:1}
Let $\Gamma<T$ be a closed subgroup of $T$ and $X$ be a paracompact $T$-space that has an open covering
$X=\bigcup_{i\in I} Y^{(i)}$ such that for any $i\in I$
\begin{itemize}
\item $Y^{(i)}$ is open and $T$-equivariant;
\item there exists a $T$-equivariant embedding of $Y^{(i)}$ in a finite dimensional rational representation $V^{(i)}$ of $T$.
\end{itemize}
Denote $\Lm_\Gamma$ the set of all weights of $T$ occurring as weights of some $V^{(i)}$ and having nontrivial
restriction to $\Gamma$. 

Then the natural restriction morphism $H_T^\bullet(X,k)\to H_T^\bullet(X^\Gamma,k)$ becomes an isomorphism after inverting all
elements of $\Lm_\Gamma\otimes_\Z k$.
\end{theorem}
\begin{proof} For simplicity of notation, we assume that $Y^{(i)}$ is a subset of $V^{(i)}$.
Let us write
$$
V^{(i)}=\C_{\lm^{(i)}_1}\oplus\cdots\oplus\C_{\lm^{(i)}_{n_i}},
$$
where $\C_\lm$ is the representation of $T$ with weight $\lm\in X(T)$.
Let $U$ be an open $K$-invariant neighbourhood of $X^\Gamma$ in $X$.
Then the set $Y^{(i)}\setminus U$ does not have $\Gamma$-fixed points, which we prefer to write as
\begin{equation}\label{eq:1}
(Y_i\setminus U)\cap(V^{(i)})^\Gamma=\emptyset.
\end{equation}

Without loss of generality, we can assume that $\lm^{(i)}_j$ restricts trivially to $\Gamma$ if and only if $j\le m_i$.
Then $(V^{(i)})^\Gamma$ consists of the points $(c_1,\ldots,c_{m_i},0,\ldots,0)$.
Consider the open subsets $W_j^{(i)}=\{(c_1,\ldots,c_{n_i})\in V^{(i)}\suchthat c_j\ne0\}$ of $V^{(i)}$.
It follows from~(\ref{eq:1}) that
$$
Y_i\setminus U\subset\bigcup_{j=m_i+1}^{n_i}W_j^{(i)}.
$$
For any set of the union in the right-hand side, there exists a $T$-equivariant map
$W_j^{(i)}\to\C^\times_{\lm^{(i)}_j}\cong T/\ker\lm^{(i)}_j$, which is the projection to the $j^{th}$ coordinate.

Let us take an open $K$-invariant subset $O\subset (Y^{(i)}\setminus U)\cap W_j^{(i)}$ for $j>m_i$.
Then composition of maps $O\hookrightarrow W_j^{(i)}\to\C^\times_{\lm^{(i)}_j}\to\pt$ gives rise
to the following sequence of cohomologies:
$$
H_K^2(\pt,k)\to H_K^2(\C^\times_{\lm^{(i)}_j},k)\to H_K^2(W_j^{(i)},k)\to H_K^2(O,k).
$$

Identifying $T$-equivariant and $K$-equivariant
cohomologies, we obtain that the image of the first Chern class $c_1(\lm^{(i)}_j)\otimes k$ is zero
(already for the first map as follows from~\cite[\sectsign 1.9(1)]{Jantzen}).
Writing this Chern class as $\lm^{(i)}_j\otimes k$, we get that it
annihilates $H_K^\bullet(O,\F)$ for any $\F\in D^b_K(O,k)$.



Gluing all subsets $(Y^{(i)}\setminus U)\cap W_j^{(i)}$ by the Mayer-Vietoris
sequence for open subsets by the method described in Section~\ref{Mayer-Vietoris}, we get the following property:
\begin{equation}\label{eq:clubsuit}
\begin{array}{l}
\displaystyle\text{\it there exist naturals }a_i\;\text{\it such that } q=\prod_{i\in I}\prod_{j=m_i+1}^{n_i}(\lm_j^{(i)}\otimes k)^{a_i}\text{\it annihilates }H_K^\bullet(X\setminus U,\F)\\[-2pt]
\displaystyle\text{\it for any object }\F\in D^b_K(X\setminus U,k).
\end{array}
\end{equation}

Consider the following direct limit $L^n(k):={\dlim}_{U\supset X^\Gamma}H_K^n(U,k)$ that runs over
all $K$-invariant open neighbourhoods $U$ of $X^\Gamma$. Denote by $\alpha^n_U:H_K^n(U,k)\to L^n(k)$ its
natural morphisms. We define $L^\bullet(k)=\bigoplus_{n\in\Z}L^n(k)$. It is an $S_k$-module and
we get homomorphisms $\alpha_U^\bullet:H_K^\bullet(U,k)\to L^\bullet(k)$ of $S_k$-modules.

We are going to apply Lemma~\ref{lemma:1} to prove that $\alpha_X^\bullet$ becomes an isomorphism after inverting
$q$. We know that $q\in S^{2t}_k$ for some $t\in\Z$.


Let us check condition~\ref{lemma:1:part:1} of Lemma~\ref{lemma:1}. Let $u\in L^n(k)$. By the definition of the direct limit
$u=\alpha_U^n(\bar u)$ for some $\bar u\in H_K^n(U,k)$ and some $K$-invariant open $U$ containing $X^\Gamma$.
We have the following exact sequence
$$
\begin{CD}
H_K^\bullet(X,k)@>r_{U,X}^\bullet>>H_K^\bullet(U,k)@>\partial^\bullet>>H_K^{\bullet+1}(X\setminus U,i^!\csh kX),
\end{CD}
$$
where $i:X\setminus U\hookrightarrow X$ is the natural embedding. Hence and from~(\ref{eq:clubsuit}),
we get $\partial^{n+2t}(q\bar u)=q\partial^n(\bar u)=0$.
The exactness of the above sequence yields $q\bar u=r_{U,X}^{n+2t}(v)$ for some $v\in H_K^{n+2t}(X,k)$.
It remains to recall the following commutative diagram:
$$
\begin{tikzcd}
{}&L^\bullet(k)\\
H_K^\bullet(X,k)\arrow{ur}{\alpha_X^\bullet}\arrow{dr}[swap]{r_{U,X}^\bullet}&\\
&H_K^\bullet(U,k)\arrow{uu}[swap]{\alpha_U^\bullet}
\end{tikzcd}
$$
from the definition of the direct limit and write
$$
qu=q\alpha_U^n(\bar u)=\alpha_U^{n+2t}(q\bar u)=\alpha_U^{n+2t}\circ r_{U,X}^{n+2t}(v)=\alpha_X^{n+2t}(v).
$$

Let us check now condition~\ref{lemma:1:part:2} of  Lemma~\ref{lemma:1}. Take some $v\in H^n_K(X,k)$ such that
$\alpha^n_X(v)=0$. By the definition of the direct limit, we get $v|_U=0$ for some $K$-invariant open $U$ containing $X^\Gamma$.
The exactness of the sequence
$$
\begin{CD}
H_K^\bullet(X\setminus U,i^!\csh kX)@>\beta_U^\bullet>>H_K^\bullet(X,k)@>r_{U,X}^\bullet>>H_K^\bullet(U,k)
\end{CD}
$$
implies $v=\beta_U^n(w)$ for some $w\in H_K^\bullet(X\setminus U,i^!\csh kX)$.
Multiplying by $q$ and applying~(\ref{eq:clubsuit}), we get $qv=\beta_U^{n+2t}(qw)=0$.

The universal mapping property for direct limits yields the (unique) morphism $\gamma^\bullet$ such that
the diagram
$$
\begin{tikzcd}
L^\bullet(k)\arrow[dashed]{rr}{\gamma^\bullet}&&H_K^\bullet(X^\Gamma,k)\\
&H_K^\bullet(U,k)\arrow{ul}{\alpha_U^\bullet}\arrow{ur}[swap]{r^\bullet_{X^\Gamma,U}}&
\end{tikzcd}
$$
is commutative for any open $U$ containing $X^\Gamma$. By~\cite[(1.9)]{Quillen}, $\gamma^\bullet$ is an isomorphism.
It is obviously an isomorphism of $S_k$-modules.
Considering the case $U=X$ and applying the fact that $\alpha_X^\bullet$ becomes an isomorphism after inverting $q$,
we get that $r^\bullet_{X^\Gamma,X}$ also becomes an isomorphism after inverting $q$ and moreover after inverting all
elements of $\Lambda_\Gamma\otimes_\Z k$.
\end{proof}

As our next step, we explain how to adjust Theorem~6 from Brion's paper~\cite{Brion} to the case of 
arbitrary coefficients.

\begin{corollary}[cf. \mbox{\cite[Theorem 6]{Brion}}]\label{corollary:1}
Let $H_T^\bullet(X,k)$ be a free $S_k$-module. Under the hypothesis of Theorem~\ref{theorem:1} with the additional assumption
that $\Lm_T\otimes{\mathbf 1}_k$ does not contain zero divisors of $S_k$, the restriction
$$
i^*_{X,X^T}:H_T^\bullet(X,k)\to H_T^\bullet(X^T,k)
$$
is an embedding.

Moreover, if $ H^\bullet_T(X^T,k)$ does not have $S_k$-torsion (for example, $X^T$ is finite) and the following conditions hold:

{

\renewcommand{\labelenumi}{{\rm \theenumi}}
\renewcommand{\theenumi}{{\rm(C\arabic{enumi})}}

\begin{enumerate}
\item\label{corollary:1:condition:1} $k$ is a unique factorization domain;
\item\label{corollary:1:condition:2} $\lambda\otimes{\mathbf 1}_k$ is prime in $S(X(T)\otimes_\Z k)$ for any $\lm\in\Lambda_T$;
\item\label{corollary:1:condition:3} $\lm\otimes{\mathbf 1}_k\notin\Lm_{\ker\lm}\otimes_\Z k$ for any $\lm\in\Lambda_T$,
\end{enumerate}

}

then we have
$$
\im i^*_{X,X^T}=\bigcap_{\lm\in\Lm_T}\im i^*_{X^{\ker\lm},X^T}.
$$
\end{corollary}
\begin{proof} Let $M=H_T^\bullet(X,k)$, $N=H_T^\bullet(X^T,k)$, $R=\Lambda_T\otimes_\Z k$, $S=S_k$, $S'=R^{-1}S$,
$M'=R^{-1}M$, $N'=R^{-1}N$, $\phi=i^*_{X,X^T}$ and $\phi'$ be the morphism from $M'$ to $N'$ induced by $\phi$.
By Theorem~\ref{theorem:1}, $\phi'$ is an isomorphism.

Let $\{e_j\}_{j\in J}$ be an $S$-basis of $M$. Then $\{e_j/1\}_{j\in J}$ is an $S'$-basis of $M'$.
Suppose that $\phi(\alpha_1e_{j_1}+\cdots+\alpha_ke_{j_k})=0$ for $\alpha_1,\ldots,\alpha_k\in S$ and mutually distinct indices
$j_1,\ldots,j_k\in J$. We get
$$
\phi'\(\frac{\alpha_1}1\cdot\frac{e_{j_1}}1+\cdots+\frac{\alpha_k}1\cdot\frac{e_{j_k}}1\)=
\phi'\(\frac{\alpha_1e_{j_1}+\cdots+\alpha_ke_{j_k}}1\)=\frac{\phi(\alpha_1e_{j_1}+\cdots+\alpha_ke_{j_k})}1=0.
$$
Hence $\alpha_1/1=\cdots=\alpha_k/1=0$ in $S'$. Therefore $\alpha_1=\cdots=\alpha_k=0$, as $R$ does not contain zero divisors.

Now let us prove the second statement. Let $e^*_j:M\to S$ and $(e')^*_j:M'\to S'$ be the $j^{th}$ coordinate functions
for $M$ and $M'$, respectively. Consider the commutative diagram
$$
\begin{tikzcd}
S\arrow[hook]{d}[swap]{\iota}&M\arrow{l}[swap]{e_j^*}\arrow{r}{\phi}\arrow[hook]{d}&N\arrow[hook]{d}\arrow[dashed]{dll}\\
S'&M'\arrow{l}{(e')_j^*}&N'\arrow{l}{(\phi')^{-1}}[swap]{\sim}
\end{tikzcd}
$$
Denoting the dashed arrow by $f_j$, we get the following relation:
\begin{equation}\label{eq:4}
f_j\circ \phi=\iota\circ e^*_j.
\end{equation}
Note that all functions $f_j$ uniquely define elements of $N$:
\begin{equation}\label{eq:3}
f_j(u)=f_j(u')\;\forall j\in J\Longrightarrow u=u'.
\end{equation}

Let us take $u\in\bigcap_{\lm\in\Lambda_T}\im  i^*_{X^{\ker\lm},X^T}$. 
Consider the coefficients $f_j(u)\in S'$. If they all belong to $\iota(S)$, then in view of~(\ref{eq:4}),
the following calculation is possible:
$$
f_j\circ \phi\(\sum_{j\in J}\iota^{-1}(f_j(u))e_j\)=\iota\circ e^*_j\(\sum_{j\in J}\iota^{-1}(f_j(u))e_j\)=f_j(u).
$$
Now~(\ref{eq:3}) implies that $u=\phi\(\sum_{j\in J}f_j(u)e_j\)\in\im i^*_{X,X^T}$.

It only remains to prove that $f_j(u)\in S$ for all $j\in J$. Suppose the contrary holds.
By~\ref{corollary:1:condition:2}, in this case $f_j(u)$ contains an uncancellable prime denominator $\lm\otimes{\mathbf 1}_k$
for some $j\in J$ and $\lm\in\Lm_T$.

To proceed, let us introduce the following notation: $\Gamma=\ker\lm$, $N_\lm=H_T^\bullet(X^\Gamma,k)$,
$R_\lm=\Lm_\Gamma\otimes_\Z k$, $S'_\lm=R_\lm^{-1}S$, $M'_\lm=R_\lm^{-1}M$, $N'_\lm=R_\lm^{-1}N_\lm$,
$\phi_\lm=i^*_{X,X^\Gamma}$ and $\phi'_\lm$ be the morphism from $M'_\lm$ to $N'_\lm$ induced by $\phi_\lm$.
By theorem~\ref{theorem:1}, $\phi'_\lm$ is an isomorphism.

As $u\in\im i^*_{X^\Gamma,X^T}$, we can write $u=i^*_{X^\Gamma,X^T}(v)$ for some $v\in N_\lm$.
Similarly to the diagram above, we have the following commutative diagram:
$$
\begin{tikzcd}
M\arrow{r}{\phi_\lm}\arrow[hook]{d}&N_\lm\arrow{d}\\
M'_\lm&N'_\lm\arrow{l}{(\phi'_\lm)^{-1}}[swap]{\sim}
\end{tikzcd}
$$
There exists some product $\P_\lm$ of elements of $R_\lm$ such that
$(\P_\lm/1)(\phi'_\lm)^{-1}(v/1)=m/1$ for some $m\in M$. Applying $\phi'_\lm$ to this equality, we get
$\P_\lm v/1=(\phi'_\lm)(m/1)=\phi_\lm(m)/1$, which is an equality in $N'_\lm$. Therefore,
there exists another product $\P'_\lm$ of elements of $R_\lm$ such that
$$
\P'_\lm\P_\lm v=\P'_\lm\phi_\lm(m)=\phi_\lm(\P'_\lm m).
$$
Applying $i_{X^\Gamma,X^T}^*$ to the both sides of this equality, we get
$$
\P'_\lm\P_\lm u=i_{X^\Gamma,X^T}^*(\P'_\lm\P_\lm v)=i_{X^\Gamma,X^T}^*\circ\phi_\lm(\P'_\lm m)=\phi(\P'_\lm m).
$$
Finally applying $f_j$, we get
$$
(\P'_\lm\P_\lm/1)f_j(u)=f_j(\P'_\lm\P_\lm u)=f_j\circ \phi(\P'_\chi m)=e_j^*(\P'_\chi m)/1\in\iota(S).
$$
This is a contradiction, as $\P'_\lm\P_\lm$ by our GKM-restriction~\ref{corollary:1:condition:3} does not have
factors proportional to $\lm\otimes{\mathbf 1}_k$.
\end{proof}

\subsection{Case of Bott-Samelson varieties} In what follows, we shall only consider the case
where the ring of coefficients $k$ is a PID 
of characteristic not equal to $2$ if the root system contains a component of type $C_n$.
As the ordinary cohomology $H^\bullet(\Sigma,k)$ vanishes in odd degrees and is a free $k$-module in each degree,
the degeneracy of the Leray spectral sequence at the $E_2$-term implies
$$
H_T^\bullet(\Sigma,k)\simeq H^\bullet(\Sigma,k)\otimes_kS_k.
$$
Therefore, we can apply the first part of Corollary~\ref{corollary:1} to prove that the restriction morphism
$H_T^\bullet(\Sigma,k)\to H_T^\bullet(\Gamma,k)$ is an embedding. We denote its image by $\X(k)$.

Similarly,
$
H_T^\bullet(\Sigma_x,k)\simeq H^\bullet(\Sigma_x,k)\otimes_kS_k
$
and we can apply Corollary~\ref{corollary:1} to prove that the restriction morphism
$H_T^\bullet(\Sigma_x,k)\to H_T^\bullet(\Gamma_x,k)$ is an embedding. We denote its image by $\X^x(k)$.

In order to ensure conditions~\ref{corollary:1:condition:1}--\ref{corollary:1:condition:3} of Corollary~\ref{corollary:1},
we want to fix ring $\Z'$ for each root system as follows: $\Z'=\Z[1/2]$ if the root system contains a component of type $C_n$
and $\Z'=\Z$ otherwise. This choice automatically guarantees that Theorem~\ref{theorem:1} and Corollary~\ref{corollary:1}
hold for $k=\Z'$, since $\Lambda_T\subset\Phi$ in these assertions.

Therefore, from now on, we will assume that the cohomologies (ordinary and equivariant) are taken with
coefficients $\Z'$ unless otherwise explicitly stated. We also set $S=S_{\Z'}$, $\X=\X(\Z')$ and $\X^x=\X^x(\Z')$.

%
%
%

\section{Bases of the images $\X$ and $\X^x$}\label{Bases_of_XXx}

\subsection{H\"arterich's localization theorems} We formulate here the following two results due to H\"arterich~\cite{Haerterich}.
It is important to note that one needs to be more careful with the ring of coefficients when applying the localization theorems
in the proofs of these results. For our ring of coefficients $\Z'$, one can apply Theorem~\ref{theorem:1} and
Corollary~\ref{corollary:1}.


\begin{proposition}[\mbox{\cite[Theorem 6.2]{Haerterich}}]\label{proposition:1}
An element $f\in H_T^\bullet(\Gamma)$ belongs to the image $\X$ of the restriction
$i_{\Sigma,\Gamma}^*:H_T^\bullet(\Sigma)\to H_T^\bullet(\Gamma)$ if and only if
$$
\sum_{\delta\in\Gamma,\delta\sim_\alpha\gamma,J_\alpha(\delta)\subset J_\alpha(\gamma)}(-1)^{|J_\alpha(\delta)|}f(\delta)\=0\pmod{\alpha^{|J_\alpha(\gamma)|}}
$$
for any positive root $\alpha$ and gallery $\gamma\in\Gamma$.
\end{proposition}


\begin{proposition}[\mbox{\cite[Theorem 6.3]{Haerterich}}]\label{proposition:2} An element $f\in H_T^\bullet(\Gamma_x)$ belongs to the image $\X^x$ of the restriction
$i_{\Sigma_x,\Gamma_x}^*:H_T^\bullet(\Sigma_x)\to H_T^\bullet(\Gamma_x)$ if and only if
$$
\sum_{\delta\in\Gamma_x,\delta\sim_\alpha\gamma,D_\alpha(\delta)\subset D_\alpha(\gamma)}(-1)^{|D_\alpha(\delta)|}f(\delta)\=0\pmod{\alpha^{|D_\alpha(\gamma)|}}
$$
for any positive root $\alpha$ and gallery $\gamma\in\Gamma_x$.
\end{proposition}

The reader can find the proofs of these results either in H\"arterich's original preprint~\cite{Haerterich}
or derive them from the proof of Proposition~\ref{proposition:3} by similarity.

\subsection{Copy and concentration}\label{Copy and concentration}
In this section, we describe two ways to get elements of $\X$ from elements of $\X'$. 
Suppose that $r>0$.
For $f'\in H^\bullet_T(\Gamma')$, we define its {\it copy} $\Delta f'\in H^\bullet_T(\Gamma)$ by
$\Delta f'(\gamma)=f'(\gamma')$ for any $\gamma\in\Gamma$.
Clearly, $\Delta$ is an $S$-linear operation.

\begin{lemma}\label{lemma:5}
$\Delta f'\in\X$ if $f'\in\X'$.
\end{lemma}
\begin{proof} By Proposition~\ref{proposition:1}, we must prove that
\begin{equation}\label{eq:17}
\sum_{\delta\in\Gamma,\delta\sim_\alpha\gamma,J_\alpha(\delta)\subset J_\alpha(\gamma)}(-1)^{|J_\alpha(\delta)|}f'(\delta')\=0\pmod{\alpha^{|J_\alpha(\gamma)|}}
\end{equation}
for any $\gamma\in\Gamma$ and $\alpha\in\Phi^+$. We shall use the notation $M_\alpha(\gamma)=\{i_1<\cdots<i_\ll\}$ and $x=\pi(\gamma)$.

{\it Case 1: $r\notin M_\alpha(\gamma)$.}
In this case, $\delta\sim_\alpha\gamma$ implies $\delta_r=\gamma_r$. Therefore, we can rewrite~(\ref{eq:17}) as follows
\begin{equation}\label{eq:18}
\sum_{\delta'\in\Gamma',\delta'\sim_\alpha\gamma',J_\alpha(\delta')\subset J_\alpha(\gamma')}(-1)^{|J_\alpha(\delta')|}f'(\delta')\=0\pmod{\alpha^{|J_\alpha(\gamma')|}},
\end{equation}
which holds by Proposition~\ref{proposition:1} applied to $f'\in\X'$.

{\it Case 2: $r\in M_\alpha(\gamma)\setminus J_\alpha(\gamma)$.}
%
%
%
Choosing in~(\ref{eq:17}) the gallery $\delta$ so that $r\notin J_\alpha(\delta)$, we can rewrite this equivalence as~(\ref{eq:18}).

{\it Case 3: $r\in J_\alpha(\gamma)$.} Consider the following equivalence relation on the set
$\{\delta\in\Gamma\suchthat\delta\sim_\alpha\gamma\}$: $\delta\=\tau\Leftrightarrow\delta'=\tau'$.
Clearly, every equivalence class of this relation consists of exactly to elements. Therefore the sum in~(\ref{eq:17})
can be broken into a sum of the following subsums:
$$
(-1)^{|J_\alpha(\delta)|}f'(\delta')+(-1)^{|J_\alpha(\tau)|}f'(\tau')
$$
for different $\delta\=\tau$. As $|J_\alpha(\delta)|$ and $|J_\alpha(\tau)|$ have different parities, the above sum equals zero.
\end{proof}

For $f'\in H^\bullet_T(\Gamma')$ and $t\in\{e,s_r\}$, we define
$\nabla_tf'\in H^\bullet_T(\Gamma)$ called the {\it concentration} of $f'$ at $t$ by
$$
\nabla_tf'(\gamma)=
\left\{
\begin{array}{ll}
\bbeta_r(\gamma)f'(\gamma')&\text{ if }\gamma_r=t;\\[6pt]
0&\text{ otherwise}.
\end{array}
\right.
$$
for any $\gamma\in\Gamma$. Clearly, $\nabla_t$ is an $S$-linear operation.

\begin{lemma}\label{lemma:6}
$\nabla_tf'\in\X$ if $f'\in\X'$.
\end{lemma}
\begin{proof} We shall write down the proof for $\nabla_e f'$, the proof for $\nabla_{s_r}f'$ being similar.

By Proposition~\ref{proposition:1}, we must prove that
\begin{equation}\label{eq:19}
\sum_{\delta\in\Gamma,\delta\sim_\alpha\gamma,\delta_r=e,J_\alpha(\delta)\subset J_\alpha(\gamma)}(-1)^{|J_\alpha(\delta)|}\bbeta_r(\delta)f'(\delta')\=0\pmod{\alpha^{|J_\alpha(\gamma)|}}
\end{equation}
for any $\gamma\in\Gamma$ and $\alpha\in \Phi^+$. Clearly, it suffices to consider the case $J_\alpha(\gamma)\ne\emptyset$. We shall use the notation $M_\alpha(\gamma)=\{i_1<\cdots<i_\ll\}$ and $x=\pi(\gamma)$.

{\it Case 1: $r\notin M_\alpha(\gamma)$.} In this case, $\delta\sim_\alpha\gamma$ implies $\delta_r=\gamma_r$.
Thus it suffices to consider the case $\gamma_r=e$, as otherwise our sum is equal to zero.
We can rewrite~(\ref{eq:19}) as follows
\begin{multline*}
\sum_{\delta\in\Gamma_x,\delta\sim_\alpha\gamma,\delta_r=e,J_\alpha(\delta)\subset J_\alpha(\gamma)}(-1)^{|J_\alpha(\delta)|}x(-\alpha_r)f'(\delta')\\
+\sum_{\delta\in\Gamma_{s_\alpha x},\delta\sim_\alpha\gamma,\delta_r=e,J_\alpha(\delta)\subset J_\alpha(\gamma)}(-1)^{|J_\alpha(\delta)|}s_\alpha x(-\alpha_r)f'(\delta')\=0\pmod{\alpha^{|J_\alpha(\gamma)|}}.
\end{multline*}
As $s_\alpha x(-\alpha_r)\=x(-\alpha_r)\pmod\alpha$, it suffices to prove that
\begin{equation}\label{eq:20}
\sum_{\delta\in\Gamma,\delta\sim_\alpha\gamma,\delta_r=e,J_\alpha(\delta)\subset J_\alpha(\gamma)}(-1)^{|J_\alpha(\delta)|}f'(\delta')\=0\pmod{\alpha^{|J_\alpha(\gamma)|}}
\end{equation}
and
\begin{equation}\label{eq:21}
\sum_{\delta\in\Gamma_x,\delta\sim_\alpha\gamma,\delta_r=e,J_\alpha(\delta)\subset J_\alpha(\gamma)}(-1)^{|J_\alpha(\delta)|}f'(\delta')\=0\pmod{\alpha^{|J_\alpha(\gamma)|-1}}.
\end{equation}
We can rewrite~(\ref{eq:20}) as follows
$$
\sum_{\delta'\in\Gamma',\delta'\sim_\alpha\gamma',J_\alpha(\delta')\subset J_\alpha(\gamma')}(-1)^{|J_\alpha(\delta')|}f'(\delta')\=0\pmod{\alpha^{|J_\alpha(\gamma')|}}.
$$
It holds by Proposition~\ref{proposition:1} applied to $f'\in\X'$.
Noting that $J_\alpha(\delta)\subset J_\alpha(\gamma)$ is equivalent to $D_\alpha(\delta')\subset D_\alpha(\gamma')$ in~(\ref{eq:21})
by Proposition~\ref{lemma:4}, we can rewrite~(\ref{eq:21}) as follows
$$
\sum_{\delta'\in\Gamma_x,\delta'\sim_\alpha\gamma',D_\alpha(\delta')\subset D_\alpha(\gamma')}(-1)^{|J_\alpha(\delta')|}f'(\delta')\=0\pmod{\alpha^{|J_\alpha(\gamma)|-1}}.
$$
By Proposition~\ref{lemma:4}, we have
\begin{equation}\label{eq:23}
|D_\alpha(\gamma')|\ge|J_\alpha(\gamma')|-1=|J_\alpha(\gamma)|-1
\end{equation}
and
$|J_\alpha(\delta')|=|D_\alpha(\delta')|$ for $s_\alpha x>x$ and $|J_\alpha(\delta')|=|D_\alpha(\delta')|+1$ for $s_\alpha x<x$.
Therefore, the above equivalence follows from Proposition~\ref{proposition:2} applied to the element $f'\big|_{\Gamma'_x}$,
which belongs to ${\X'}^x$ as is easy to see from the commutative diagram
$$
\begin{tikzcd}
H^\bullet_T(\Sigma')\arrow{r}\arrow{d}&H^\bullet_T(\Sigma'_x)\arrow{d}\\
H^\bullet_T(\Gamma')\arrow{r}&H^\bullet_T(\Gamma'_x)
\end{tikzcd}
$$

{\it Case 2: $r\in M_\alpha(\gamma)\setminus J_\alpha(\gamma)$.} In this case, $i_\ll=r$, $|D_\alpha(\gamma)|=|J_\alpha(\gamma)|$
and $s_\alpha x>x$ by Proposition~\ref{lemma:4}. If $\delta$ belonged to $\Gamma_{s_\alpha x}$ in~(\ref{eq:19}),
we would get by Proposition~\ref{lemma:4}
that $r\in J_\alpha(\delta)$ and thus the inclusion $J_\alpha(\delta)\subset J_\alpha(\gamma)$ would not hold.
On the other hand, for any $\delta\in\Gamma_x$ such that $\delta\sim_\alpha\gamma$, we have
$r\in M_\alpha(\delta)\setminus J_\alpha(\delta)$, whence $\bbeta_r(\delta)=-\alpha$.
Therefore, it suffices to prove that
\begin{equation}\label{eq:22}
\sum_{\delta'\in\Gamma_x,\delta'\sim_\alpha\gamma',J_\alpha(\delta')\subset J_\alpha(\gamma')}(-1)^{|J_\alpha(\delta')|}f'(\delta')\=0\pmod{\alpha^{|J_\alpha(\gamma)|-1}}.
\end{equation}
If $\gamma'\in\Gamma_x$, then by Proposition~\ref{lemma:4} the summation runs over $\delta'\in\Gamma_x$ such that
$\delta'\sim_\alpha\gamma'$ and $D_\alpha(\delta')\subset D_\alpha(\gamma')$. Therefore, (\ref{eq:22}) follows
from~(\ref{eq:23}) and~Proposition~\ref{proposition:2} applied to $f'\big|_{\Gamma'_x}$.

We assume now that $\gamma'\in\Gamma_{xs_r}=\Gamma_{s_\alpha x}$. Note that $M_\alpha(\gamma')=\{i_1<\cdots<i_{\ll-1}\}$.
This set is not empty (i.e., $\ll>1$), as $s_\alpha\pi(\gamma')=x<s_\alpha x=\pi(\gamma')$, whence $i_{\ll-1}\in J_\alpha(\gamma')$.
Consider the gallery ${\widetilde\gamma}'$ that is obtained from $\gamma'$ by replacing $\gamma_{i_{\ll-1}}$ with
$\gamma_{i_{\ll-1}}s_{i_{\ll-1}}$. We clearly have
$$
{\widetilde\gamma}'\sim_\alpha\gamma',\qquad {\widetilde\gamma}'\in\Gamma_x,\qquad J_\alpha({\widetilde\gamma}')=J_\alpha(\gamma')\setminus\{i_{\ll-1}\},\qquad D_\alpha({\widetilde\gamma}')=D_\alpha(\gamma')
$$
Finally it remains to note that in~(\ref{eq:22}), we have $s_\alpha\pi(\delta')=s_\alpha x>x=\pi(\delta')$,
whence $i_{\ll-1}\notin J_\alpha(\delta')$. Thus $J_\alpha(\delta')\subset J_\alpha(\gamma')$ is equivalent
to $J_\alpha(\delta')\subset J_\alpha({\widetilde\gamma}')$ and hence by Proposition~\ref{lemma:4} to
$D_\alpha(\delta')\subset D_\alpha({\widetilde\gamma}')$. Thus we can rewrite~(\ref{eq:22}) as follows
$$
\sum_{\delta'\in\Gamma_x,\delta'\sim_\alpha{\widetilde\gamma}',D_\alpha(\delta')\subset D_\alpha({\widetilde\gamma}')}(-1)^{|J_\alpha(\delta')|}f'(\delta')\=0\pmod{\alpha^{|J_\alpha(\gamma)|-1}}.
$$
This equivalence again follows from~(\ref{eq:23}) and~Proposition~\ref{proposition:2} applied to $f'\big|_{\Gamma'_x}$.

{\it Case 3: $r\in J_\alpha(\gamma)$.} In this case, $i_\ll=r$ 
and $s_\alpha x<x$. We can rewrite~(\ref{eq:19}) as follows
\begin{multline*}
\sum_{\delta'\in\Gamma_x,\delta'\sim_\alpha\gamma',J_\alpha(\delta')\subset J_\alpha(\gamma')}(-1)^{|J_\alpha(\delta')|+1}\alpha f'(\delta')\\[6pt]
-\sum_{\delta'\in\Gamma_{s_\alpha x},\delta'\sim_\alpha\gamma',J_\alpha(\delta')\subset J_\alpha(\gamma')}(-1)^{|J_\alpha(\delta')|}\alpha f'(\delta')\=0\pmod{\alpha^{|J_\alpha(\gamma)|}}.
\end{multline*}
It suffices to prove that
\begin{multline*}
\sum_{\delta'\in\Gamma_x,\delta'\sim_\alpha\gamma',J_\alpha(\delta')\subset J_\alpha(\gamma')}(-1)^{|J_\alpha(\delta')|}f'(\delta')\\[6pt]
+\sum_{\delta'\in\Gamma_{s_\alpha x},\delta'\sim_\alpha\gamma',J_\alpha(\delta')\subset J_\alpha(\gamma')}(-1)^{|J_\alpha(\delta')|}f'(\delta')\=0\pmod{\alpha^{|J_\alpha(\gamma)|-1}},
\end{multline*}
which follows from Proposition~\ref{proposition:1}, as  $|J_\alpha(\gamma)|-1=|J_\alpha(\gamma')|$.
\end{proof}

For notational purposes, its convenient to define
$$
\widetilde\nabla_t f'(\gamma)=
\left\{
\begin{array}{ll}
\widetilde\bbeta_r(\gamma)f'(\gamma')&\text{ if }\gamma_r=t;\\[6pt]
0&\text{ otherwise}
\end{array}
\right.
$$
\begin{corollary}\label{corollary:3}
$\widetilde\nabla_tf'\in\X$ if $f'\in\X'$
\end{corollary}
\begin{proof}
The result follows from $\widetilde\nabla_e f'=\nabla_e f'$ and $\widetilde\nabla_{s_r}f'=-\nabla_{s_r}f'$.
\end{proof}

\subsection{Folding the ends}\label{Folding the ends} For $r>0$, we define the automorphism $\gamma\mapsto\dot\gamma$ of $\Gamma$ by
$$
\dot\gamma_i=\left\{
\begin{array}{ll}
\gamma_i&\text{ if }i<r;\\
s_r\gamma_r&\text{ if }i=r.
\end{array}
\right.
$$
It satisfies the following properties:
\begin{itemize}
\item $M_\alpha(\dot\gamma)=M_\alpha(\gamma)$;\\[-8pt]
\item $\dot\delta\sim_\alpha\dot\gamma\Longleftrightarrow\delta\sim_\alpha\gamma$;\\[-8pt]
\item If $r\notin M_\alpha(\gamma)$, then $J_\alpha(\dot\gamma)=J_\alpha(\gamma)$. If $r\in M_\alpha(\gamma)$,
       then $J_\alpha(\dot\gamma)=J_\alpha(\gamma)\simd\{r\}$, where $\simd$ stands for the symmetric difference;\\[-8pt]
\item $D_\alpha(\dot\gamma)=D_\alpha(\gamma)$;\\[-8pt]
\item $\dot\gamma\in\Gamma_x\Leftrightarrow\gamma\in\Gamma_{xs_r}$,
\end{itemize}
whose proofs are left to the reader.

This automorphism of $\Gamma$ induces an automorphism of $H^\bullet_T(\Gamma)$ by $\dot f(\gamma)=f(\dot\gamma)$.
Clearly, these automorphisms are of order $2$.

\begin{lemma}\label{lemma:7}
$\dot\X=\X$,\; $\dot\X^x=\X^{xs_r}$.
\end{lemma}
\begin{proof}
Actually we only have to prove that $\dot\X\subset\X$. Take any $f\in\X$.
By Proposition~\ref{proposition:1}, we must check the equivalence
\begin{equation}\label{eq:24}
\sum_{\delta\in\Gamma,\delta\sim_\alpha\gamma,J_\alpha(\delta)\subset J_\alpha(\gamma)}(-1)^{|J_\alpha(\delta)|}f(\dot\delta)\=0\pmod{\alpha^{|J_\alpha(\gamma)|}}
\end{equation}
for arbitrary $\gamma\in\Gamma$ and $\alpha\in \Phi^+$.

{\it Case 1: $r\notin M_\alpha(\gamma)$.} In this case, we can rewrite~(\ref{eq:24}) as follows
$$
\sum_{\delta\in\Gamma,\dot\delta\sim_\alpha\dot\gamma,J_\alpha(\dot\delta)\subset J_\alpha(\dot\gamma)}(-1)^{|J_\alpha(\dot\delta)|}f(\dot\delta)\=0\pmod{\alpha^{|J_\alpha(\dot\gamma)|}}.
$$
It holds by Proposition~\ref{proposition:1}.

{\it Case 2: $r\in M_\alpha(\gamma)\setminus J_\alpha(\gamma)$.} In this case, $\gamma\sim_\alpha\dot\gamma$ and
$J_\alpha(\dot\gamma)=J_\alpha(\gamma)\sqcup\{r\}$. We can rewrite~(\ref{eq:24}) as follows\footnote{If $r\notin A$, then $B\subset A$ if and only if $r\in B\simd\{r\}$ and $B\simd\{r\}\subset A\simd\{r\}$.}
\begin{equation}\label{eq:25}
\sum_{\delta\in\Gamma,\dot\delta\sim_\alpha\dot\gamma,r\in J_\alpha(\dot\delta),J_\alpha(\dot\delta)\subset J_\alpha(\dot\gamma)}(-1)^{|J_\alpha(\dot\delta)|}f(\dot\delta)\=0\pmod{\alpha^{|J_\alpha(\gamma)|}}.
\end{equation}
To prove it, let us write the following two equivalences
$$
\sum_{\delta\in\Gamma,\dot\delta\sim_\alpha\dot\gamma,J_\alpha(\dot\delta)\subset J_\alpha(\dot\gamma)}(-1)^{|J_\alpha(\dot\delta)|}f(\dot\delta)\=0\pmod{\alpha^{|J_\alpha(\dot\gamma)|}},
$$
$$
\sum_{\delta\in\Gamma,\dot\delta\sim_\alpha\gamma,J_\alpha(\dot\delta)\subset J_\alpha(\gamma)}(-1)^{|J_\alpha(\dot\delta)|}f(\dot\delta)\=0\pmod{\alpha^{|J_\alpha(\gamma)|}},
$$
which hold by Proposition~\ref{proposition:1}. Subtracting the latter from the former and considering everything modulo $\alpha^{|J_\alpha(\gamma)|}$,
we get~(\ref{eq:25}).

{\it Case 3: $r\in J_\alpha(\gamma)$.} In this case, $\gamma\sim_\alpha\dot\gamma$ and we can rewrite~(\ref{eq:24})
as follows\footnote{If $r\in A$, then $B\subset A$ if and only if $B\simd\{r\}\subset A$.}
$$
-\sum_{\delta\in\Gamma,\dot\delta\sim_\alpha\gamma,J_\alpha(\dot\delta)\subset J_\alpha(\gamma)}(-1)^{|J_\alpha(\dot\delta)|}f(\dot\delta)\=0\pmod{\alpha^{|J_\alpha(\gamma)|}}.
$$
It holds by Proposition~\ref{proposition:1}.
\end{proof}

\subsection{Fixing the ends}\label{Fixing the ends} Let $r>0$. Consider the natural embedding $\iota:\Sigma'\hookrightarrow\Sigma$ defined by
$[p_1,\ldots,p_{r-1}]\mapsto[p_1,\ldots,p_{r-1},e]$. This is a $B$-equivariant hence also a $T$-equivariant embedding.
We get the following commutative diagram for restrictions:
$$
\begin{tikzcd}
H^\bullet_T(\Sigma)\arrow{d}\arrow{r}&H^\bullet_T(\Sigma')\arrow{d}\\
H^\bullet_T(\Gamma)\arrow{r}&H^\bullet_T(\Gamma')
\end{tikzcd}
$$
Let $f$ be an element of $\X$ (i.e., in the image of the left arrow).
It follows from the commutativity of the above diagram that the composition $f'=f\circ\iota$
belongs to $\X'$ (i.e., to the image of the right arrow).

\begin{lemma}\label{lemma:8}
Let $f\in\X$, $r>0$ and $t\in\{e,s_r\}$. We define $f'\in H^\bullet_T(\Gamma')$ by $f'(\gamma')=f(\gamma'\cdot t)$.
Then $f'\in\X'$.
\end{lemma}
\begin{proof}
The argument preceding the formulation of this lemma proves the claim for $t=e$. Now let $t=s_r$.
By Lemma~\ref{lemma:7}, we get $\dot f\in\X$. Then by the case $t=e$, we get $\dot f\circ\iota\in\X'$.
Now the result follows from
$$
\dot f\circ\iota(\gamma')=\dot f(\gamma'\cdot e)=f(\gamma'\cdot s_r)=f'(\gamma').
$$
\end{proof}

\begin{lemma}\label{lemma:9}
Let $f\in\X$, $r>0$ and $t\in\{e,s_r\}$. Suppose that $f(\gamma)=0$ for $\gamma_r\ne t$.
Then $f(\gamma)$ is divisible in $S$ by $\bbeta_r(\gamma)$ for any $\gamma\in\Gamma$.
Moreover, the function $\gamma'\mapsto f(\gamma'\cdot t)/\bbeta_r(\gamma'\cdot t)$, where $\gamma'\in\Gamma'$, belongs to $\X'$.
\end{lemma}
\begin{proof}
We shall prove the first claim by induction with respect to $\trianglelefteq$.
So suppose that $f(\delta)$ is divisible by $\bbeta_r(\delta)$ for any $\delta\vartriangleleft\gamma$.
We must prove that $f(\gamma)$ is divisible by $\bbeta_r(\gamma)$. Clearly, we need only to consider the case $\gamma_r=t$.

We take for $\alpha$ the positive of the two roots $\bbeta_r(\gamma)$ and $-\bbeta_r(\gamma)$.
Thus $r\in M_\alpha(\gamma)$.

{\it Case 1: $r\in J_\alpha(\gamma)$.} In this case $|J_\alpha(\gamma)|>0$. Thus by Proposition~\ref{proposition:1},
we get
$$
\sum_{\delta\in\Gamma,\delta\sim_\alpha\gamma,J_\alpha(\delta)\subset J_\alpha(\gamma)}(-1)^{|J_\alpha(\delta)|}f(\delta)\=0\pmod{\alpha}.
$$
As $\delta\sim_\alpha\gamma$ and $J_\alpha(\delta)\subset J_\alpha(\gamma)$ imply $\delta\trianglelefteq\gamma$,
the claim follows.

{\it Case 2: $r\notin J_\alpha(\gamma)$.} In this case, $r\in J_\alpha(\dot\gamma)$, whence $|J_\alpha(\dot\gamma)|>0$.
Moreover, $\dot\gamma\sim_\alpha\gamma$.
By Proposition~\ref{proposition:1}, we get
\begin{equation}\label{eq:27}
\sum_{\delta\in\Gamma,\delta\sim_\alpha\gamma,\delta_r=t,J_\alpha(\delta)\subset J_\alpha(\dot\gamma)}(-1)^{|J_\alpha(\delta)|}f(\delta)\=0\pmod{\alpha}.
\end{equation}
We claim that
\begin{equation}\label{eq:26}
\delta\sim_\alpha\gamma,\delta_r=t,J_\alpha(\delta)\subset J_\alpha(\dot\gamma)\Longrightarrow\delta\trianglelefteq\gamma.
\end{equation}
We get $\delta\vartriangleleft\dot\gamma$ as $\delta_r=t\ne\dot\gamma_r$. Thus there exists some $i_0$ such that
$\delta^{i_0}<\dot\gamma^{i_0}$ and $\delta_i=\gamma_i$ for $i<i_0$.
If $i_0<r$ then $\delta^{i_0}<\gamma^{i_0}$ and thus $\delta\vartriangleleft\gamma$.
On the other hand, if $i_0=r$, then $\gamma=\delta$, since $\delta_r=t=\gamma_r$.

Now it follows from~(\ref{eq:27}),~(\ref{eq:26}) and the inductive hypothesis that $f(\gamma)$ is divisible by $\alpha$.

Let us prove the second claim. We denote by $f'$ the function under consideration:
$f'(\gamma')=f(\gamma'\cdot t)/\bbeta_r(\gamma'\cdot t)$.
By Proposition~\ref{proposition:1}, we must check the equivalence
\begin{equation}\label{eq:31}
\sum_{\delta'\in\Gamma',\delta'\sim_\alpha\gamma',J_\alpha(\delta')\subset J_\alpha(\gamma')}(-1)^{|J_\alpha(\delta')|}f'(\delta')\=0\pmod{\alpha^{|J_\alpha(\gamma')|}}
\end{equation}
for any $\gamma'\in\Gamma'$ and $\alpha\in \Phi^+$. Clearly, we can assume that $J_\alpha(\gamma')\ne \emptyset$. We set $\gamma:=\gamma'\cdot t$. Let us fix the notation
$$
M_\alpha(\gamma)=\{i_1<\cdots<i_\ll\},\qquad y=\pi(\gamma').
$$
By Proposition~\ref{proposition:1}, we get
\begin{equation}\label{eq:28}
\sum_{\delta\in\Gamma,\delta\sim_\alpha\gamma,J_\alpha(\delta)\subset J_\alpha(\gamma)}(-1)^{|J_\alpha(\delta)|}f(\delta)\=0\pmod{\alpha^{|J_\alpha(\gamma)|}}.
\end{equation}

{\it Case~1: $r\notin M_\alpha(\gamma)$}. 
We can rewrite~(\ref{eq:28}) as follows
\begin{multline*}
yt(-\alpha_r)\sum_{\delta\in\Gamma_{yt},\delta\sim_\alpha\gamma,\delta_r=t,J_\alpha(\delta)\subset J_\alpha(\gamma)}(-1)^{|J_\alpha(\delta)|}f'(\delta')\\[6pt]
+s_\alpha yt(-\alpha_r)\sum_{\delta\in\Gamma_{s_\alpha yt},\delta\sim_\alpha\gamma,\delta_r=t,J_\alpha(\delta)\subset J_\alpha(\gamma)}(-1)^{|J_\alpha(\delta)|}f'(\delta')
\=0\pmod{\alpha^{|J_\alpha(\gamma)|}}.
\end{multline*}
We have $s_\alpha yt(-\alpha_r)=yt(-\alpha_r)+c\alpha$ for some $c\in\Z$. Thus the above equivalence takes form
\begin{equation}\label{eq:29}
\begin{array}{l}
\displaystyle\hspace{-10pt} s_\alpha yt(-\alpha_r)\sum_{\delta'\in\Gamma',\delta'\sim_\alpha\gamma',J_\alpha(\delta')\subset J_\alpha(\gamma')}(-1)^{|J_\alpha(\delta')|}f'(\delta')\\[20pt]
\displaystyle \hspace{100pt}-c\alpha\sum_{\delta'\in\Gamma'_y,\delta'\sim_\alpha\gamma',J_\alpha(\delta')\subset J_\alpha(\gamma')}(-1)^{|J_\alpha(\delta')|}f'(\delta')\=0\pmod{\alpha^{|J_\alpha(\gamma')|}}.
\end{array}
\end{equation}
Our aim is to get rid of the second line. As the restriction $f|_{\Gamma_{yt}}$ belongs to $\X^{yt}$,
Proposition~\ref{proposition:2} implies that
$$
\sum_{\delta\in\Gamma_{yt},\delta\sim_\alpha\gamma,\delta_r=t,D_\alpha(\delta)\subset D_\alpha(\gamma)}(-1)^{|D_\alpha(\delta)|}f(\delta)\=0\pmod{\alpha^{|D_\alpha(\gamma)|}},
$$
which can be written as follows
$$
\pm yt(-\alpha_r)\sum_{\delta'\in\Gamma_y,\delta'\sim_\alpha\gamma',J_\alpha(\delta')\subset J_\alpha(\gamma')}(-1)^{|J_\alpha(\delta)|}f'(\delta')\=0\pmod{\alpha^{|D_\alpha(\gamma)|}},
$$
where $+$ is taken if $s_\alpha yt>yt$ and $-$ is taken otherwise (see Proposition~\ref{lemma:4}).
Moreover, $yt(-\alpha_r)=\gamma^r(-\alpha_r)$ is not proportional to $\alpha$ by the hypothesis of the current case.
Hence it follows from the above equivalence that
$$
c\alpha\sum_{\delta'\in\Gamma_y,\delta'\sim_\alpha\gamma',J_\alpha(\delta')\subset J_\alpha(\gamma')}(-1)^{|J_\alpha(\delta)|}f'(\delta')\=0\pmod{\alpha^{|D_\alpha(\gamma)|+1}}.
$$
It remains to note that $|J_\alpha(\gamma')|=|J_\alpha(\gamma)|\ge|D_\alpha(\gamma)|+1$, add the above equivalence to~(\ref{eq:29})
and note that $s_\alpha yt(-\alpha_r)$ is also not proportional to $\alpha$.

{\it Case 2: $r\in M_\alpha(\gamma)\setminus J_\alpha(\gamma)$.} In this case $i_\ll=r$, $\gamma\sim_\alpha\dot\gamma$ и
$J_\alpha(\dot\gamma)=J_\alpha(\gamma)\cup\{r\}$. 
By Proposition~\ref{proposition:1}, we get
$$
\sum_{\delta\in\Gamma,\delta\sim_\alpha\gamma,\delta_r=t,J_\alpha(\delta)\subset J_\alpha(\dot\gamma)}(-1)^{|J_\alpha(\delta)|}\bbeta_r(\delta)f'(\delta')\=0\pmod{\alpha^{|J_\alpha(\dot\gamma)|}}.
$$
As $r\in J_\alpha(\dot\gamma)$, this equivalence can be rewritten as follows
\begin{equation}\label{eq:30}
\sum_{\delta'\in\Gamma',\delta'\sim_\alpha\gamma',J_\alpha(\delta')\subset J_\alpha(\gamma')}(-1)^{|J_\alpha(\delta'\cdot t)|}\bbeta_r(\delta'\cdot t)f'(\delta')\=0\pmod{\alpha^{|J_\alpha(\gamma')|+1}}.
\end{equation}
Considering separately the cases $\delta'\in\Gamma'_y$ and $\delta'\in\Gamma'_{s_\alpha y}$, we get
$(-1)^{|J_\alpha(\delta'\cdot t)|}\bbeta_r(\delta'\cdot t)f'(\delta')=(-1)^{|J_\alpha(\delta')|}yt(-\alpha_r)$.
We know that $yt(-\alpha_r)=\bbeta_r(\gamma)=-\alpha$. Thus dividing~(\ref{eq:30}) by $-\alpha$,
we get~(\ref{eq:31}).

{\it Case 3: $r\in J_\alpha(\gamma)$.} In this case $i_\ll=r$ and~(\ref{eq:28}) can be rewritten as follows
\begin{equation}\label{eq:32}
\sum_{\delta'\in\Gamma',\delta'\sim_\alpha\gamma',J_\alpha(\delta')\subset J_\alpha(\gamma')}(-1)^{|J_\alpha(\delta'\cdot t)|}\bbeta_r(\delta'\cdot t)f'(\delta')\=0\pmod{\alpha^{|J_\alpha(\gamma')|+1}}.
\end{equation}
Considering separately the cases $\delta'\in\Gamma'_y$ and $\delta'\in\Gamma'_{s_\alpha y}$, we get
$(-1)^{|J_\alpha(\delta'\cdot t)|}\bbeta_r(\delta'\cdot t)f'(\delta')=-(-1)^{|J_\alpha(\delta')|}yt(-\alpha_r)$.
We know that $yt(-\alpha_r)=\bbeta_r(\gamma)=\alpha$. Thus dividing~(\ref{eq:32}) by $-\alpha$,
we get~(\ref{eq:31}).
\end{proof}

\subsection{Bases for $\X$}\label{Bases_for_X}
For any $\rho\in\Upsilon$, we construct the subset $B_\rho$ of $H^\bullet(\Sigma)$ inductively by
$$
B_\emptyset=\{1\},\qquad B_\rho=\Delta(B_{\rho'_0})\cup\nabla_{\rho_\emptyset}(B_{\rho'_1}).
$$
By Lemmas~\ref{lemma:5} and~\ref{lemma:6}, we get $B_\rho\subset\X$.

\begin{theorem}\label{theorem:4}
$B_\rho$ is an $S$-basis of $\X$.
\end{theorem}
\begin{proof} We apply induction on $r$. This result is clearly true for $r=0$.
Therefore, we assume that $r>0$ and that $B_{\rho'_0}$ and $B_{\rho'_1}$ are bases of $\X'$.

Let $f\in\X$. By Lemma~\ref{lemma:8}, the function $f'\in H^\bullet_T(\Gamma')$ defined by
$f'(\delta')=f(\delta'\cdot \rho_\emptyset s_r)$ belongs to $\X'$. We have
$$
\big(f-\Delta(f')\big)(\delta'\cdot \rho_\emptyset s_r)=f(\delta'\cdot \rho_\emptyset s_r)-\Delta(f')(\delta'\cdot \rho_\emptyset s_r)=f'(\delta')-f'(\delta')=0.
$$
for any $\delta'\in\Gamma'$. Let us define
$$
h'(\delta')=\frac{\big(f-\Delta(f')\big)(\delta'\cdot \rho_\emptyset )}{\bbeta_r(\delta'\cdot \rho_\emptyset )}.
$$
for any $\delta'\in\Gamma'$. By Lemma~\ref{lemma:9}, $h'$ is a well defined function of $\X'$. The above formulas show that
$f-\Delta(f')=\nabla_{\rho_\emptyset}(h')$, whence $f=\Delta(f')+\nabla_{\rho_\emptyset}(h')$.
By the inductive hypothesis and the linearity of $\Delta$ and $\nabla_{\rho_\emptyset}$, the function $f$ belongs to
the $S$-span of $B_\rho$.

It remains to prove the $S$-linear independence of elements of $B_\rho$. Let $B_{\rho'_0}=\{b^{(0)}_1,\ldots,b^{(0)}_{n_0}\}$
and $B_{\rho'_1}=\{b^{(1)}_1,\ldots,b^{(1)}_{n_1}\}$. Suppose that we have
$$
\sum_{i=1}^{n_0}\alpha^{(0)}_i\Delta(b^{(0)}_i)+\sum_{i=1}^{n_1}\alpha^{(1)}_i\nabla_{\rho_\emptyset}(b^{(1)}_i)=0
$$
for some $\alpha^{(0)}_i,\alpha^{(1)}_i\in S$. Consider the decomposition $\Gamma=\Gamma'\cdot\rho_\emptyset\sqcup\Gamma'\cdot\rho_\emptyset s_r$.
Restricting the above equality to $\Gamma'\cdot\rho_\emptyset s_r$, we get
$
\sum_{i=1}^{n_0}\alpha^{(0)}_ib^{(0)}_i=0.
$
Hence all $\alpha_i^{(0)}=0$ and $\sum_{i=1}^{n_1}\alpha^{(1)}_i\nabla_{\rho_\emptyset}(b^{(1)}_i)=0$.
Thus $\sum_{i=1}^{n_1}\alpha^{(1)}_i\bbeta_r(\delta)b^{(1)}_i(\delta')=0$ for any $\delta\in\Gamma'\cdot\rho_\emptyset s_r$.
Cancelling a nonzero element $\bbeta_r(\delta)$, we get $\sum_{i=1}^{n_1}\alpha^{(1)}_ib^{(1)}_i(\delta')=0$
for any $\delta'\in\Gamma'$. Hence all $\alpha^{(1)}_i=0$.
\end{proof}

\subsection{Basis for $\X^x$}
For any gallery $\gamma\in\Gamma$, we define
$$
\aa(\gamma)=\prod_{i\in D(\gamma)}\widetilde\bbeta_i(\gamma)=\prod_{\alpha\in \Phi^+}\alpha^{|D_\alpha(\gamma)|},\quad
\b_\emptyset=1,\quad
\b_\gamma=\left\{
\begin{array}{ll}
\Delta(\b_{\gamma'})&\text{if }r\notin D(\gamma);\\[3pt]
\widetilde\nabla_{\gamma_r}(\b_{\gamma'})&\text{if }r\in D(\gamma).
\end{array}
\right.
$$
By Lemma~\ref{lemma:5} and Corollary~\ref{corollary:3}, we get $\b_\gamma\in\X$.

\begin{lemma}\label{lemma:10}
Let $\gamma\in\Gamma_x$. Then $\b_\gamma(\gamma)=\aa(\gamma)$ and $\b_\gamma(\delta)=0$ for any $\delta\in\Gamma_x$
such that $\delta<\gamma$.
\end{lemma}
\begin{proof}
The first formula follows directly from the definition of $\b_\gamma$ and $\widetilde\nabla_t$. Let us prove the second claim
inductively.
From $\delta<\gamma$ it follows that there exists some $i_0=1,\ldots,r-1$ such that $\delta^{i_0}<\gamma^{i_0}$ and
$\delta^i=\gamma^i$ for $i>i_0$. Clearly $\delta'<\gamma'$.
Assume that $\delta_r=\gamma_r$. Then $\delta',\gamma'\in\Gamma_{x\gamma_r}$.
Thus by induction $\b_\gamma(\delta)=\b_{\gamma'}(\delta')=0$ if $r\notin D(\gamma)$
and $\b_\gamma(\delta)=\widetilde\bbeta_r(\delta)\b_{\gamma'}(\delta')=0$ if $r\in D(\gamma)$.
Now assume on the contrary that $\delta_r\ne\gamma_r$. However $\delta^r=\gamma^r=x$. Hence $\delta^{r-1}\ne\gamma^{r-1}$.
This means that $i_0=r-1$ and $\gamma^{r-1}s_r=\delta^{r-1}<\gamma^{r-1}$. Hence $r\in D(\gamma)$.
Therefore $\b_\gamma(\delta)=\widetilde\nabla_{\gamma_r}(\b_{\gamma'})(\delta)=0$.
\end{proof}

%

\begin{theorem}\label{theorem:2}
The set $\big\{\b_\gamma|_{\Gamma_x}\;\big|\;\gamma\in\Gamma_x\big\}$ is an $S$-basis of $\X^x$. In particular,
the restriction $\X\to\X^x$ is surjective.
\end{theorem}
\begin{proof}
This set is $S$-linearly independent by Lemma~\ref{lemma:10}, as $<$ is the total order on $\Gamma_x$ and $\aa(\gamma)\ne0$.
Let us prove that any element $f\in\X^x$ is representable as an $S$-linear combination.
We apply induction on the cardinality of the set (the upper closure of the support of $f$)
$$
C(f)=\{\delta\in\Gamma_x\suchthat\text{ there exists }\gamma\in\Gamma_x\text{ such that }\delta\ge\gamma\text{ and }f(\gamma)\ne0\}.
$$
If $C(f)=\emptyset$, then $f=0$ and the result follows. Suppose now that $C(f)\ne\emptyset$ and let $\gamma$ be its
element minimal with respect to $<$. As
\begin{equation}\label{eq:39}
\delta\in\Gamma_x,\delta\sim_\alpha\gamma,D_\alpha(\delta)\subset D_\alpha(\gamma)\Longrightarrow\delta\le\gamma,
\end{equation}
Proposition~\ref{proposition:2} implies that $f(\gamma)$ is divisible by $\prod_{\alpha\in \Phi^+}\alpha^{|D_\alpha(\gamma)|}=\aa(\gamma)$. Consider the difference $h=f-f(\gamma)/\aa(\gamma)\b_\gamma$.
By Lemma~\ref{lemma:10}, we get $C(h)\subset\{\delta\in\Gamma_x\suchthat\delta>\gamma\}\varsubsetneq C(f)$.
By induction, $h$ belongs to the $S$-span of our set. Thus so does~$f$.
\end{proof}

\begin{corollary}\label{corollary:4}
The restrictions $H^\bullet_T(\Sigma)\to H^\bullet_T(\Sigma_x)$ and $H^\bullet(\Sigma)\to H^\bullet(\Sigma_x)$
are surjective.
\end{corollary}
\begin{proof}
The surjectivity of the first morphism follows from Theorem~\ref{theorem:2}.
As $\Sigma$ and $\Sigma_x$ are equivariantly formal,
the second morphism is obtained from the first one by applying $?\otimes_S\Z'$ (where $S^i\Z'=0$ for $i>0$).
Hence it is also surjective.
\end{proof}

\begin{remark}\label{remark:1} \rm Let us define $\rho_r(x)\in\Upsilon$ as follows:
$$
x\rho_r(x)_\emptyset>x\rho_r(x)_\emptyset s_r,\quad \rho_r(x)'_0=\rho_{r-1}(x\rho_r(x)_\emptyset s_r),\quad \rho_r(x)'_1=\rho_{r-1}(x\rho_r(x)_\emptyset).
$$
It is elementary to check that up to sign all elements $\b_\gamma$ with $\gamma\in\Gamma_x$ belong to $B_{\rho_r(x)}$.
\end{remark}

Finally, we write down the exact inductive formula for the values of the basis functions
\begin{equation}\label{eq:40}
\b_\gamma(\delta)=\left\{
\begin{array}{ll}
\b_{\gamma'}(\delta')&\text{ if }r\notin D(\gamma);\\
(\delta')^{r-1}(-\alpha_r)\b_{\gamma'}(\delta')&\text{ if }r\in D(\gamma)\text{ and }\delta_r=\gamma_r;\\
0&\text{ if }r\in D(\gamma)\text{ and }\delta_r\ne\gamma_r.
\end{array}
\right.
\end{equation}

\section{Basis of the image ${\bar{\mathcal X}}^x$}\label{Basis_of_the_image_bar}

\subsection{Localization for $\bar\Sigma_x$}\label{Localization_for_Sigma_without_x} Let $k$ be a principal ideal domain
with invertible $2$ if the root system contains a component of type $C_n$.
As $\bar\Sigma_x$ is just a $T$-subspace of $\Sigma$, we can apply Theorem~\ref{theorem:1} to it as well. However, it is more
difficult to apply Corollary~\ref{corollary:1}.
Actually, the only problem to overcome is to prove that $H_T^\bullet(\bar\Sigma_x,k)$ is a free $S_k$-module.
Unfortunately, we can not solve this problem in the same way as for $\Sigma$: we do not know if $\bar\Sigma_x$
has an affine paving.

Consider the natural embeddings $i:\Sigma_x\hookrightarrow\Sigma$ and $j:\bar\Sigma_x\hookrightarrow\Sigma$.
From the non-equivariant distinguished triangle
\begin{equation}\label{eq:5}
j_!j^*\csh{k}{\Sigma}\to\csh{k}{\Sigma}\to i_*i^*\csh{k}{\Sigma}\stackrel{+1}\to,
\end{equation}
we get the following exact sequence
$$
H^{2m}(\Sigma,k)\to H^{2m}(\Sigma_x,k)\to H^{2m+1}(\Sigma,j_!\csh{k}{\bar\Sigma_x})\to H^{2m+1}(\Sigma,k)=0.
$$
The left morphism is surjective by Corollary~\ref{corollary:4} and the following corollary of the projection formula
(cf.~\cite[VI.5.1]{Iversen}):

\begin{proposition}\label{proposition:4}
Let $\Z'\to k$ be the natural ring homomorphism. For any topological space $X$,
we get the following exact sequence:
$$
0\to H_c^i(X)\otimes_{\Z'}k\to H_c^i(X,k)\to\Tor_1(H_c^{i+1}(X),k)\to0.
$$
\end{proposition}

\noindent
Hence $H^{2m+1}(\Sigma,j_!\csh{k}{\bar\Sigma_x})=0$.
Since $\Sigma$ is compact, we get
$$
0=H^{2m+1}(\Sigma,j_!\csh{k}{\bar\Sigma_x})=H_c^{2m+1}(\Sigma,j_!\csh{k}{\bar\Sigma_x})=H_c^{2m+1}(\bar\Sigma_x,k).
$$

The Poincar\'e duality in the form~\cite[Theorem 3.3.3]{Dimca} yields the following exact sequence:
$$
0{=}\Ext^1_{k\text{-mod}}(H^{2\dim\Sigma-2m+1}_c(\bar\Sigma_x,k),k)\to H^{2m}(\bar\Sigma_x,k)
\to\Hom_{k\text{-mod}}(H_c^{2\dim\Sigma-2m}(\bar\Sigma_x,k),k)\to0.
$$
Hence
\begin{equation}\label{eq:x6}
H^{2m}(\bar\Sigma_x,k)\simeq\Hom_{k\text{-mod}}(H_c^{2\dim\Sigma-2m}(\bar\Sigma_x,k),k).
\end{equation}
%
From~(\ref{eq:5}), we get the following exact sequence
$$
0=H^{2m-1}(\Sigma_x,k)\to H^{2m}(\Sigma,j_!\csh{k}{\bar\Sigma_x})\to H^{2m}(\Sigma,k).
$$
The right cohomology is a finitely generated free $k$-module. We get that its submodule
$H^{2m}(\Sigma,j_!\csh{k}{\bar\Sigma_x})=H_c^{2m}(\bar\Sigma_x,k)$
is a finitely generated free $k$-module. Hence and from~(\ref{eq:x6}), it follows that
$H^{2m}(\bar\Sigma_x,k)$ is also a finitely generated free $k$-module.

Now again applying the Poincar\'e duality in the form~\cite[Theorem 3.3.3]{Dimca}, we get
$$
0{=}\Ext^1_{k\text{-mod}}(H^{2\dim\Sigma-2m}_c(\bar\Sigma_x,k),k){\to}H^{2m+1}(\bar\Sigma_x,k){\to}\Hom_{k\text{-mod}}(H_c^{2\dim\Sigma-2m-1}(\bar\Sigma_x,k),k){=}0.
$$
Hence we get $H^{2m+1}(\bar\Sigma_x,k)=0$.

Now the degeneracy of the Leray spectral sequence at the $E_2$-term implies
$$
H_T^\bullet(\bar\Sigma_x,k)\simeq H^\bullet(\bar\Sigma_x)\otimes_kS_k.
$$
This module is therefore a free $S_k$-module.

We denote by $\OX^x(k)$ the image of the restrictions $H_T^\bullet(\OSigma_x,k)\to H_T^\bullet(\OGamma_x,k)$,
which is injective by the first part of Corollary~\ref{corollary:1}.

\subsection{Review of H\"arterich's constructions}
We shall briefly sketch H\"arterich's construction, in order to be able to apply them to the cohomology of
the difference $\bar\Sigma_x$
in Section~\ref{DescriptionofbarX} and prove the criterion (Proposition~\ref{proposition:3}) for the image $\OX^x$
of the restriction $i_{\OSigma_x,\OGamma_x}^*:H_T^\bullet(\OSigma_x)\to H_T^\bullet(\OGamma_x)$.

Let $\alpha$ be a positive root and $\gamma\in\Gamma$. We set $T_\alpha:=\ker\alpha$ and $M_\alpha(\gamma)=\{i_1<\cdots<i_\ll\}$.
H\"arterich~\cite[Section 4]{Haerterich} constructs an embedding
$v_\gamma^\alpha:(G_\alpha/B_\alpha)^\ll\hookrightarrow\Sigma$
such that composition $\iota\circ v^\alpha_\gamma:(G_\alpha/B_\alpha)^\ll\to(G/B)^r$ is given by
$$
(g_1,\ldots,g_\ll)\mapsto(\gamma_{\min}^1,\ldots \gamma_{\min}^{i_1-1},g_1\gamma_{\min}^{i_1},\ldots g_1\gamma_{\min}^{i_2-1},
\ldots,g_\ll\gamma_{\min}^{i_\ll},\ldots g_\ll\gamma_{\min}^r),
$$
where $\gamma_{\min}$ is the minimal with respect to $\vartriangleleft$ element in the $\sim_\alpha$-equivalence class of $\gamma$
(i.e. the unique element of this class having no load bearing $\alpha$-walls).
Here and in what follows we write $g$ instead of $gB_\alpha$ or $gB$ if it is clear from the context that we consider
an element of $G_\alpha/B_\alpha$ or $G/B$ respectively.
Note that for $\ll=0$, the map $v_\gamma^\alpha$ takes $(G_\alpha/B_\alpha)^\ll$ isomorphically to $\{\gamma\}$.

Clearly, $v^\alpha_\gamma$ depends only on the $\sim_\alpha$-equivalence class of $\gamma$.
Corollary 4.4 from~\cite{Haerterich} states
\begin{equation}\label{eq:10}
\Sigma^{T_\alpha}=\bigsqcup_{\gamma\in\rep(\Gamma,\sim_\alpha)}\im v^\alpha_\gamma.
\end{equation}
We note that to ensure the $T$-equivariance of $v_\gamma^\alpha$, we must define an appropriate $T$-action on
$(G_\alpha/B_\alpha)^\ll$. This can be done as follows:
$$
t\cdot(g_1,\ldots,g_\ll)=(tg_1t^{-1},\ldots,tg_\ll t^{-1}).
$$
Similarly the isomorphism $\iota:\Sigma^2_\ll\ito G_\alpha/B_\alpha$ becomes an isomorphism of $T$-spaces if we define the following
$T$-action on $\Sigma^2_\ll$:
$$
t\cdot[p_1,\ldots,p_\ll]:=[tp_1t^{-1},\ldots,tp_\ll t^{-1}].
$$

It is clear that $(\im v^\alpha_\gamma)^T=\{\delta\in\Gamma\suchthat\delta\sim_\alpha\gamma\}$. We can compute the preimage
of each point of the last set with respect to the map $v_\alpha^\gamma\circ\iota:\Sigma_\ll^2\to\Sigma$.
Indeed, without loss of generality, it suffices to compute $(v_\alpha^\gamma\circ\iota)^{-1}(\gamma)$.
Let us define
$$
g_j=\left\{
\begin{array}{ll}
s_\alpha&\text{ if }i_j\in J_\alpha(\gamma);\\[3pt]
e&\text{ otherwise}.
\end{array}\right.
$$
This definition and~\cite[Remark preceeding (4.1)]{Haerterich} ensures that
$v^\alpha_\gamma(g_1,\ldots,g_\ll)=\gamma$. We define $\bar g:=[g_1,g_1^{-1}g_2,\ldots,g_{\ll-1}^{-1}g_\ll]$ as an element of $(G_\alpha/B_\alpha)^\ll$.
Hence $v_\alpha^\gamma\circ\iota(\bar g)=\gamma$. The following equivalence
$$
j\in J(\bar g)\Leftrightarrow \bar g^js_\alpha<\bar g^j\Leftrightarrow g_j=\bar g^j=s_\alpha \Leftrightarrow i_j\in J_\alpha(\gamma)
$$
proves that
\begin{equation}\label{eq:11}
i_{J(\delta)}=J_\alpha(v_\gamma^\alpha\circ\iota(\delta))
\end{equation}
for any $\delta\in\Gamma^2_\ll$.

Now we are going to explain how to compute the intersection $\im v_\gamma^\alpha\cap\pi^{-1}(x)$.
Suppose that $\ll>0$ and $v^\alpha_\gamma(g_1,\ldots,g_\ll)\in\pi^{-1}(x)$.
Consider the following commutative diagram:
$$
\begin{tikzcd}
(G_\alpha/B_\alpha)^\ell\arrow{r}{v^\alpha_\gamma}&\Sigma\arrow{dr}[swap]{\pi}\arrow{r}{\iota}&(G/B)^r\arrow{d}{\pr_r}\\
&&G/B
\end{tikzcd}
$$
Recalling the definition of $\iota\circ v_\gamma^\alpha$ (see~H\"arterich~\cite[(4.1)]{Haerterich}), we get
\begin{equation}\label{eq:8}
g_\ll\gamma_{\min}^rB=xB.
\end{equation}

If $g_\ll\in B_\alpha$, then~(\ref{eq:8}) together with the Bruhat decomposition yields $x=\gamma_{\min}^r$.
On the contrary, if the last equality holds, then~(\ref{eq:8}) is true for any $g_\ll\in B_\alpha$.

If $g_\ll\in U_\alpha s_\alpha B_\alpha$, then $g_\ll=x_\alpha(c)s_\alpha b$ for some $c\in\C$ and $b\in B_\alpha$.
Applying~\cite[(4.2)]{Haerterich}, we can rewrite~(\ref{eq:8}) as follows
\begin{equation}\label{eq:9}
g_\ll\gamma_{\min}^rB=x_\alpha(c)s_\alpha b\gamma_{\min}^rB=x_\alpha(c)s_\alpha\gamma_{\min}^rB.
\end{equation}
This representation is already canonical in the sense of~\cite[1.13]{Jantzen}, as
$$
(s_\alpha\gamma_{\min}^r)^{-1}(\alpha)=(\gamma_{\min}^r)^{-1}s_\alpha(\alpha)=(\gamma_{\min}^r)^{-1}(-\alpha)=-(\gamma_{\min}^r)^{-1}(\alpha)<0.
$$
Now comparing~(\ref{eq:9}) with~(\ref{eq:8}) by the Bruhat decomposition, we get $x=s_\alpha\gamma_{\min}^r$ and $c=0$.
This analysis proves the following formulas:
\begin{equation}\label{eq:41}
\im v_\gamma^\alpha\cap\pi^{-1}(x)=\left\{
\begin{array}{ll}
\emptyset&\text{ if }\gamma\notin\Gamma_{\{x,s_\alpha x\}};\\[6pt]
v^\alpha_\gamma((G_\alpha/B_\alpha)^{\ll-1}\times\{e\})&\text{ if }x=\pi(\gamma_{\min});\\[6pt]
v^\alpha_\gamma((G_\alpha/B_\alpha)^{\ll-1}\times\{s_\alpha\})&\text{ if }x=s_\alpha\pi(\gamma_{\min})
\end{array}
\right.
\end{equation}
if $\ll>0$ and
\begin{equation}\label{eq:45}
\im v_\gamma^\alpha\cap\pi^{-1}(x)=
\left\{
\begin{array}{ll}
\emptyset&\text{ if }\gamma\notin\Gamma_x;\\[6pt]
\{\gamma\}&\text{ if }\gamma\in\Gamma_x
\end{array}
\right.
\end{equation}
if $\ll=0$.

\subsection{Description of ${\bar{\mathcal X}}^x$}\label{DescriptionofbarX}
We will prove the following analog of Propositions~\ref{proposition:1} and~\ref{proposition:2}.

\begin{proposition}\label{proposition:3}
An element $f\in H_T^\bullet(\OGamma_x)$ belongs to the image $\OX^x$ of the restriction\linebreak
$i_{\OSigma_x,\OGamma_x}^*:H_T^\bullet(\OSigma_x)\to H_T^\bullet(\OGamma_x)$ if and only if
\begin{equation}\label{eq:42}
\sum_{\delta\in\Gamma,\delta\sim_\alpha\gamma,J_\alpha(\delta)\subset J_\alpha(\gamma)}(-1)^{|J_\alpha(\delta)|}f(\delta)\=0\pmod{\alpha^{|J_\alpha(\gamma)|}}
\end{equation}
for any positive root $\alpha$ and gallery $\gamma\in\OGamma_{\{x,s_\alpha x\}}$ and
\begin{equation}\label{eq:43}
\sum_{\delta\in\Gamma_{s_\alpha x},\delta\sim_\alpha\gamma,D_\alpha(\delta)\subset D_\alpha(\gamma)}(-1)^{|D_\alpha(\delta)|}f(\delta)\=0\pmod{\alpha^{|D_\alpha(\gamma)|}}
\end{equation}
for any positive root $\alpha$ and gallery $\gamma\in\Gamma_{s_\alpha x}$.
\end{proposition}
\begin{proof}
Subtracting $\pi^{-1}(x)$ from~(\ref{eq:10}), we get
$$
\OSigma_x^{T_\alpha}=\Sigma^{T_\alpha}\setminus\pi^{-1}(x)=\bigsqcup_{\gamma\in\rep(\Gamma,\sim_\alpha)}\im v^\alpha_\gamma\setminus\pi^{-1}(x).
$$
Note that $(\im v^\alpha_\gamma\setminus\pi^{-1}(x))^T=\{\delta\in\OGamma_x\suchthat\delta\sim_\alpha\gamma\}$.

By Corollary~\ref{corollary:2}, an element $f\in H^\bullet_T(\OGamma_x)$ belongs to the image of
$H^\bullet_T(\OSigma_x^{T_\alpha})\to H^\bullet_T(\OGamma_x)$ if and only if each restriction
$f^\gamma:=f|_{\{\delta\in\OGamma_x\suchthat\delta\sim_\alpha\gamma\}}$ belongs to the image of
$H^\bullet_T(\im v^\alpha_\gamma\setminus\pi^{-1}(x))\to H^\bullet_T(\{\delta\in\OGamma_x\suchthat\delta\sim_\alpha\gamma\})$.
Clearly, it suffices to consider only the case $\{\delta\in\OGamma_x\suchthat\delta\sim_\alpha\gamma\}\ne\emptyset$.
We fix such a $\gamma\in\Gamma$ and consider the set $M_\alpha(\gamma)=\{i_1<\cdots<i_\ll\}$.

{\it Case~1: $\gamma\in\Gamma_{\{x,s_\alpha x\}}$}. By~(\ref{eq:41}) or~(\ref{eq:45}), we get that
$\im v^\alpha_\gamma\setminus\pi^{-1}(x)=\im v^\alpha_\gamma$
and\linebreak $\{\delta\in\OGamma_x\suchthat\delta\sim_\alpha\gamma\}=\{\delta\in\Gamma\suchthat\delta\sim_\alpha\gamma\}$.
We have the commutative diagram
$$
\begin{tikzcd}
\Sigma^2_\ll\arrow{r}{v_\gamma^\alpha\circ\iota}[swap]{\sim}&\im v^\alpha_\gamma&\\
\Gamma^2_\ll\arrow{r}{v_\gamma^\alpha\circ\iota}[swap]{\sim}\arrow[hook]{u}&(\im v^\alpha_\gamma)^T\arrow[hook]{u}\arrow[-, double equal sign distance]{r}&\{\delta\in\Gamma\suchthat\delta\sim_\alpha\gamma\}
\end{tikzcd}
$$

$$
\begin{tikzcd}
\Sigma^2_\ll\arrow{r}{\iota}[swap]{\sim}&(G_\alpha/B_\alpha)^\ll\arrow{r}{v_\gamma^\alpha}&\im v^\alpha_\gamma&\\
\Gamma^2_\ll\arrow{r}{\iota}[swap]{\sim}\arrow[hook]{u}&\{e,s_\alpha\}^\ll\arrow{r}{v_\gamma^\alpha}\arrow[hook]{u}&(\im v^\alpha_\gamma)^T\arrow[hook]{u}\arrow[-, double equal sign distance]{r}&\{\delta\in\Gamma\suchthat\delta\sim_\alpha\gamma\}
\end{tikzcd}
$$

Hence we get the following commutative diagram for cohomologies:
$$
\begin{tikzcd}[column sep=10ex]
H^\bullet_T(\Sigma^2_\ll)\arrow[hook]{d}&\arrow{l}[swap]{(v^\alpha_\gamma\circ\iota)^*}{\sim}H^\bullet_T(\im v^\alpha_\gamma)\arrow[hook]{d}\\
H^\bullet_T(\Gamma^2_\ll)&\arrow{l}[swap]{(v^\alpha_\gamma\circ\iota)^*}{\sim}H^\bullet_T(\{\delta\in\Gamma\suchthat\delta\sim_\alpha\gamma\})
\end{tikzcd}
$$
Thus $f^\gamma$ belongs to the image of the right vertical arrow if and only if $f\circ v^\alpha_\gamma\circ\iota$
belongs to the image of the left vertical arrow. By~\cite[Proposition 5.4(a)]{Haerterich} this is equivalent to
\begin{equation}\label{eq:12}
\sum_{\delta\in\Gamma^2_\ll,\;J(\delta)\subset J(\tau)}(-1)^{|J(\delta)|}f\circ v_\gamma^\alpha\circ\iota(\delta)\=0\pmod{\alpha^{|J(\tau)|}}
\end{equation}
for any $\tau\in\Gamma^2_\ll$.
By~(\ref{eq:11}), we have the following equivalences
$$
|J(\delta)|=|J_\alpha(v_\gamma^\alpha\circ\iota(\delta))|,\qquad J(\delta)\subset J(\tau)\Leftrightarrow J_\alpha(v_\gamma^\alpha\circ\iota(\delta))\subset J_\alpha(v_\gamma^\alpha\circ\iota(\tau)).
$$
So~(\ref{eq:12}) can be rewritten as follows
$$
\sum_{\delta\in\Gamma_\ll^2,\;J_\alpha(v_\gamma^\alpha\circ\iota(\delta))\subset J_\alpha(v_\gamma^\alpha\circ\iota(\tau))}(-1)^{|J_\alpha(v_\gamma^\alpha\circ\iota(\delta))|}f\circ v_\gamma^\alpha\circ\iota(\delta)\=0\pmod{\alpha^{|J_\alpha(v_\gamma^\alpha\circ\iota(\tau))|}}.
$$
Replacing $v_\gamma^\alpha\circ\iota(\delta)$ and $v_\gamma^\alpha\circ\iota(\tau)$ with $\delta$ and $\gamma$ respectively,
we get the final version~(\ref{eq:42}).

{\it Case~2: $x=\pi(\gamma_{\min})$.} The condition $\{\delta\in\OGamma_x\suchthat\delta\sim_\alpha\gamma\}\ne\emptyset$
implies $\ll>0$.
Let $Y=(G_\alpha/B_\alpha)^{\ll-1}$ and $Z=\{e,s_\alpha\}^{\ll-1}$ be its set
of $T$-fixed points. By~(\ref{eq:41}), we get
$$
\im v^\alpha_\gamma\setminus\pi^{-1}(x)=v^\alpha_\gamma\big(Y\times((G_\alpha/B_\alpha)\setminus\{e\})\big),
$$
Consider the following maps:
\begin{itemize}
\item the natural inclusions $\tilde\imath:\Gamma_{\ll-1}^2\hookrightarrow\Sigma_{\ll-1}^2$, $i_{Y,Z}:Z \hookrightarrow Y$ and $\hat\imath:\{s_\alpha\}\hookrightarrow(G_\alpha/B_\alpha)\setminus\{e\}$;\\[-6pt]
\item $a:Y\to Y\times((G_\alpha/B_\alpha)\setminus\{e\})$ and $b:Z\to Z\times\{s_\alpha\}$ that add the point
      $s_\alpha$ to the last position;\\[-6pt]
\item the projection $p:Y\times ((G_\alpha/B_\alpha)\setminus\{e\})\to Y$ to the first $\ll-1$ coordinates.
\end{itemize}

By definition $p\circ a=\id$. Thus $a^*\circ p^*=\id$ on the level of cohomology. In particular,\linebreak $a^*$ is surjective.
We have the following commutative diagram:
$$
\begin{tikzcd}[column sep=4.5ex]
\Sigma_{\ll-1}^2\arrow{r}{\iota}[swap]{\sim}&Y\arrow{r}{a}&Y\times((G_\alpha/B_\alpha)\setminus\{e\})\arrow{r}{v_\gamma^\alpha}[swap]{\sim}&\im v^\alpha_\gamma\setminus\pi^{-1}(x)&\\
\Gamma_{\ll-1}^2\arrow[hook]{u}{\tilde\imath}\arrow{r}{\iota}[swap]{\sim}&Z\arrow[hook]{u}{i_{Y,Z}}\arrow{r}{b}[swap]{\sim}&Z\times\{s_\alpha\}\arrow{r}{v_\gamma^\alpha}[swap]{\sim}\arrow[hook]{u}{i_{Y,Z}\times\hat\imath}&(\im v^\alpha_\gamma\setminus\pi^{-1}(x))^T\arrow[hook]{u}\arrow[-, double equal sign distance]{r}&\{\delta\in\Gamma_{s_\alpha x}\suchthat\delta\sim_\alpha\gamma\}
\end{tikzcd}
$$
Hence we get the following commutative diagram for cohomologies:
$$
\begin{tikzcd}[column sep=5.5ex]
H^\bullet_T(\Sigma_{\ll-1}^2)\arrow{d}[swap]{\tilde\imath^*}&\arrow{l}[swap]{\iota^*}{\sim}H^\bullet_T(Y)\arrow{d}[swap]{i_{Y,Z}^*}&\arrow{l}[swap]{a^*}H^\bullet_T\big(Y\times((G_\alpha/B_\alpha)\setminus\{e\})\big)\arrow{d}[swap]{(i_{Y,Z}\times\hat\imath)^*}&\arrow{l}[swap]{(v_\gamma^\alpha)^*}{\sim}H^\bullet_T(\im v^\alpha_\gamma\setminus\pi^{-1}(x)\arrow{d})\\
H^\bullet_T(\Gamma_{\ll-1}^2)&\arrow{l}[swap]{\iota^*}{\sim}H^\bullet_T(Z)&\arrow{l}[swap]{b^*}{\sim}H^\bullet_T(Z\times\{s_\alpha\})&\arrow{l}[swap]{(v_\gamma^\alpha)^*}{\sim}H^\bullet_T(\{\delta\in\Gamma_{s_\alpha x}\suchthat\delta\sim_\alpha\gamma\})
\end{tikzcd}
$$
The surjectivity of $a^*$ proves that $b^*$ maps isomorphically $\im(i_{Y,Z}\times\hat\imath)^*$ onto $\im i_{Y,Z}^*$. Therefore the same is true about the
whole bottom line of the above diagram: the image of the rightmost vertical arrow is mapped isomorphically onto $\im\tilde\imath^*$.
Thus $f^\gamma$ belongs to the image of the rightmost vertical arrow if and only if
$f\circ v_\alpha^\gamma\circ b\circ\iota\in\im\tilde\imath^*$. By~\cite[Proposition 5.4(a)]{Haerterich}
this is equivalent to
\begin{equation}\label{eq:13}
\sum_{\delta\in\Gamma_{\ll-1}^2,\;J(\delta)\subset J(\tau)}(-1)^{|J(\delta)|}f\circ v^\alpha_\gamma\circ b\circ\iota(\delta)\=0\pmod{\alpha^{|J(\tau)|}}.
\end{equation}
for any $\tau\in\Gamma_{\ll-1}^2$.
Consider the commutative diagram
$$
\begin{tikzcd}
\Gamma_{\ll-1}^2\arrow{r}{\iota}\arrow{d}[swap]{q}&Z\arrow{d}{b}\\
\Gamma_\ll^2\arrow{r}[swap]{\iota}&Z\times\{s_\alpha\}
\end{tikzcd}
$$
where $q(\delta_1,\ldots,\delta_{\ll-1})=(\delta_1,\ldots,\delta_{\ll-1},\delta_{\ll-1}\cdots\delta_2\delta_1s_\alpha)$.
Note that $v^\alpha_\gamma\circ b\circ\iota(\delta)$ runs over the set
$\{\delta\in\Gamma_{s_\alpha x}\suchthat\delta\sim_\alpha\gamma\}$ as $\delta$ runs over $\Gamma_{\ll-1}^2$.
By~(\ref{eq:11}), we get
$$
J_\alpha(v^\alpha_\gamma\circ b\circ\iota(\delta))=J_\alpha(v^\alpha_\gamma\circ\iota\circ q(\delta))=i_{J(q(\delta))}=i_{J(\delta)}\sqcup\{i_\ll\}.
$$
Hence
$$
|D_\alpha(v^\alpha_\gamma\circ b\circ\iota(\delta))|=|J_\alpha(v^\alpha_\gamma\circ b\circ\iota(\delta))\setminus\{i_\ll\}|=|i_{J(\delta)}|=|J(\delta)|.
$$
By~(\ref{eq:11}) and Proposition~\ref{lemma:4}, the inclusion relation is also preserved:
$$
J(\delta)\subset J(\tau)\Leftrightarrow
J_\alpha(v^\alpha_\gamma\circ b\circ\iota(\delta))\subset J_\alpha(v^\alpha_\gamma\circ b\circ\iota(\tau))\Leftrightarrow
D_\alpha(v^\alpha_\gamma\circ b\circ\iota(\delta))\subset D_\alpha(v^\alpha_\gamma\circ b\circ\iota(\tau)).
$$

Now we can replace~(\ref{eq:13}) with
$$
\sum_{\delta\in\Gamma_{\ll-1}^2,\;D_\alpha(v^\alpha_\gamma\circ b\circ\iota(\delta))\subset D_\alpha(v^\alpha_\gamma\circ b\circ\iota(\tau))}(-1)^{|D_\alpha(v^\alpha_\gamma\circ b\circ\iota(\delta))|}f\circ v^\alpha_\gamma\circ b\circ\iota(\delta)\=0\pmod{\alpha^{|D_\alpha(v^\alpha_\gamma\circ b\circ\iota(\tau))|}}.
$$
Replacing $v^\alpha_\gamma\circ b\circ\iota(\delta)$ and $v^\alpha_\gamma\circ b\circ\iota(\tau)$
with $\delta$ and $\gamma$ respectively, we get~(\ref{eq:43}).

{\it Case~3: $x=s_\alpha\pi(\gamma_{\min})$.} For $\ll>0$, this case can be obtained from Case~2 by interchanging $e$ and $s_\alpha$.
We get the following version of~(\ref{eq:13}):
\begin{equation}\label{eq:15}
\sum_{\delta\in\Gamma_{\ll-1}^2,\;J(\delta)\subset J(\tau)}(-1)^{|J(\delta)|}f\circ v^\alpha_\gamma\circ b'\circ\iota(\delta)\=0\pmod{\alpha^{|J(\tau)|}}.
\end{equation}
where $b':Z\to Z\times\{e\}$ adds the point $e$ to the last position.
We have a similar commutative diagram
$$
\begin{tikzcd}
\Gamma_{\ll-1}^2\arrow{r}{\iota}\arrow{d}[swap]{q'}&Z\arrow{d}{b'}\\
\Gamma_\ll^2\arrow{r}[swap]{\iota}&Z\times\{e\}
\end{tikzcd}
$$
where $q'(\delta_1,\ldots,\delta_{\ll-1})=(\delta_1,\ldots,\delta_{\ll-1},\delta_{\ll-1}\cdots\delta_2\delta_1)$.
Here again $v^\alpha_\gamma\circ b\circ\iota(\delta)$ runs over the set
$\{\delta\in\Gamma_{s_\alpha x}\suchthat\delta\sim_\alpha\gamma\}$ as $\delta$ runs over $\Gamma_{\ll-1}^2$.
By~(\ref{eq:11}), we get
$$
J_\alpha(v^\alpha_\gamma\circ b'\circ\iota(\delta))=J_\alpha(v^\alpha_\gamma\circ\iota\circ q'(\delta))=i_{J(q'(\delta))}=i_{J(\delta)}.
$$
Hence
$$
|D_\alpha(v^\alpha_\gamma\circ b'\circ\iota(\delta))|=|J_\alpha(v^\alpha_\gamma\circ b'\circ\iota(\delta))\setminus\{i_\ll\}|=|i_{J(\delta)}|=|J(\delta)|
$$
and the same is true for the inclusion
$$
J(\delta)\subset J(\tau)\Leftrightarrow
J_\alpha(v^\alpha_\gamma\circ b'\circ\iota(\delta))\subset J_\alpha(v^\alpha_\gamma\circ b'\circ\iota(\tau))\Leftrightarrow
D_\alpha(v^\alpha_\gamma\circ b'\circ\iota(\delta))\subset D_\alpha(v^\alpha_\gamma\circ b'\circ\iota(\tau)).
$$
Therefore, we again get~(\ref{eq:43}).

Finally, assume that $\ll=0$. In this case, $\gamma=\gamma_{\min}$ and $\im v^\alpha_\gamma\setminus\pi^{-1}(x)=\{\gamma\}$.
Hence the restriction $H^\bullet_T(\im v^\alpha_\gamma\setminus\pi^{-1}(x))\to H^\bullet_T(\{\delta\in\OGamma_x\suchthat\delta\sim_\alpha\gamma\})$
is the identity map. Thus $f^\gamma$ is always in its image. On the other hand, condition~(\ref{eq:43}) is also satisfied,
as $D_\alpha(\gamma)=\emptyset$.


All the cases being considered, it suffices to apply Corollary~\ref{corollary:1} in order to conclude the proof.
\end{proof}


\subsection{Versions of Lemmas~\ref{lemma:7},~\ref{lemma:8}, and~\ref{lemma:9} for $\bar\Sigma_x$}
First, we get the following results similar to the first two lemmas.

\begin{lemma}[cf. Lemma~\ref{lemma:7}]\label{lemma:13}
$\dot{{\bar\X}}^x=\bar\X^{xs_r}$.
\end{lemma}
\begin{proof} Take any $f\in\bar\X^x$. By Proposition~\ref{proposition:3}, in order to prove that $\dot f\in\bar\X^{xs_r}$,
we must check the following equivalences:
$$
\sum_{\delta\in\Gamma,\delta\sim_\alpha\gamma,J_\alpha(\delta)\subset J_\alpha(\gamma)}(-1)^{|J_\alpha(\delta)|}f(\dot\delta)\=0\pmod{\alpha^{|J_\alpha(\gamma)|}}
$$
for $\gamma\in\OGamma_{\{xs_r,s_\alpha xs_r\}}$ and
$$
\sum_{\delta\in\Gamma_{s_\alpha xs_r},\delta\sim_\alpha\gamma,D_\alpha(\delta)\subset D_\alpha(\gamma)}(-1)^{|D_\alpha(\delta)|}f(\dot\delta)\=0\pmod{\alpha^{|D_\alpha(\gamma)|}}
$$
for $\gamma\in\Gamma_{s_\alpha xs_r}$.
The first equivalence can be proved exactly as in Lemma~\ref{lemma:7}. Note that $\gamma,\dot\gamma\in\OGamma_{\{x,s_\alpha x\}}$
if $r\in M_\alpha(\gamma)$ (for Case~2). In view of the properties listed in Section~\ref{Folding the ends},
the second equivalence can be rewritten in the form
$$
\sum_{\dot\delta\in\Gamma_{s_\alpha x},\dot\delta\sim_\alpha\dot\gamma,D_\alpha(\dot\delta)\subset D_\alpha(\dot\gamma)}(-1)^{|D_\alpha(\dot\delta)|}f(\dot\delta)\=0\pmod{\alpha^{|D_\alpha(\dot\gamma)|}}.
$$
It holds by Proposition~\ref{proposition:3}.
\end{proof}

\begin{lemma}[cf. Lemma~\ref{lemma:8}]\label{lemma:14}
Let $f\in\bar\X^x$, $r>0$ and $t\in\{e,s_r\}$. We define $f'\in H^\bullet_T(\bar\Gamma'_{xt})$ by $f'(\gamma')=f(\gamma'\cdot t)$.
Then $f'\in(\bar\X')^{xt}$.
\end{lemma}
\begin{proof} The same embedding $\iota$ as in Section~\ref{Fixing the ends} induces the commutative diagram
$$
\begin{tikzcd}
H^\bullet_T(\bar\Sigma_x)\arrow{d}\arrow{r}&H^\bullet_T(\bar\Sigma'_x)\arrow{d}\\
H^\bullet_T(\bar\Gamma_x)\arrow{r}&H^\bullet_T(\bar\Gamma'_x)
\end{tikzcd}
$$
Therefore the lemma holds for $t=e$. In order to prove it for $t=s_r$, consider $\dot f$ and apply
Lemma~\ref{lemma:13}.
\end{proof}

The version of Lemma~\ref{lemma:9} requires however a more accurate choice of $t$.

\begin{lemma}[cf. Lemma~\ref{lemma:9}]\label{lemma:12}
Let $x\in W$, $f\in\bar\X^x$ and $r>0$. We can choose a unique $t\in\{e,s_r\}$ such that $xt<xts_r$.
Suppose that $f(\gamma)=0$ for $\gamma_r\ne t$.
Then $f(\gamma)$ is divisible by $\bbeta_r(\gamma)$ for any $\gamma\in\bar\Gamma_x$.
Moreover, the function $\gamma'\mapsto f(\gamma'\cdot t)/\bbeta_r(\gamma'\cdot t)$, where $\gamma'\in{\bar\Gamma}'_{xt}$,
belongs to $({\bar\X}')^{xt}$.
\end{lemma}
\begin{proof}
First, we show how to choose $t$. Let us choose $t\in\{e,s_r\}$ arbitrarily. The elements $xt$ and $xts_r$
are comparable with respect to the Bruhat order, as they differ by a reflection.
The element $t$ is already chosen if $xt<xts_r$. Suppose that $xt>xts_r$. The we set $t'=ts_r$ and get
$xt'=xts_r<xt=xt's_r$.
The uniqueness is clear from $xt<xts_r\Leftrightarrow xt'>xt's_r$ with $t'$ as before.

As in the proof of Lemma~\ref{lemma:9}, we shall prove the divisibility claim by induction with respect to $\trianglelefteq$.
So suppose that $f(\delta)$ is divisible by $\bbeta_r(\delta)$ for any $\delta\in\bar\Gamma_x$ such that
$\delta\vartriangleleft\gamma$ for some $\gamma\in\bar\Gamma_x$.
We must prove that $f(\gamma)$ is divisible by $\bbeta_r(\gamma)$. Clearly, we need only to consider the case $\gamma_r=t$.

We take for $\alpha$ the positive of the two roots $\bbeta_r(\gamma)$ and $-\bbeta_r(\gamma)$.
Thus $r\in M_\alpha(\gamma)$. Note the following chain of equivalences:
\begin{equation}\label{eq:33}
\gamma\in\Gamma_{s_\alpha x}\Leftrightarrow\gamma^r=s_\alpha x=\gamma^rs_r(\gamma^r)^{-1}x\Leftrightarrow\gamma^rs_r=x\Leftrightarrow\gamma^r=xs_r\Leftrightarrow\gamma\in\Gamma_{xs_r}.
\end{equation}
Similarly, we get $\dot\gamma\in\Gamma_{s_\alpha x}\Leftrightarrow\dot\gamma\in\Gamma_{xs_r}$.

The case $\gamma\in\bar\Gamma_{\{x,s_\alpha x\}}$ is identical to Cases~1 and~2 of Lemma~\ref{lemma:9},
where one applies Proposition~\ref{proposition:3} instead of Proposition~\ref{proposition:1}.
Note that in this case $\dot\gamma\in\bar\Gamma_{\{x,s_\alpha x\}}$ by~(\ref{eq:33}),
Case~1 corresponds to $r\in J_\alpha(\gamma)$ and Case~2 to $r\notin J_\alpha(\gamma)$.

Consider the case $\gamma\in\Gamma_{s_\alpha x}$. By Proposition~\ref{proposition:3}, we get
$$
\sum_{\delta\in\Gamma_{s_\alpha x},\delta\sim_\alpha\gamma,D_\alpha(\delta)\subset D_\alpha(\gamma)}(-1)^{|D_{\alpha}(\delta)|}f(\delta)\=0\pmod{\alpha^{|D_\alpha(\gamma)|}}.
$$
We have $\pi(\delta)=\pi(\gamma)$ in the summation. It follows from this fact and Proposition~\ref{lemma:4}
that $J_\alpha(\delta)\subset J_\alpha(\gamma)$. Hence $\delta\trianglelefteq\gamma$.

It remains to check that $D_\alpha(\gamma)\ne\emptyset$. By~(\ref{eq:33}), we get $\gamma\in\Gamma_{xs_r}$.
Thus $\gamma^{r-1}=\gamma^rt=xs_rt$, whence our condition $xt<xts_r$ implies $\gamma^{r-1}s_r<\gamma^{r-1}$ and
$r\in D_\alpha(\gamma)$.

Let us prove the last claim. We denote by $f'$ the function under consideration:
$f'(\gamma')=f(\gamma'\cdot t)/\bbeta_r(\gamma'\cdot t)$.
By Proposition~\ref{proposition:3}, we must check the equivalence
$$
\sum_{\delta'\in\Gamma',\delta'\sim_\alpha\gamma',J_\alpha(\delta')\subset J_\alpha(\gamma')}(-1)^{|J_\alpha(\delta')|}f'(\delta')\=0\pmod{\alpha^{|J_\alpha(\gamma')|}}
$$
for any $\gamma'\in\bar\Gamma'_{\{xt,s_\alpha xt\}}$ and the equivalence
\begin{equation}\label{eq:34}
\sum_{\delta'\in\Gamma'_{s_\alpha xt},\delta'\sim_\alpha\gamma',D_\alpha(\delta')\subset D_\alpha(\gamma')}(-1)^{|D_\alpha(\delta')|}f'(\delta')\=0\pmod{\alpha^{|D_\alpha(\gamma')|}}
\end{equation}
for any $\gamma'\in\bar\Gamma'_{s_\alpha xt}$. The first one can be proved exactly as in Lemma~\ref{lemma:9},
as $\gamma\in\bar\Gamma_{\{x,s_\alpha x\}}$ and $\dot\gamma\in\bar\Gamma_{\{x,s_\alpha x\}}$ if $r\in M_\alpha(\gamma)$ (Case~2),
where $\gamma=\gamma'\cdot t$.

It remains to prove~(\ref{eq:34}). In this case, $\gamma=\gamma'\cdot t\in\Gamma_{s_\alpha x}$. We consider the following cases.

{\it Case a: $r\notin M_\alpha(\gamma)$.} In this case $\bbeta_t(\gamma)\ne\pm\alpha$. Hence
$s_\alpha x(-\alpha_r)\ne\pm\alpha$. By Proposition~\ref{proposition:3}, we get
\begin{equation}\label{eq:35}
\sum_{\delta\in\Gamma_{s_\alpha x},\delta\sim_\alpha\gamma,\delta_r=t,D_\alpha(\delta)\subset D_\alpha(\gamma)}(-1)^{|D_\alpha(\delta)|}f(\delta)\=0\pmod{\alpha^{|D_\alpha(\gamma)|}}.
\end{equation}
As $\delta_r=\gamma_r=t$ in this summation, we can rewrite the above equivalence as follows
$$
s_\alpha x(-\alpha_r)\sum_{\delta'\in\Gamma_{s_\alpha xt},\delta'\sim_\alpha\gamma',D_\alpha(\delta')\subset D_\alpha(\gamma')}(-1)^{|D_\alpha(\delta')|}f'(\delta')\=0\pmod{\alpha^{|D_\alpha(\gamma')|}}.
$$
Cancelling out $s_\alpha x(-\alpha_r)$, we get~(\ref{eq:34}).

{\it Case b: $r\in M_\alpha(\gamma)$.} In this case, $xs_r=s_\alpha x$. For any $\delta\in\Gamma_{s_\alpha x}$
such that $\delta_r=t$ and $\delta\sim_\alpha\gamma$, we have
$\delta^{r-1}s_r=s_\alpha x t s_r=xt<xts_r=s_\alpha x t=\delta^{r-1}$.
Hence $r\in D_\alpha(\delta)$. Therefore we can rewrite~(\ref{eq:35}) as follows
$$
\pm\alpha\sum_{\delta'\in\Gamma_{s_\alpha xt},\delta'\sim_\alpha\gamma',D_\alpha(\delta')\subset D_\alpha(\gamma')}(-1)^{|D_\alpha(\delta')|}f'(\delta')\=0\pmod{\alpha^{|D_\alpha(\gamma')|+1}}
$$
Cancelling out $\pm\alpha$, we get~(\ref{eq:34}).
\end{proof}

\subsection{Basis for ${\bar{\mathcal X}}^x$}\label{Basis_for_barX}
For any gallery $\gamma\in\Gamma$ and $x\in W$, we define
$$
\c_\emptyset^x=1,\quad
\c_\gamma^x=\left\{
\begin{array}{ll}
\Delta(\c_{\gamma'}^{x\gamma_r})&\text{if }x\gamma_r>x\gamma_rs_r;\\[3pt]
\nabla_{\gamma_r}(\c_{\gamma'}^{x\gamma_r})&\text{if }x\gamma_r<x\gamma_rs_r.
\end{array}
\right.
$$
By Lemmas~\ref{lemma:5} and~\ref{lemma:6}, we get $\c_\gamma^x\in\X$.

\begin{theorem}\label{theorem:3}
The set $\big\{\c_\gamma^x|_{\bar\Gamma_x}\;\big|\;\gamma\in\bar\Gamma_x\big\}$ is an $S$-basis of $\bar\X^x$. In particular,
the restrictions $\X\to\bar\X^x$ and  $H_T^\bullet(\Sigma)\to H_T^\bullet(\bar\Sigma_x)$ are
surjective.
\end{theorem}
\begin{proof} We apply induction on $r$, the result being obvious for $r=0$.
Now let $r>0$ and $f$ be an element of $\bar\X^x$. Choose $q\in\{e,s_r\}$ so that $xq>xqs_r$
and define $f'(\gamma')=f(\gamma'\cdot q)$ for $\gamma'\in\bar\Gamma'_{xq}$.
By Lemma~\ref{lemma:14}, we get $f'\in(\bar\X')^{xq}$.
By the inductive hypothesis,
$f'=\sum_{\gamma'\in\bar\Gamma'_{xq}}a_{\gamma'}\c^{xq}_{\gamma'}|_{\bar\Gamma'_{xq}}$
for some $a_{\gamma'}\in S$. Consider the difference
\begin{equation}\label{eq:46}
h=f-\sum_{\gamma\in\bar\Gamma_x,\gamma_r=q}a_{\gamma'}\c^x_\gamma|_{\bar\Gamma_x}.
\end{equation}
By the above definitions, we get $h(\delta)=0$ for any $\delta\in\bar\Gamma_x$ such that
$\delta_r=q$:
\begin{multline*}
h(\delta)=f(\delta)-\sum_{\gamma\in\bar\Gamma_x,\gamma_r=q}a_{\gamma'}\c^x_\gamma(\delta)=
f'(\delta')-\sum_{\gamma\in\bar\Gamma_x,\gamma_r=q}a_{\gamma'}\Delta(\c_{\gamma'}^{xq})(\delta)=\\
=\sum_{\gamma'\in\bar\Gamma'_{xq}}a_{\gamma'}\c^{xq}_{\gamma'}(\delta')
-\sum_{\gamma\in\bar\Gamma_x,\gamma_r=q}a_{\gamma'}\c_{\gamma'}^{xq}(\delta')=0.
\end{multline*}
Let $t$ be the element of $\{e,s_r\}$ distinct form $q$.
We clearly have $xt<xts_r$. Thus by Lemma~\ref{lemma:12}, we get that the function
$h'$ defined by $h'(\gamma')=h(\gamma'\cdot t)/\bbeta_r(\gamma'\cdot t)$ for $\gamma'\in{\bar\Gamma}'_{xt}$
is a well-defined element of $(\bar\X')^{xt}$. By induction,
$h'=\sum_{\gamma'\in\bar\Gamma'_{xt}}b_{\gamma'}\c^{xt}_{\gamma'}|_{\bar\Gamma'_{xt}}$ for some $b_{\gamma'}\in S$.
We get
$
h=\sum_{\gamma\in\bar\Gamma_x,\gamma_r=t}b_{\gamma'}\c^x_\gamma|_{\bar\Gamma_x}$.
Indeed both sides evaluate to $0$ at $\delta\in\bar\Gamma_x$ such that $\delta_r=q$ and
for $\delta\in\bar\Gamma_x$ such that $\delta_r=t$, we get
\begin{multline*}
h(\delta)-\sum_{\gamma\in\bar\Gamma_x,\gamma_r=t}b_{\gamma'}\c^x_\gamma(\delta)
=h(\delta'\cdot t)-\sum_{\gamma\in\bar\Gamma_x,\gamma_r=t}b_{\gamma'}\nabla_t(\c_{\gamma'}^{xt})(\delta'\cdot t)\\
=
\bbeta_r(\delta)h'(\delta')-\sum_{\gamma\in\bar\Gamma_x,\gamma_r=t}b_{\gamma'}\bbeta_r(\delta)\c_{\gamma'}^{xt}(\delta')=\\
=\bbeta_r(\delta)\sum_{\gamma'\in\bar\Gamma'_{xt}}b_{\gamma'}\c^{xt}_{\gamma'}(\delta')-
\sum_{\gamma\in\bar\Gamma_x,\gamma_r=t}b_{\gamma'}\bbeta_r(\delta)\c_{\gamma'}^{xt}(\delta')=0.
\end{multline*}
Hence and from~(\ref{eq:46}), we get that $f$ is a $S$-linear combination of elements of our set.

Finally, let us prove the linear independence. Suppose that we have $\sum_{\gamma\in\bar\Gamma_x}a_\gamma\c_\gamma^x=0$
for some $a_\gamma\in S$. We can write this sum as follows
\begin{equation}\label{eq:36}
\sum_{\gamma\in\bar\Gamma_x,\gamma_r=q}a_\gamma\Delta(\c_{\gamma'}^{xq})
+\sum_{\gamma\in\bar\Gamma_x,\gamma_r=t}a_\gamma\nabla_t(\c_{\gamma'}^{xt})=0.
\end{equation}
Evaluation at $\delta\in\bar\Gamma_x$ with $\delta_r=q$ yields
$\sum_{\gamma'\in\bar\Gamma'_{xq}}a_\gamma\c_{\gamma'}^{xq}(\delta')=0$. Hence by the inductive hypothesis,
$a_\gamma=0$ for any
$\gamma\in\bar\Gamma_x$ such that $\gamma_r=q$.
Therefore~(\ref{eq:36}) takes the form
$\sum_{\gamma\in\bar\Gamma_x,\gamma_r=t}a_\gamma\nabla_t(\c_{\gamma'}^{xt})=0$.
Evaluation at $\delta\in\bar\Gamma_x$ with $\delta_r=t$ yields
$\sum_{\gamma'\in\bar\Gamma'_{xt}}a_\gamma\bbeta_r(\delta)\c_{\gamma'}^{xt}(\delta')=0$,
whence $\sum_{\gamma'\in\bar\Gamma'_{xt}}a_\gamma\c_{\gamma'}^{xt}(\delta')=0$.
By the inductive hypothesis $a_\gamma=0$ for any
$\gamma\in\bar\Gamma_x$ such that $\gamma_r=t$.
\end{proof}

\begin{remark}\label{remark:2}\rm We have $\{\c_\gamma^x\suchthat\gamma\in\Gamma\}=B_{\xi_r(x)}$, where $\xi_r(x)\in\Upsilon$ is defined
inductively as follows: $\xi_0(x)=e$ and
$$
x\xi_r(x)_\emptyset<x\xi_r(x)_\emptyset s_r,\quad \xi_r(x)'_0=\xi_{r-1}(x\xi_r(x)_\emptyset s_r),\quad \xi_r(x)'_1=\xi_{r-1}(x\xi_r(x)_\emptyset).
$$
\end{remark}

\section{The costalk-to-stalk embedding and the decomposition of the direct image}\label{costalk-to-stalk embedding}

\subsection{Change of coefficients}\label{Change_of_coefficients}
Let $k$ be a principal ideal domain with invertible $2$ if the root system contains a component of type $C_n$.
%
%
%
%
Consider the canonical ring homomorphism $\Z'\to k$. It extends to the ring homomorphism $S\to S_k$.
We get the following commutative diagram:
\begin{equation}\label{eq:50}
\begin{tikzcd}
H^i(\Sigma)\otimes_{\Z'}k\arrow{r}\arrow{d}&H^i(\Sigma,k)\arrow{d}\\
H^i(\Gamma)\otimes_{\Z'}k\arrow{r}&H^i(\Gamma,k)
\end{tikzcd}
\end{equation}
As $H^i(\Sigma)$ and $H^i(\Sigma,k)$ vanish in odd degrees, $\Sigma$ is compact and $\Gamma$ is finite,
Proposition~\ref{proposition:4} implies that the horizontal arrows are isomorphisms.
Hence we get the following chain of isomorphisms:
$$
H^\bullet_T(\Sigma)\otimes_SS_k\simeq(H^\bullet(\Sigma)\otimes_{\Z'}S)\otimes_SS_k\simeq(H^\bullet(\Sigma)\otimes_{\Z'}k)\otimes_kS_k\simeq H^\bullet(\Sigma,k)\otimes_kS_k\simeq H_T^\bullet(\Sigma,k).
$$
The similar chain yields an isomorphism $H^\bullet_T(\Gamma)\otimes_SS_k\ito H^\bullet_T(\Gamma,k)$. Diagram~(\ref{eq:50})
proves that these isomorphisms are compatible with the restriction from $\Sigma$ to $\Gamma$. This means that we get
the following commutative diagram:
$$
\begin{tikzcd}
{}&H_T^\bullet(\Sigma)\otimes_SS_k\arrow{r}{\sim}\arrow[hook]{d}&H_T^\bullet(\Sigma,k)\arrow[hook]{d}\\
\X\otimes_SS_k\arrow[hook]{r}\arrow{ur}{\sim}&H_T^\bullet(\Gamma)\otimes_SS_k\arrow{r}{\sim}&H_T^\bullet(\Gamma,k)
\end{tikzcd}
$$
If we go along the upper path, then we get an isomorphism of $S_k$-modules $\X\otimes_SS_k\ito\X(k)$.
However this map is the same as the map of the lower path.
Hence we get the following result.

\begin{lemma}\label{lemma:15}
There exist an isomorphism of $S_k$-modules (dashed arrow) such that the following diagram is commutative:
$$
\begin{tikzcd}
\X\otimes_SS_k\arrow[dashed]{r}{\exists}[swap]{\sim}\arrow[hook]{d}&\X(k)\arrow[hook]{d}\\
H^\bullet_T(\Gamma)\otimes_S S_k\arrow{r}[swap]{\sim}&H_T^\bullet(\Gamma,k)
\end{tikzcd}
$$
\end{lemma}

Arguing similarly with $\Sigma_x$ and $\Gamma_x$, we get the following results.

\begin{lemma}\label{lemma:16}
%
%
There exist an isomorphism of $S_k$-modules (dashed arrow) such that the following diagram is commutative:
$$
\begin{tikzcd}
\X^x\otimes_SS_k\arrow[dashed]{r}{\exists}[swap]{\sim}\arrow[hook]{d}&\X^x(k)\arrow[hook]{d}\\
H^\bullet_T(\Gamma_x)\otimes_S S_k\arrow{r}[swap]{\sim}&H_T^\bullet(\Gamma_x,k)
\end{tikzcd}
$$
\end{lemma}

The case of $\OX^x(k)$ is more difficult, as $\OSigma_x$ is in general not compact.
However, we can use the Poincar\'e duality $H^i(\bar\Sigma_x,k)\simeq\Hom_{k\text{-mod}}(H_c^{2\dim\Sigma-i}(\bar\Sigma_x,k),k)$
establishes in Section~\ref{Localization_for_Sigma_without_x}.
We get the following sequence of canonical maps:

\vspace{-18pt}

\begin{multline*}
H^i(\OSigma_x)\otimes_{\Z'}k\simeq\Hom_{\Z'\text{-mod}}(H_c^{2\dim\Sigma-i}(\bar\Sigma_x),\Z')\otimes_{\Z'}k\stackrel\phi\to
\Hom_{k\text{-mod}}(H_c^{2\dim\Sigma-i}(\bar\Sigma_x)\otimes_{\Z'}k,k)\\[6pt]
\shoveright{\simeq\Hom_{k\text{-mod}}(H_c^{2\dim\Sigma-i}(\bar\Sigma_x,k),k)\simeq H^i(\bar\Sigma_x,k).}\\
\end{multline*}

\vspace{-18pt}

\noindent
From~Section~\ref{Localization_for_Sigma_without_x}, we know that $H_c^{2\dim\Sigma-i}(\bar\Sigma_x)$ is a finitely generated
free $\Z'$-module. Hence we conclude that the morphism $\phi$ in the sequence above is an isomorphism.
If we replace $\OSigma_x$ by $\OGamma_x$ in this argument, then we get an isomorphism
$H^i(\OGamma_x)\otimes_{\Z'}k\simeq H^i(\OGamma_x,k)$. It is rather easy to see that these
isomorphisms are compatible with the restriction from $\OSigma_x$ to $\OGamma_x$.
An argument similar to the one preceding Lemma~\ref{lemma:15}, proves the following result.

\begin{lemma}\label{lemma:17}
There exist an isomorphism of $S_k$-modules (dashed arrow) such that the following diagram is commutative:
$$
\begin{tikzcd}
\OX^x\otimes_SS_k\arrow[dashed]{r}{\exists}[swap]{\sim}\arrow[hook]{d}&\OX^x(k)\arrow[hook]{d}\\
H^\bullet_T(\OGamma_x)\otimes_S S_k\arrow{r}[swap]{\sim}&H_T^\bullet(\OGamma_x,k)
\end{tikzcd}
$$
\end{lemma}

These three lemmas allow us to construct bases of $\X(k)$, $\X^x(k)$, $\bar\X(k)^x$ from the bases of
$\X$, $\X^x$, $\bar\X^x$ given by Theorems~\ref{theorem:4},~\ref{theorem:2},~\ref{theorem:3} respectively.
Moreover, we can construct operators $\Delta$ and $\nabla_t$ on $H^\bullet_T(\Gamma',k)$
similarly to Section~\ref{Copy and concentration} and obtain analogs of Lemmas~\ref{lemma:5} and~\ref{lemma:6}.

\subsection{Description of the costalk-to-stalk embedding}\label{Main_construction} From the $T$-equivariant distinguished triangle
$$
i_*i^!\csh{k}{\Sigma}\to\csh{k}{\Sigma}\to j_*j^*\csh{k}{\Sigma}\stackrel{+1}\to,
$$
where $i$ and $j$ are as in Section~\ref{Localization_for_Sigma_without_x}, we get the following exact sequence
\begin{equation}\label{eq:7}
H_T^n(\Sigma_x,i^!\csh{k}{\Sigma})\to H_T^n(\Sigma,k)\to H_T^n(\bar\Sigma_x,k)\to H_T^{n+1}(\Sigma_x,i^!\csh{k}{\Sigma})\to
\end{equation}
We are actually interested in the left map. It would be very convenient if we could prove that its source
$H_T^n(\Sigma_x,i^!\csh{k}{\Sigma})$ vanishes in odd degrees. This is fortunately true, as the sequence
$$
H_T^{2m}(\Sigma,k)\to H_T^{2m}(\bar\Sigma_x,k)\to H_T^{2m+1}(\Sigma_x,i^!\csh{k}{\Sigma})\to H_T^{2m+1}(\Sigma,k)=0
$$
is exact by~(\ref{eq:7}) and the left map is surjective by Theorem~\ref{theorem:3} and Lemmas~\ref{lemma:15}
and~\ref{lemma:17}.

Consider the following commutative diagram with the exact first row:
$$
\begin{tikzcd}
0\arrow{r}&H_T^{2m}(\Sigma_x,i^!\csh{k}\Sigma)\arrow[dashed]{d}{\wr}[swap]{\exists}\arrow{r}&H_T^{2m}(\Sigma,k)\arrow{d}[swap]{\wr}\arrow{r}&H_T^{2m}(\bar\Sigma_x,k)\arrow{r}\arrow{d}[swap]{\wr}&0\\
0\arrow{r}&\ker\phi^{2m}\arrow{r}&\X(k)^{2m}\arrow{r}{\phi^{2m}}&{\bar\X}^x(k)^{2m}\arrow{r}&0
\end{tikzcd}
$$
Here $\phi^{2m}$ is the restriction map $f\mapsto f|_{\OGamma_x}$ and the solid vertical
arrows are induced by embeddings $\Gamma\hookrightarrow\Sigma$ and $\OGamma_x\hookrightarrow\OSigma_x$ respectively.
We thus have proved the following result.

\begin{lemma}\label{lemma:2}
There exists an isomorphism of $S_k$-modules (dashed arrow) such that the following diagram is commutative:
$$
\begin{tikzcd}
H_T^\bullet(\Sigma_x,i^!\csh{k}\Sigma)\arrow[dashed]{d}{\wr}[swap]{\exists}\arrow{r}&H_T^\bullet(\Sigma,k)\arrow{d}{\wr}\\
\X_x(k)\arrow[hook]{r}&\X(k)
\end{tikzcd}
$$
where
$
\X_x(k)=\{f\in\X(k)\suchthat f|_{\OGamma_x}=0\}
$
and the bottom arrow is the natural embedding.
\end{lemma}

\begin{corollary}\label{corollary:5}
The functor $H^\bullet(\Sigma_x,\_)$ applied to the natural morphism $i^!\csh{k}\Sigma\to i^*\csh{k}\Sigma$,
where $i:\Sigma_x\hookrightarrow\Sigma$ is the embedding,
yields a map isomorphic to the embedding $\X_x(k)\hookrightarrow\X^x(k)$.
\end{corollary}

It remains to discuss the behavior of $\X_x(k)$ with respect to the change of the ring of coefficients $k$.
By the remark at the end of Section~\ref{Basis_for_barX}, we have $\{\c_\gamma^x\suchthat\gamma\in\Gamma\}=B_{\xi_r(x)}$.
Thus we can write $B_{\xi_r(x)}=\{b_1,\ldots,b_m,b_{m+1},\ldots,b_n\}$
so that $\{b_1|_{\bar\Gamma_x},\ldots,b_m|_{\bar\Gamma_x}\}$ is a basis of $\bar\X^x$.
We have the following decompositions $b_j|_{\bar\Gamma_x}=\sum_{i=1}^mc_{i,j}b_i|_{\bar\Gamma_x}$ for some (homogeneous) $c_{i,j}\in S$.
Let $u=\sum_{i=1}^nx_ib_i$, where $x_i\in S$, be an arbitrary element of $\X$. We get
$$
u|_{\bar\Gamma_x}=\sum_{i=1}^nx_ib_i|_{\bar\Gamma_x}=\sum_{i=1}^mx_ib_i|_{\bar\Gamma_x}+\sum_{j=m+1}^nx_j\sum_{i=1}^mc_{i,j}b_i|_{\bar\Gamma_x}
=\sum_{i=1}^m\(x_i+\sum_{j=m+1}^nc_{i,j}x_j\)b_i|_{\bar\Gamma_x}.
$$
Hence $\X_x=\X_x(\Z')$ is a free $S$-module with basis $\big\{-\sum_{i=1}^mc_{i,j}b_i+b_j\big\}_{j=m+1}^n$.
Arguing similarly, we get that $\X_x(k)$ is a free $S_k$-module with basis
$\big\{-\sum_{i=1}^m(c_{i,j}\otimes 1_k)(b_i\otimes 1_k)+b_j\otimes 1_k\}_{j=m+1}^n$.

\begin{lemma}\label{lemma:18}
%
There exists an isomorphism of $S_k$-modules (dashed arrow) such that the following diagram is commutative:
$$
\begin{tikzcd}
\X_x\otimes_SS_k\arrow[dashed]{r}{\exists}[swap]{\sim}\arrow[hook]{d}&\X_x(k)\arrow[hook]{d}\\
H^\bullet_T(\Gamma)\otimes_S S_k\arrow{r}[swap]{\sim}&H_T^\bullet(\Gamma,k)
\end{tikzcd}
$$
\end{lemma}

\subsection{Description of $\X_x(k)$} We are going to describe this module via the dual of $\X^x(k)$. This is a well known description due to
Fiebig~\cite[Lemmas 6.8, 6.9 and 6.13]{Fiebig}. We present here an alternative proof that does not require invertibility of $2$ in $k$.
Let
$$
D\X^x(k)=\{g\in\Map(\Gamma_x,Q_k)\suchthat(g,f)\in S_k\text{ for any }f\in\X^x(k)\},
$$
where $Q_k$ is the ring of quotients of $S_k$ and $(g,f)=\sum_{\gamma\in\Gamma_x}g_\gamma f_\gamma$ (the standard scalar product).
It will be convenient, for example, in Lemma~\ref{lemma:20} to identify elements of $D\X^x(k)$ with their extensions by zero to $\Gamma$.

Consider $P_k\in H_T^\bullet(\Gamma,k)$ defined by
$$
P_k(\gamma)=\prod_{i=1}^r\bbeta_i(\gamma)\otimes1_k=\pm\prod_{\alpha\in \Phi^+}(\alpha\otimes1_k)^{|M_\alpha(\gamma)|}.
$$
Note that any $P_k(\gamma)$ is divisible in $S_k$ by the Euler class $e_x(k)=\prod_{\alpha\in \Phi^+,s_\alpha x<x}\alpha\otimes1_k$
(see for example~\cite[Lemma 4.9.7]{Shchigolev}).

\begin{lemma}\label{lemma:20}
$\X_x(k)=P_kD\X^x(k)$.
\end{lemma}
\begin{proof} First we prove by induction on $r$ that $P_kg\in H_T^\bullet(\Gamma_x,k)$ for any $g\in D\X^x(k)$.
This is clear for $r=0$, so we assume that $r>0$. Let $t\in\{e,s_r\}$. By Lemma~\ref{lemma:6}, we get
$$
S_k\ni\big(g,\nabla_tf'\big|_{\Gamma_x}\big)=\sum_{\gamma'\in\Gamma'_{xt}}g(\gamma'\cdot t)\bbeta_r(\gamma'\cdot t)f'(\gamma')
$$
for any $f'\in\X'(k)$. Hence the function
$$
g'(\gamma')=\bbeta_r(\gamma'\cdot t)g(\gamma'\cdot t)
$$
where $\gamma'\in\Gamma'_{xt}$, belongs to $D(\X')^{xt}(k)$.
By the inductive hypothesis, the product $P'_kg'$ has values in $S_k$:
$$
S_k\ni P'(\gamma')g'(\gamma')=\(\prod_{i=1}^{r-1}\bbeta_i(\gamma')\)\bbeta_r(\gamma'\cdot t)g(\gamma'\cdot t)=P(\gamma'\cdot t)g(\gamma'\cdot t).
$$
As $t$ is arbitrary, the function $P_kg$ has values in $S_k$.

Now we are going to prove the lemma $k=\Z'$. In this case, we denote $P=P_{\Z'}$ and $D\X^x=D\X^x(\Z')$.
We apply induction on $r$, the result being obvious for $r=0$. Assume that $r>0$.

Let us prove that $Pg\in\X_x$ for $g\in D\X^x$. We actually must prove that the extension by zero of $Pg$ to $\Gamma$
belongs to $\X$, which by Proposition~\ref{proposition:1} is equivalent to checking that
$$
\sum_{\delta\in\Gamma_x,\delta\sim_\alpha\gamma,J_\alpha(\delta)\subset J_\alpha(\gamma)}(-1)^{|J_\alpha(\delta)|}P(\delta)g(\delta)\=0\pmod{\alpha^{|J_\alpha(\gamma)|}}
$$
for any $\alpha\in \Phi^+$ and $\gamma\in\Gamma$.
In this summation, $P(\delta)$ is clearly divisible by $\alpha^{|M_\alpha(\delta)|}=\alpha^{|M_\alpha(\gamma)|}$.
Hence it also divisible by $\alpha^{|J_\alpha(\gamma)|}$.
Therefore it remains to prove that the function
$$
p^\alpha_\gamma(\delta)=\left\{
\begin{array}{ll}
(-1)^{|J_\alpha(\delta)|}P(\delta)/\alpha^{|J_\alpha(\gamma)|}&\text{if }\delta\sim_\alpha\gamma\text{ and }J_\alpha(\delta)\subset J_\alpha(\gamma);\\
0&\text{otherwise}
\end{array}
\right.
$$
where $\delta\in\Gamma$, belongs to $\X$. It is rather difficult to prove it directly, applying Proposition~\ref{proposition:1}.
Therefore, we define the following function $q_\gamma^\alpha$ by induction: $q_\emptyset^\alpha=1$ and
$$
q^\alpha_\gamma=
\left\{
\begin{array}{ll}
\nabla_{\gamma_r}q^\alpha_{\gamma'}&\text{ if }r\notin M_\alpha(\gamma);\\[6pt]
\nabla_{\gamma_r} q^\alpha_{\gamma'}+\nabla_{\gamma_rs_r} q^\alpha_{\gamma'}-\alpha\Delta q^\alpha_{\gamma'}&\text{ if }r\in M_\alpha(\gamma)\setminus J_\alpha(\gamma)\\[6pt]
-\Delta q^\alpha_{\gamma'}&\text{ if }r\notin J_\alpha(\gamma).
\end{array}
\right.
$$
if $r>0$. By Lemmas~\ref{lemma:5} and~\ref{lemma:6}, we get $q^\alpha_{\gamma}\in\X$.
Therefore, it suffices to prove that
$$
2^{|M_\alpha(\gamma)|-|J_\alpha(\gamma)|}p^\alpha_\gamma=q^\alpha_\gamma.
$$
This formula is obvious for $r=0$. Therefore, we consider the case $r>0$ and apply induction.

{\it Case 1: $r\notin M_\alpha(\gamma)$.} If $\delta\not\sim_\alpha\gamma$, then either $\delta_r\ne\gamma_r$ or
$\delta'\not\sim_\alpha\gamma'$. In both cases, $q^\alpha_\gamma(\delta)=\nabla_{\gamma_r}q^\alpha_{\gamma'}(\delta)=0$.
Now assume that $\delta\sim_\alpha\gamma$. Then $\delta_r=\gamma_r$, $\delta'\sim_\alpha\gamma'$,
$J_\alpha(\delta)=J_\alpha(\delta')$, $J_\alpha(\gamma)=J_\alpha(\gamma')$, $M_\alpha(\gamma)=M_\alpha(\gamma')$.
If $J_\alpha(\delta)\not\subset J_\alpha(\gamma)$, then $J_\alpha(\delta')\not\subset J_\alpha(\gamma')$ and we get
$$
q^\alpha_{\gamma}(\delta)=\nabla_{\gamma_r}q^\alpha_{\gamma'}(\delta)=\bbeta_r(\delta)q^\alpha_{\gamma'}(\delta')=2^{|M_\alpha(\gamma')|-|J_\alpha(\gamma')|}\bbeta_r(\delta)p^\alpha_{\gamma'}(\delta')=0.
$$
If $J_\alpha(\delta)\subset J_\alpha(\gamma)$, then $J_\alpha(\delta')\subset J_\alpha(\gamma')$ and we get
\begin{multline*}
q^\alpha_\gamma(\delta)=\nabla_{\gamma_r}q^\alpha_{\gamma'}(\delta)=\bbeta_r(\delta)q^\alpha_{\gamma'}(\delta')=
2^{|M_\alpha(\gamma')|-|J_\alpha(\gamma')|}\bbeta_r(\delta)p^\alpha_{\gamma'}(\delta')\\[6pt]
\shoveleft{=2^{|M_\alpha(\gamma)|-|J_\alpha(\gamma)|}(-1)^{|J_\alpha(\delta)|}\bbeta_r(\delta)P(\delta')/\alpha^{|J_\alpha(\gamma)|}}\\[6pt]
=2^{|M_\alpha(\gamma)|-|J_\alpha(\gamma)|}(-1)^{|J_\alpha(\delta)|}P(\delta)/\alpha^{|J_\alpha(\gamma)|}
=2^{|M_\alpha(\gamma)|-|J_\alpha(\gamma)|}p^\alpha_\gamma(\delta).
\end{multline*}

{\it Case 2: $r\in M_\alpha(\gamma)\setminus J_\alpha(\gamma)$.} If $\delta\not\sim_\alpha\gamma$, then
$\delta'\not\sim_\alpha\gamma'$. In this case, $q^\alpha_\gamma(\delta)=0$, as $q^\alpha_{\gamma'}(\delta')=2^{|M_\alpha(\gamma')|-|J_\alpha(\gamma')|}p^\alpha_{\gamma'}(\delta')=0$.
If $J_\alpha(\delta)\not\subset J_\alpha(\gamma)$, then either $r\in J_\alpha(\delta)$ or
$J_\alpha(\delta')\not\subset J_\alpha(\gamma')$. In the former case, we get $\bbeta_r(\delta)=\alpha$ and
$$
q^\alpha_\gamma(\delta)= \nabla_{\gamma_r}q^\alpha_{\gamma'}(\delta)+\nabla_{\gamma_rs_r}q^\alpha_{\gamma'}(\delta)
-\alpha\Delta q^\alpha_{\gamma'}(\delta)
=\bbeta_r(\delta)q^\alpha_{\gamma'}(\delta')-\alpha q^\alpha_{\gamma'}(\delta')=0.
$$
In the latter case, we get $q^\alpha_{\gamma'}(\delta')=2^{|M_\alpha(\gamma')|-|J_\alpha(\gamma')|}p^\alpha_{\gamma'}(\delta')=0$,
whence $q^\alpha_\gamma(\delta)=0$.

Now suppose that $\delta\sim_\alpha\gamma$ and $J_\alpha(\delta)\subset J_\alpha(\gamma)$.
Then $\delta'\sim_\alpha\gamma'$, $J_\alpha(\delta')\subset J_\alpha(\gamma')$ and $r\notin J_\alpha(\delta)$.
It follows from the last formula that $\bbeta_r(\delta)=-\alpha$. Hence we get
\begin{multline*}
q^\alpha_\gamma(\delta)=\nabla_{\gamma_r}q^\alpha_{\gamma'}(\delta)+\nabla_{\gamma_rs_r}q^\alpha_{\gamma'}(\delta)
-\alpha\Delta q^\alpha_{\gamma'}(\delta)
=\bbeta_r(\delta)q^\alpha_{\gamma'}(\delta')-\alpha q^\alpha_{\gamma'}(\delta')\\
=-2\alpha q^\alpha_{\gamma'}(\delta')
=2^{|M_\alpha(\gamma')|-|J_\alpha(\gamma')|+1}(-1)^{|J_\alpha(\delta')|}(-\alpha)P(\delta')/\alpha^{|J_\alpha(\gamma')|}\\
=2^{|M_\alpha(\gamma)|-|J_\alpha(\gamma)|}(-1)^{|J_\alpha(\delta)|}P(\delta)/\alpha^{|J_\alpha(\gamma)|}
=2^{|M_\alpha(\gamma)|-|J_\alpha(\gamma)|}p^\alpha_\gamma(\delta).
\end{multline*}

{\it Case 3: $r\in J_\alpha(\gamma)$.} If $\delta\not\sim_\alpha\gamma$, then
$\delta'\not\sim_\alpha\gamma'$. In this case, $q^\alpha_\gamma(\delta)=0$, as $q^\alpha_{\gamma'}(\delta')=2^{|M_\alpha(\gamma')|-|J_\alpha(\gamma')|}p^\alpha_{\gamma'}(\delta')=0$.
If $J_\alpha(\delta)\not\subset J_\alpha(\gamma)$, then $J_\alpha(\delta')\not\subset J_\alpha(\gamma')$
and we again get $q^\alpha_\gamma(\delta)=0$, as $q^\alpha_{\gamma'}(\delta')=2^{|M_\alpha(\gamma')|-|J_\alpha(\gamma')|}p^\alpha_{\gamma'}(\delta')=0$.

Now suppose that $\delta\sim_\alpha\gamma$ and $J_\alpha(\delta)\subset J_\alpha(\gamma)$.
Then $\delta'\sim_\alpha\gamma'$ and $J_\alpha(\delta')\subset J_\alpha(\gamma')$.

If $r\notin J_\alpha(\delta)$, then $\bbeta_r(\delta)=-\alpha$ and
\begin{multline*}
q^\alpha_\gamma(\delta)=-\Delta q^\alpha_{\gamma'}(\delta)=-q^\alpha_{\gamma'}(\delta')
=-2^{|M_\alpha(\gamma')|-|J_\alpha(\gamma')|}p^\alpha_{\gamma'}(\delta')\\[6pt]
=-2^{|M_\alpha(\gamma)|-|J_\alpha(\gamma)|}(-1)^{|J_\alpha(\delta')|}P(\delta')/\alpha^{|J_\alpha(\gamma')|}
=2^{|M_\alpha(\gamma)|-|J_\alpha(\gamma)|}(-1)^{|J_\alpha(\delta')|}(-\alpha)P(\delta')/\alpha^{|J_\alpha(\gamma)|}\\[6pt]
=2^{|M_\alpha(\gamma)|-|J_\alpha(\gamma)|}(-1)^{|J_\alpha(\delta)|}P(\delta)/\alpha^{|J_\alpha(\gamma)|}
=2^{|M_\alpha(\gamma)|-|J_\alpha(\gamma)|}p^\alpha_{\gamma}(\delta).
\end{multline*}

If $r\in J_\alpha(\delta)$, then $\bbeta_r(\delta)=\alpha$ and
\begin{multline*}
q^\alpha_\gamma(\delta)=-\Delta q^\alpha_{\gamma'}(\delta)=-q^\alpha_{\gamma'}(\delta')
=-2^{|M_\alpha(\gamma')|-|J_\alpha(\gamma')|}p^\alpha_{\gamma'}(\delta')\\[6pt]
=-2^{|M_\alpha(\gamma)|-|J_\alpha(\gamma)|}(-1)^{|J_\alpha(\delta')|}P(\delta')/\alpha^{|J_\alpha(\gamma')|}
=2^{|M_\alpha(\gamma)|-|J_\alpha(\gamma)|}(-1)^{|J_\alpha(\delta)|}\alpha P(\delta')/\alpha^{|J_\alpha(\gamma)|}\\[6pt]
=2^{|M_\alpha(\gamma)|-|J_\alpha(\gamma)|}(-1)^{|J_\alpha(\delta)|}P(\delta)/\alpha^{|J_\alpha(\gamma)|}
=2^{|M_\alpha(\gamma)|-|J_\alpha(\gamma)|}p^\alpha_\gamma(\delta).
\end{multline*}

Finally, we prove the inverse inclusion. Let $f\in\X_x$. We apply induction on the cardinality of the following
set (lower closure of the support of $f$):
$$
\widehat C(f)=\{\delta\in\Gamma_x\suchthat\text{ there exists }\gamma\in\Gamma_x\text{ such that }\delta\le\gamma\text{ and }f(\gamma)\ne0\}.
$$
If $\widehat C(f)=\emptyset$, then $f=0$ and the result follows.
Suppose now that $\widehat C(f)\ne\emptyset$ and let $\gamma$ be its element maximal with respect to $<$.
Let $\alpha$ be a positive root.

First suppose that $s_\alpha x>x$. In this case, $|J_\alpha(\gamma)|=|D_\alpha(\gamma)|$.
From~\cite[Theorem~6.2(3)]{Haerterich}, we get
$$
\sum_{\delta\in\Gamma_x,\delta\sim_\alpha\gamma,J_\alpha(\gamma)\subset J_\alpha(\delta)}(-1)^{|J_\alpha(\delta)|}f(\delta)\=0\pmod{\alpha^{|M_\alpha(\gamma)|-|J_\alpha(\gamma)|}}.
$$
Proposition~\ref{lemma:4} and~(\ref{eq:39}) imply that $\delta\ge\gamma$ for any $\delta$ in the above summation.
Thus $f(\gamma)$ is divisible by $\alpha^{|M_\alpha(\gamma)|-|J_\alpha(\gamma)|}=\alpha^{|M_\alpha(\gamma)|-|D_\alpha(\gamma)|}$.

Now suppose that $s_\alpha x<x$. In this case, $|J_\alpha(\gamma)|=|D_\alpha(\gamma)|+1$.
Let $j$ be the greatest element of $M_\alpha(\gamma)$. Note that $j$ is also the greatest element of $J_\alpha(\gamma)$.
Let $\widetilde\gamma$ be obtained from $\gamma$ by replacing $\gamma_j$ with $\gamma_js_j$. We clearly have
$\pi(\widetilde\gamma)=s_\alpha x$, $\widetilde\gamma\sim_\alpha\gamma$ and
$J_\alpha(\gamma)=J_\alpha(\widetilde\gamma)\sqcup\{j\}$, whence
$|J_\alpha(\widetilde\gamma)|=|J_\alpha(\gamma)|-1=|D_\alpha(\gamma)|$.
From~\cite[Theorem~6.2(3)]{Haerterich}, we get
$$
\sum_{\delta\in\Gamma_x,\delta\sim_\alpha\gamma,J_\alpha(\widetilde\gamma)\subset J_\alpha(\delta)}(-1)^{|J_\alpha(\delta)|}f(\delta)\=0\pmod{\alpha^{|M_\alpha(\gamma)|-|D_\alpha(\gamma)|}}.
$$
As $j\in J_\alpha(\delta)$ for any $\delta$ in the above summation, we can replace there the condition
$J_\alpha(\widetilde\gamma)\subset J_\alpha(\delta)$ with $J_\alpha(\gamma)\subset J_\alpha(\delta)$.
Hence again $\delta\ge\gamma$ in the above summation and $f(\gamma)$ is divisible by $\alpha^{|M_\alpha(\gamma)|-|D_\alpha(\gamma)|}$.

As a result, we get that $f(\gamma)$ is divisible by
$$
\prod_{\alpha\in \Phi^+}\alpha^{|M_\alpha(\gamma)|-|D_\alpha(\gamma)|}=\pm\frac{P(\gamma)}{\aa(\gamma)}.
$$
It follows from Theorem~\ref{theorem:2} that $D\X^x$ has an $S$-basis $\{\hat\b_\gamma\}_{\gamma\in\Gamma_x}$ such that
$\hat\b_\gamma(\gamma)=1/\aa(\gamma)$ and $\hat\b_\gamma(\delta)=0$ for $\delta\in\Gamma_x$ with $\delta>\gamma$.
To get this basis, one should invert and transpose the matrix given by~(\ref{eq:40}).

Consider the difference $h=f-f(\gamma)/(P(\gamma)/\aa(\gamma))P\hat\b_\gamma$.
We get $C(h)\subset\{\delta\in\Gamma_x\suchthat\delta<\gamma\}$ $\varsubsetneq C(f)$.
By induction, $h$ belongs to the $PD\X^x$. Thus so does~$f$.

Finally, let us return to the general case.
By Lemma~\ref{lemma:16}, the basis $\{\hat\b_\gamma\}_{\gamma\in\Gamma_x}$ of $D\X^x$ mentioned above and
the similar basis for $D\X^x(k)$ yield an isomorphism (dashed arrow) making the following diagram commutative:
$$
\begin{tikzcd}
PD\X^x\otimes_S S_k\arrow[dashed]{r}{\sim}\arrow{d}&P_kD\X^x(k)\arrow{d}\\
H^\bullet_T(\Gamma)\otimes_S S_k\arrow{r}{\sim}&H^\bullet_T(\Gamma,k)
\end{tikzcd}
$$
Multiplying it by $P_k$, we get by Lemma~\ref{lemma:18}, an isomorphism (dashed arrow) making the following diagram commutative:
$$
\begin{tikzcd}
{}&\X_x\otimes_S S_k\arrow{rr}{\sim}\arrow{ddl}&&\X_x(k)\arrow{ddl}\\[-6pt]
(PD\X^x)\otimes_S S_k\arrow[crossing over]{rr}{\sim}\arrow{d}\arrow[-, double equal sign distance]{ru}&&P_kD\X^x(k)\arrow{d}\arrow[dashed]{ur}{\sim}&\\
H^\bullet_T(\Gamma)\otimes_S S_k\arrow{rr}{\sim}&&H^\bullet_T(\Gamma,k)
\end{tikzcd}
$$
The commutativity of the right triangle means that the isomorphism represented by the dashed arrow is over $H^\bullet_T(\Gamma,k)$,
which proves that it is the equality of subsets.
\end{proof}

\begin{proposition}\label{proposition:5} Let $H_x$ be the matrix defined by~(\ref{eq:40}) and
$P_x$ be the diagonal matrix with $\gamma^{th}$ entry $P_{\Z'}(\gamma)$. We set $H_{x,k}=H_x\otimes_SS_k$ and
$P_{x,k}=P_x\otimes_SS_k$. The costalk-to-stalk
embedding $\X_x(k)\hookrightarrow\X^x(k)$ is described by the transition matrix
$(H_{x,k}^{-1})^T\,P_{x,k}\,H_{x,k}^{-1}$. All entries of this matrix are divisible in $S_k$ by the Euler
class $e_x(k)=\prod_{\alpha\in \Phi^+,s_\alpha x<x}\alpha\otimes1_k$.
\end{proposition}
\begin{proof} We need only to prove the divisibility. It suffices to consider the case $k=\Z'$. We denote $e_x=e_x(\Z')$.
Let $f\in\X_x$. By
Proposition~\ref{proposition:1}, we get
$$
\sum_{\delta\in\Gamma_x,\delta\sim_\alpha\gamma,J_\alpha(\delta)\subset J_\alpha(\gamma)}(-1)^{|J_\alpha(\delta)|}f(\delta)\=0\pmod{\alpha^{|J_\alpha(\gamma)|}}
$$
for any $\alpha\in\Phi^+$. If $s_\alpha x<x$, then $|J_\alpha(\gamma)|=|D_\alpha(\gamma)|+1$. By Proposition~\ref{lemma:4} and~(\ref{eq:39}),
we get $\delta\le\gamma$ in the above summation. Hence we get by induction that $f(\gamma)$ is divisible in $S$ by $e_x=e_x(\Z')$.
Dividing the above equivalence by $\alpha$ if $s_\alpha x<x$ and taking into account that different roots are not proportional, we obtain
$$
\sum_{\delta\in\Gamma_x,\delta\sim_\alpha\gamma,D_\alpha(\delta)\subset D_\alpha(\gamma)}(-1)^{|J_\alpha(\delta)|}f(\delta)/e_x\=0\pmod{\alpha^{|D_\alpha(\gamma)|}}.
$$
Thus we have proved that $f/e_x\in\X^x$.
\end{proof}


\subsection{Euler classes}\label{Euler_Classes} Consider the following closed $T$-equivariant embedding
$$
f:\{0\}\hookrightarrow V=\C_{\lm_1}\oplus\cdots\oplus\C_{\lm_d},
$$
where $\lm_i$ are characters of $T$ and $\C_{\lm_i}$ is the corresponding one-dimensional representation of $T$
(see the proof of Theorem~\ref{theorem:1}). We assume that $k$ is a field such that
$\lm_i\otimes_\Z k\ne0$ for any $i$. 
We get the following exact sequence
\begin{equation}\label{eq:47}
H^\bullet_T(\{0\},f^!\csh kV)\to H^\bullet_T(V,k)\to H^\bullet_T(V\setminus\{0\},k)
\end{equation}
We want to look more closely at the right map. By~\cite[(2.15)]{Atiyah_Bott}, we have the following
commutative diagram with exact top row:
$$
\begin{tikzcd}
H^\bullet_T(V,V\setminus\{0\},k)\arrow{r}&H^\bullet_T(V,k)\arrow{d}{f^*}[swap]{\wr}\arrow{r}&H^\bullet_T(V\setminus\{0\},k)\\
H_T^{\bullet-2d}(\{0\},k)\arrow{u}[swap]{\Phi}{\wr}\arrow{ur}{f_*}\arrow{r}{e\cup ?}&H_T^\bullet(\{0\},k)&
\end{tikzcd}
$$
where $\Phi$ is the Thom isomorphism, $f_*$ is the push-forward and $e=\prod_{i=1}^d\lm_i\otimes_\Z k$
is the corresponding Euler class. By our assumption $e\ne0$. Hence $H^\bullet_T(V\setminus\{0\},k)$ vanishes in
odd degrees and $H^\bullet_T(V,k)\to H^\bullet_T(V\setminus\{0\},k)$ is epimorphic in any degree.
Coming back to~(\ref{eq:47}), we obtain the isomorphisms
$H^\bullet_T(\{0\},f^!\csh kV)\simeq H_T^{\bullet-2d}(\{0\},k)\simeq S_k[-2d]$ and
$H^\bullet_T(V,k)\simeq S_k$ under which the map $H^\bullet_T(\{0\},f^!\csh kV)\to H^\bullet_T(V,k)$
becomes the multiplication by $e$.

%
%

\subsection{Defect of a homomorphism} Recall that $S_k$ has the maximal ideal $\m=\bigoplus_{i>0}S_k^i$.
We clearly have $S_k/m\simeq k$. Let $\phi:U\to V$ be a homomorphism of graded $S_k$-modules.
Then we can consider the quotient $\im\phi/\m V$, which is a graded $S_k/\m$-module and thus also a graded $k$-vector space.
If $U$ is a finitely generated $S_k$-module, then we can define the {\it defect} of $\phi$ as the graded dimension of this quotient:
$$
\df(\phi)=\sum_{n\in\Z}\dim_k(\im\phi/\m V)^nv^{-n}.
$$
This is an element of the ring of Laurent polynomials $\Z[v,v^{-1}]$. Clearly
$\df(\phi_1\oplus\phi_2)=\df(\phi_1)+\df(\phi_2)$ and $\df(\phi[n])=v^n\df(\phi)$.
If $\phi$ is an embedding and $U$ and $V$ are finitely generated free $S_k$-modules,
then $\df(\phi)$ can be calculated as follows.

\begin{proposition}[\mbox{\cite[Corollary 3.3.3]{Shchigolev}}] Let $\phi:U\hookrightarrow V$ be an embedding of graded $S_k$-modules.
Suppose that $\{u^{(n)}_i\}_{n\in\Z,1\le i\le l_n}$ and $\{v^{(n)}_j\}_{n\in\Z,1\le j\le k_n}$
are bases of $U$ and $V$ respectively labelled in such a way that $u^{(n)}_i$ and $v^{(n)}_j$ have degree $n$.
Let
$$
\phi(u^{(n)}_i)=\sum\nolimits_{m\in\Z, 1\le j\le k_m}a_{j,i}^{(m,n)}v^{(m)}_j
$$
for corresponding homogeneous $a_{j,i}^{(m,n)}\in S_k$. For each $n\in\Z$, we denote by $A^{(n)}$
the $k_n\times l_n$-ma\-trix whose $ji^{\text{th}}$ entry is $a_{j,i}^{(n,n)}\in k$.
Then $\df(\phi)=\sum_{n\in\Z}\mathop{\rm rk}_k A^{(n)}v^{-n}$.
\end{proposition}

Finally, we describe a homomorphism of graded modules $\phi:U\to V$ can be divided by a homogeneous element $e\in S_k^d$
that is no zero divisor for $V$. Suppose that for any $u\in U$ there exists $(\phi/e)(u)\in V$ such that
$\phi(u)=e(\phi/e)(u)$. Then we get a uniquely defined homomorphism $\phi/e:U\to V$ of $S_k$-modules
such that $(\phi/e)(U_{i+d})\subset V_i$. It is a homomorphism of graded modules $\phi/e:U[d]\to V$.

\subsection{Application to parity sheaves}\label{Application_to_parity_sheaves} Our calculations have not yet involved any stratifications.
In this section, we are going to apply our results to the stratification $G/B=\bigsqcup_{x\in W}BxB/B$.
We denote $X=G/B$ and $X_x=BxB/B$ for brevity and assume that $k$ is a field. By~\cite{Williamson_parity_sheaves} there exists the following decomposition:
$$
\pi_*\csh{k}{\Sigma}[r]=\bigoplus_{x\in W}\bigoplus_{d\in\Z}\E(x,k)[-d]^{\oplus m(x,d)},
$$
where $\E(x,k)\in D_T(X,k)$ is the $T$-equivariant parity sheaf such that $\supp\E(x,k)\subset\overline{X_x}$
and $\E(x,k)|_{X_x}=\csh k{X_x}[d_x]$, where $d_x=\dim X_x$. Our aim is to calculate the multiplicities $m(x,d)$.

We rewrite the above decomposition as follows:
\begin{equation}\label{eq:51}
\pi_*\csh{k}{\Sigma}=\bigoplus_{x\in W}\bigoplus_{d\in\Z}\E(x,k)[-d-r]^{\oplus m(x,d)}
\end{equation}
and consider the natural embedding $i_x:\{x\}\hookrightarrow X$.
We are going to take the following steps:
\begin{itemize}
\item apply to both sides of~(\ref{eq:51}) the morphism of functors $H_T^\bullet(\{x\},i_x^!\_)\to H_T^\bullet(\{x\},i_x^*\_)$;\\[-6pt]
\item divide it by the Euler class $e_x=\prod_{\alpha\in\Phi^+,s_\alpha x<x}\alpha\otimes1_k$;\\[-6pt]
\item take the defect of the resulting map.
\end{itemize}

First consider the left-hand side of~(\ref{eq:51}). We have the following Cartesian diagram:
$$
\begin{tikzcd}
\Sigma_x\arrow{d}{\pi_x}\arrow{r}{i}&\Sigma\arrow{d}{\pi}\\
\{x\}\arrow{r}{i_x}&X
\end{tikzcd}
$$
As $\pi$ is proper, the base change yields $i_x^!\pi_*\csh{k}{\Sigma}\simeq(\pi_x)_*i^!\csh{k}{\Sigma}$ and
$i_x^*\pi_*\csh{k}{\Sigma}\simeq(\pi_x)_*i^*\csh{k}{\Sigma}$. Hence the map
$H_T^\bullet(\{x\},i_x^!\pi_*\csh{k}{\Sigma})\to H_T^\bullet(\{x\},i_x^*\pi_*\csh{k}{\Sigma})$ is isomorphic to
$H_T^\bullet(\Sigma_x,i^!\csh{k}{\Sigma})\to H_T^\bullet(\Sigma_x,i^*\csh{k}{\Sigma})$, which
in its turn is isomorphic to $\X_x(k)\hookrightarrow\X^x(k)$ by Corollary~\ref{corollary:5}.

It order to tackle the right-hand side of~(\ref{eq:51}), let us compute the map
\begin{equation}\label{eq:49}
H_T^\bullet(\{x\},i_x^!\E(y,k))\to H_T^\bullet(\{x\},i_x^*\E(y,k)).
\end{equation}
If $x\notin\overline{X_y}$, then $i_x^!\E(y,k)=i_x^*\E(y,k)=0$ as $\E(y,k)|_{X\setminus\overline{X_y}}=0$.
Therefore, we must only consider the case $x\in\overline{X_y}$ that is $x\le y$.

Let $U=\bigsqcup_{z\ge x}X_z$. This is an open subset of $X$, which contains
$X_x$ as a closed subset. Moreover, the restriction $\F=\E(y,k)|_U$ is an indecomposable parity sheaf on $U$.
Let $i:X_x\hookrightarrow U$, $f:\{x\}\hookrightarrow X_x$ and $\tilde\imath_x:\{x\}\hookrightarrow U$ denote the natural embeddings:
$$
\begin{tikzcd}
\{x\}\arrow{r}{f}\arrow[bend right]{rrr}[swap]{i_x}\arrow[bend left]{rr}{\tilde\imath_x}&X_x\arrow{r}{i}&U\arrow{r}&X
\end{tikzcd}
$$
Let $\phi:i^!\F\to i^*\F$ denote the natural morphism.

Suppose first that $x<y$. Then $\F$ has no direct summands supported on $X_x$.
Let us write
$i^!\F=\bigoplus_{m\in\Z}Q^m$ and $i^*\F=\bigoplus_{n\in\Z}P^n$, where $Q^m=H^m(i^!\F)[-m]$ and $P^n=H^n(i^*\F)[-n]$.
We denote by
$\phi_{n,m}:Q^m\to P^n$
the corresponding morphism of the direct summands.
By~\cite[Corollary 2.22]{Williamson_parity_sheaves}, we get that $\phi_{n,m}=0$ for $m\le n$.

We denote by $A:f^!\to f^*$ the natural morphism of functors.
Then the composition $f^*\phi\circ A(i^!\F)$ is the natural morphism $\tilde\imath_x^!\F\to\tilde\imath_x^*\F$:
$$
\begin{tikzcd}
\tilde\imath_x^!\F=f^!\circ i^!\F\arrow{r}{A(i^!\F)}&f^*\circ i^!\F\arrow{r}{f^*\phi}&f^*\circ i^*\F=\tilde\imath_x^*\F.
\end{tikzcd}
$$
Recalling our decompositions of $i^!\F$ and $i^*\F$, we represent this morphism as the direct sum of the following morphisms:
\begin{equation}\label{eq:48}
\begin{tikzcd}
f^!Q^m\arrow{r}{A(Q^m)}&f^*Q^m\arrow{r}{f^*\phi_{n,m}}&f^*P^n
\end{tikzcd}
\end{equation}
for $m>n$. We have the decompositions $Q^m=\csh k{X_x}[-m]^{\oplus a(x,m)}$ and $P^n=\csh k{X_x}[-n]^{\oplus b(x,n)}$
for some nonnegative integers $a(x,m)$ and $b(x,n)$.
Applying $H^\bullet_T(\{x\},\_)$ to~(\ref{eq:48}), we get maps from
\begin{multline*}
H^\bullet_T(\{x\},f^!Q^m)=H^\bullet_T(\{x\},f^!\csh k{X_x}[-m])^{\oplus a(x,m)}=
H^{\bullet-m}_T(\{x\},f^!\csh k{X_x})^{\oplus a(x,m)}\\
=H^{\bullet-m-2d_x}_T(\{x\},k)^{\oplus a(x,m)}=S_k[-m-2d_x]^{\oplus a(x,m)}
\end{multline*}
to
$$
H^\bullet_T(\{x\},f^*Q^m)=H^\bullet_T(\{x\},f^*\csh k{X_x}[-m])^{\oplus a(x,m)}=
H^{\bullet-m}_T(\{x\},k)^{\oplus a(x,m)}=S_k[-m]^{\oplus a(x,m)}
$$
and finally to
$$
H^\bullet_T(\{x\},f^*P^n)=H^\bullet_T(\{x\},f^*\csh k{X_x}[-n])^{\oplus b(x,n)}=
H^{\bullet-n}_T(\{x\},k)^{\oplus b(x,n)}=S_k[-n]^{\oplus b(x,n)}.
$$
So we get the following sequence of maps (see Section~\ref{Euler_Classes}):
$$
\begin{tikzcd}
S_k[-m-2d_x]^{\oplus a(x,m)}\arrow{r}{e_x\cup ?}&S_k[-m]^{\oplus a(x,m)}\arrow{r}&S_k[-n]^{\oplus b(x,n)}.
\end{tikzcd}
$$
The division by $e_x$ yields a map $S_k[-m]^{\oplus a(x,m)}\to S_k[-n]^{\oplus b(x,n)}$.
As $m>n$ the defect of this map is zero. Thus we have proved that the defect of the map
$H^\bullet_T(\{x\},\tilde\imath_x^!\F)\to H^\bullet_T(\{x\},\tilde\imath_x^*\F)$ divided by $e_x$ is zero.
Recalling that $\F=\E(y,k)|_U$, we get that the defect of 
(\ref{eq:49}) divided by $e_x$ is also zero.

It remains to calculate the defect of
(\ref{eq:49})
divided by $e_x$ in the case $x=y$.
Consider the natural embedding $j:U\setminus X_x\hookrightarrow U$. We have $j^*\F=\E(x,k)|_{U\setminus X_x}=0$
as $U\setminus X_x\subset X\setminus\overline{X_x}$. The distinguished triangle
$$
0=j_!j^*\F\to\F\to i_*i^*\F\stackrel{+1}\to
$$
yields $\F=i_*i^*\F=i_*\csh k{X_x}[d_x]$. For any $?\in\{!,*\}$, we get
$$
i_x^?\E(x,k)=\tilde\imath_x^?\F=\tilde\imath_x^?i_*\csh k{X_x}[d_x]=f^?i^?i_*\csh k{X_x}[d_x]=f^?\csh k{X_x}[d_x].
$$
Hence~(\ref{eq:49}) becomes the natural map
$$
H_T^{\bullet+d_x}(\{x\},f^!\csh k{X_x})\to H_T^{\bullet+d_x}(\{x\},f^*\csh k{X_x}).
$$
Under the identifications of Section~\ref{Euler_Classes}, this map becomes $\edgeright{S[-d_x]}{e_x\cup ?}{S[d_x]}$.
Division by $e_x$ leaves us with the identity map $S_k[d_x]\to S_k[d_x]$, whose defect is obviously is $v^{d_x}$.
Hence we get the following result.

\begin{theorem}\label{theorem:5} The defect of the inclusion $\X_x(k)\hookrightarrow\X^x(k)$ divided by $e_x$
is $$\sum_{d\in\Z}m(x,d)v^{d_x-d-r}.$$
\end{theorem}
Once we compute the above inclusion, for example, by Proposition~\ref{proposition:5},
we can recover the coefficients $m(x,d)$.

\subsection{Example of torsion}\label{Example_of_torsion} We use here the notation of Proposition~\ref{proposition:5}.
Let $G=\SL_8(\C)$, $\Pi=\{\alpha_1,\ldots,\alpha_7\}$ and
$$s=(s_3,s_2,s_1,s_5,s_4,s_3,s_2,s_6,s_5,s_4,s_3,s_7,s_6,s_5),\quad x=s_2s_3s_2s_5s_6s_5,$$
where $s_i=s_{\alpha_i}$. We arrange elements of $\Gamma_x$ in ascending order with respect to $<$.
%
%
The matrix $H_x=\{h_{i,j}\}_{i,j=1}^{29}$ computed by~(\ref{eq:40}) has the following nonzero entries:
{\footnotesize
\begin{multline*}
h_{1,j}=1\text{ for }1\le j\le 29, h_{13,13}=h_{13,14}=h_{13,15}=h_{13,16}=h_{13,17}=h_{13,18}=h_{13,19}=\alpha_5+\alpha_6,\\
h_{3,3}= h_{3,6}= h_{3,8}= h_{3,15}= h_{3,17}= h_{3,22}= h_{3,27}=\alpha_2+\alpha_3, h_{7,7}= h_{7,8}= h_{7,16}= h_{7,17}=\alpha_3+\alpha_4+\alpha_5,\\
h_{20,20}= h_{20,21}= h_{20,22}= h_{20,23}= h_{20,24}= h_{20,25}= h_{20,26}= h_{20,27}= h_{20,28}= h_{20,29}=\alpha_6,\\
h_{4,4}= h_{4,5}= h_{4,6}= h_{4,11}= h_{4,12}= h_{13,25}= h_{13,26}= h_{13,27}= h_{13,28}= h_{13,29}=\alpha_5,h_{27,27}=\alpha_6\alpha_5(\alpha_2+\alpha_3)\\
h_{9,9}= h_{9,10}= h_{9,11}= h_{9,12}= h_{9,18}= h_{9,19}= h_{9,23}= h_{9,24}= h_{9,28}= h_{9,29}=\alpha_2,h_{5,14}= h_{6,19}=-\alpha_3\alpha_6,\\
h_{2,2}= h_{2,5}= h_{2,14}= h_{2,21}= h_{2,26}= h_{3,10}= h_{3,12}= h_{3,19}= h_{3,24}= h_{3,29}=\alpha_3,h_{26,27}=-\alpha_6\alpha_2\alpha_5,\\
h_{4,16}= h_{4,17}=-\alpha_3-\alpha_4-\alpha_5-\alpha_6, h_{2,7}=h_{2,16}=-\alpha_4-\alpha_5,
h_{2,8}= h_{2,17}=-\alpha_2-\alpha_3-\alpha_4-\alpha_5,\\
h_{12,12}=h_{19,29}=\alpha_2\alpha_3\alpha_5, h_{26,26}= h_{27,29}=\alpha_6\alpha_3\alpha_5,
h_{6,17}=-(\alpha_3+\alpha_4+\alpha_5+\alpha_6)(\alpha_2+\alpha_3),\\
h_{5,6}=h_{14,27}=-\alpha_2\alpha_5, h_{10,10}= h_{10,12}= h_{10,19}= h_{10,24}= h_{10,29}=\alpha_2\alpha_3,
h_{28,28}= h_{28,29}=\alpha_6\alpha_2\alpha_5,\\
h_{24,24}= h_{24,29}=\alpha_2\alpha_3\alpha_6,
h_{5,5}= h_{6,12}= h_{14,26}= h_{15,29}=\alpha_3\alpha_5,
h_{11,18}= h_{11,19}= h_{21,22}= h_{21,27}=-\alpha_2\alpha_6,\\
h_{22,22}= h_{22,27}=\alpha_6(\alpha_2+\alpha_3),
h_{14,14}= h_{15,19}=(\alpha_5+\alpha_6)\alpha_3,
h_{2,3}= h_{2,6}= h_{2,15}= h_{2,22}= h_{2,27}=-\alpha_2,\\
h_{6,15}=-\alpha_6(\alpha_2+\alpha_3),
h_{19,19}=(\alpha_5+\alpha_6)\alpha_2\alpha_3,
h_{15,15}= h_{15,17}=(\alpha_5+\alpha_6)(\alpha_2+\alpha_3),
h_{14,15}=-(\alpha_5+\alpha_6)\alpha_2,\\
h_{14,16}=-(\alpha_5+\alpha_6)(\alpha_4+\alpha_5),
h_{5,7}=(\alpha_3+\alpha_4)(\alpha_4+\alpha_5),
h_{14,17}=-(\alpha_5+\alpha_6)(\alpha_2+\alpha_3+\alpha_4+\alpha_5),\\
h_{18,18}= h_{18,19}=(\alpha_5+\alpha_6)\alpha_2,
h_{5,8}=(\alpha_3+\alpha_4)(\alpha_2+\alpha_3+\alpha_4+\alpha_5),
h_{11,11}= h_{11,12}= h_{18,28}= h_{18,29}=\alpha_2\alpha_5,\\
h_{16,16}= h_{16,17}=(\alpha_5+\alpha_6)(\alpha_3+\alpha_4+\alpha_5),
h_{12,19}=-\alpha_2\alpha_3\alpha_6,
h_{29,29}=\alpha_2\alpha_3\alpha_5\alpha_6,h_{4,7}= h_{4,8}=-\alpha_3-\alpha_4,\\
h_{4,13}= h_{4,14}= h_{4,15}= h_{4,18}= h_{4,19}=-\alpha_6,
h_{5,15}= h_{23,23}= h_{23,24}= h_{23,28}= h_{23,29}=\alpha_2\alpha_6,\\
h_{6,6}= h_{15,27}=\alpha_5(\alpha_2+\alpha_3),
h_{8,8}= h_{8,17}=(\alpha_3+\alpha_4+\alpha_5)(\alpha_2+\alpha_3),
h_{17,17}=(\alpha_5+\alpha_6)(\alpha_3+\alpha_4+\alpha_5)(\alpha_2+\alpha_3),\\
h_{6,8}=-(\alpha_3+\alpha_4)(\alpha_2+\alpha_3),
h_{5,16}=(\alpha_3+\alpha_4+\alpha_5+\alpha_6)(\alpha_4+\alpha_5),
h_{5,17}=(\alpha_3+\alpha_4+\alpha_5+\alpha_6)(\alpha_2+\alpha_3+\alpha_4+\alpha_5),\\
h_{21,21}= h_{21,26}= h_{22,24}= h_{22,29}=\alpha_3\alpha_6,
h_{25,25}= h_{25,26}= h_{25,27}= h_{25,28}= h_{25,29}=\alpha_5\alpha_6.
\end{multline*}}
The $29^{th}$ row of the matrix $(H_x^{-1})^T$ has minimal degree. Its precise value is the following:
\begin{multline*}
\textstyle
r_{29}=\Big(\frac1{\alpha_6\alpha_2\alpha_5\alpha_3}, -\frac1{(\alpha_2+\alpha_3)\alpha_6\alpha_5\alpha_3},-\frac1{\alpha_2\alpha_6\alpha_5(\alpha_2+\alpha_3)}, -\frac1{(\alpha_5+\alpha_6)\alpha_2\alpha_5\alpha_3},\frac1{(\alpha_5+\alpha_6)(\alpha_2+\alpha_3)\alpha_5\alpha_3}, \\
\textstyle
\frac1{\alpha_2(\alpha_5+\alpha_6)\alpha_5(\alpha_2+\alpha_3)},0,0,-\frac1{\alpha_6\alpha_2\alpha_5\alpha_3}, \frac1{\alpha_6\alpha_2\alpha_5\alpha_3}, \frac1{(\alpha_5+\alpha_6)\alpha_2\alpha_5\alpha_3}, -\frac1{(\alpha_5+\alpha_6)\alpha_2\alpha_5\alpha_3},-\frac1{\alpha_6(\alpha_5+\alpha_6)\alpha_2\alpha_3},\\
\textstyle
\frac1{\alpha_6(\alpha_2+\alpha_3)(\alpha_5+\alpha_6)\alpha_3},\frac1{\alpha_2\alpha_6(\alpha_5+\alpha_6)(\alpha_2+\alpha_3)},0,0,\frac1{\alpha_6(\alpha_5+\alpha_6)\alpha_2\alpha_3}, -\frac1{\alpha_6(\alpha_5+\alpha_6)\alpha_2\alpha_3}, -\frac1{\alpha_6\alpha_2\alpha_5\alpha_3},\\
\textstyle
\frac1{(\alpha_2+\alpha_3)\alpha_6\alpha_5\alpha_3}, \frac1{\alpha_2\alpha_6\alpha_5(\alpha_2+\alpha_3)}, \frac1{\alpha_6\alpha_2\alpha_5\alpha_3}, -\frac1{\alpha_6\alpha_2\alpha_5\alpha_3}, \frac1{\alpha_6\alpha_2\alpha_5\alpha_3}, -\frac1{(\alpha_2+\alpha_3)\alpha_6\alpha_5\alpha_3}, -\frac1{\alpha_2\alpha_6\alpha_5(\alpha_2+\alpha_3)},\\
\textstyle
-\frac1{\alpha_6\alpha_2\alpha_5\alpha_3}, \frac1{\alpha_6\alpha_2\alpha_5\alpha_3}\Big).
\end{multline*}
We have the following  Euler class: $e_x=\alpha_3(\alpha_2+\alpha_3)\alpha_6\alpha_2(\alpha_5+\alpha_6)\alpha_5$. Hence
the matrix $P_x/e_x$ has the following diagonal:

{\footnotesize
\begin{multline*}
p:=\big(\alpha_3\alpha_6\alpha_2\alpha_5\alpha_1\alpha_4(\alpha_3+\alpha_4+\alpha_5)\alpha_7,-\alpha_3(\alpha_2+\alpha_3)\alpha_1\alpha_5(\alpha_3+\alpha_4)\alpha_6(\alpha_3+\alpha_4+\alpha_5)\alpha_7,\\
-(\alpha_2+\alpha_3)(\alpha_1+\alpha_2+\alpha_3)\alpha_5(\alpha_3+\alpha_4)\alpha_2\alpha_6(\alpha_3+\alpha_4+\alpha_5)\alpha_7, -\alpha_3\alpha_2\alpha_1\alpha_5(\alpha_4+\alpha_5)(\alpha_5+\alpha_6)(\alpha_3+\alpha_4+\alpha_5)\alpha_7,\\
\alpha_3(\alpha_2+\alpha_3)\alpha_1\alpha_5(\alpha_3+\alpha_4+\alpha_5)^2(\alpha_5+\alpha_6)\alpha_7,(\alpha_2+\alpha_3)(\alpha_1+\alpha_2+\alpha_3)\alpha_5(\alpha_3+\alpha_4+\alpha_5)^2\alpha_2(\alpha_5+\alpha_6)\alpha_7,\\
-(\alpha_2+\alpha_3)\alpha_1(\alpha_3+\alpha_4+\alpha_5)^2(\alpha_4+\alpha_5)(\alpha_5+\alpha_6)(\alpha_3+\alpha_4)\alpha_7,\\
(\alpha_2+\alpha_3)(\alpha_1+\alpha_2+\alpha_3)(\alpha_3+\alpha_4+\alpha_5)^2(\alpha_2+\alpha_3+\alpha_4+\alpha_5)(\alpha_5+\alpha_6)(\alpha_3+\alpha_4)\alpha_7,\\
-\alpha_3\alpha_2(\alpha_1+\alpha_2)\alpha_5\alpha_4\alpha_6(\alpha_2+\alpha_3+\alpha_4+\alpha_5)\alpha_7, \alpha_3(\alpha_1+\alpha_2+\alpha_3)\alpha_5(\alpha_3+\alpha_4)\alpha_2\alpha_6(\alpha_2+\alpha_3+\alpha_4+\alpha_5)\alpha_7,
\\ \alpha_3\alpha_2(\alpha_1+\alpha_2)\alpha_5(\alpha_4+\alpha_5)(\alpha_5+\alpha_6)(\alpha_2+\alpha_3+\alpha_4+\alpha_5)\alpha_7,
\\-\alpha_3(\alpha_1+\alpha_2+\alpha_3)\alpha_5(\alpha_3+\alpha_4+\alpha_5)\alpha_2(\alpha_5+\alpha_6)(\alpha_2+\alpha_3+\alpha_4+\alpha_5)\alpha_7,\\
-\alpha_3\alpha_2\alpha_1(\alpha_4+\alpha_5)(\alpha_5+\alpha_6)\alpha_6(\alpha_3+\alpha_4+\alpha_5)(\alpha_5+\alpha_6+\alpha_7), \alpha_3(\alpha_2+\alpha_3)\alpha_1(\alpha_3+\alpha_4+\alpha_5)^2(\alpha_5+\alpha_6)\alpha_6(\alpha_5+\alpha_6+\alpha_7),\!\\
(\alpha_2+\alpha_3)(\alpha_1+\alpha_2+\alpha_3)(\alpha_3+\alpha_4+\alpha_5)^2\alpha_2(\alpha_5+\alpha_6)\alpha_6(\alpha_5+\alpha_6+\alpha_7),\\
\end{multline*}

\begin{multline*}
(\alpha_2+\alpha_3)\alpha_1(\alpha_3+\alpha_4+\alpha_5)^2(\alpha_4+\alpha_5)(\alpha_5+\alpha_6)(\alpha_3+\alpha_4+\alpha_5+\alpha_6)(\alpha_5+\alpha_6+\alpha_7),\\
 -(\alpha_2+\alpha_3)(\alpha_1+\alpha_2+\alpha_3)(\alpha_3+\alpha_4+\alpha_5)^2(\alpha_2+\alpha_3+\alpha_4+\alpha_5)(\alpha_5+\alpha_6)(\alpha_3+\alpha_4+\alpha_5+\alpha_6)(\alpha_5+\alpha_6+\alpha_7),\\
 \alpha_3\alpha_2(\alpha_1+\alpha_2)(\alpha_4+\alpha_5)(\alpha_5+\alpha_6)\alpha_6(\alpha_2+\alpha_3+\alpha_4+\alpha_5)(\alpha_5+\alpha_6+\alpha_7),\\
  -\alpha_3(\alpha_1+\alpha_2+\alpha_3)(\alpha_3+\alpha_4+\alpha_5)\alpha_2(\alpha_5+\alpha_6)\alpha_6(\alpha_2+\alpha_3+\alpha_4+\alpha_5)(\alpha_5+\alpha_6+\alpha_7),
\\-\alpha_3\alpha_2\alpha_1\alpha_5\alpha_4\alpha_6(\alpha_3+\alpha_4+\alpha_5+\alpha_6)(\alpha_6+\alpha_7), \alpha_3(\alpha_2+\alpha_3)\alpha_1\alpha_5(\alpha_3+\alpha_4)\alpha_6(\alpha_3+\alpha_4+\alpha_5+\alpha_6)(\alpha_6+\alpha_7),\\
(\alpha_2+\alpha_3)(\alpha_1+\alpha_2+\alpha_3)\alpha_5(\alpha_3+\alpha_4)\alpha_2\alpha_6(\alpha_3+\alpha_4+\alpha_5+\alpha_6)(\alpha_6+\alpha_7),
\\ \alpha_3\alpha_2(\alpha_1+\alpha_2)\alpha_5\alpha_4\alpha_6(\alpha_2+\alpha_3+\alpha_4+\alpha_5+\alpha_6)(\alpha_6+\alpha_7),
\\-\alpha_3(\alpha_1+\alpha_2+\alpha_3)\alpha_5(\alpha_3+\alpha_4)\alpha_2\alpha_6(\alpha_2+\alpha_3+\alpha_4+\alpha_5+\alpha_6)(\alpha_6+\alpha_7),
\\\alpha_3\alpha_2\alpha_1\alpha_5(\alpha_4+\alpha_5)\alpha_6(\alpha_3+\alpha_4+\alpha_5+\alpha_6)(\alpha_5+\alpha_6+\alpha_7), \\
-\alpha_3(\alpha_2+\alpha_3)\alpha_1\alpha_5(\alpha_3+\alpha_4+\alpha_5)\alpha_6(\alpha_3+\alpha_4+\alpha_5+\alpha_6)(\alpha_5+\alpha_6+\alpha_7),\\
 -(\alpha_2+\alpha_3)(\alpha_1+\alpha_2+\alpha_3)\alpha_5(\alpha_3+\alpha_4+\alpha_5)\alpha_2\alpha_6(\alpha_3+\alpha_4+\alpha_5+\alpha_6)(\alpha_5+\alpha_6+\alpha_7), \\
 -\alpha_3\alpha_2(\alpha_1+\alpha_2)\alpha_5(\alpha_4+\alpha_5)\alpha_6(\alpha_2+\alpha_3+\alpha_4+\alpha_5+\alpha_6)(\alpha_5+\alpha_6+\alpha_7), \\
 \alpha_3(\alpha_1+\alpha_2+\alpha_3)\alpha_5(\alpha_3+\alpha_4+\alpha_5)\alpha_2\alpha_6(\alpha_2+\alpha_3+\alpha_4+\alpha_5+\alpha_6)(\alpha_5+\alpha_6+\alpha_7)\big)
\end{multline*}}
and zeros elsewhere.

Consider the triple scalar product of rows: $(a,b,c)=\sum_{i=1}^{29}a_ib_ic_i$. A (computer) calculation
shows that $(r_{29},r_{29},p)=2$. Thus we have proved that the defect of the inclusion
$\X_x(k)\hookrightarrow\X^x(k)$ divided by $e_x(k)$ is $0$ if ${\rm char}\,k=2$ and is $v^{-8}$ otherwise.
By Theorem~\ref{theorem:5}, we get
$$
\sum_{d\in\Z}m(x,d)v^{-d-8}=\left\{\begin{array}{ll}
0&\text{ if }{\rm char}\,k=2;\\
v^{-8}&\text{ otherwise.}
\end{array}\right.
$$
Hence we get the following result.

\begin{theorem}\label{theorem:6} Let $G=\SL_8(\C)$ and $\Pi=\{\alpha_1,\ldots,\alpha_7\}$ be the set of simple roots.
Consider the Bott-Samelson variety $\Sigma$ for the sequence
$s=(s_3,s_2,s_1,s_5,s_4,s_3,s_2,s_6,s_5,s_4,s_3,s_7,s_6,s_5)$ and take $x=s_2s_3s_2s_5s_6s_5$.
Let $\Sigma\to G/B$ be the canonical resolution and $k$ be a field.
If ${\rm char}\,k=2$, then $\pi_*\csh{k}{\Sigma}[14]$ has no direct summand of the form $\E(x,k)[d]$.
If ${\rm char}\,k\ne2$, then $\E(x,k)$ is its only direct summand of this form. Moreover, it occurs with multiplicity $1$.
\end{theorem}

\end{document}